\documentclass[a4paper,12pt,english,psamsfots]{amsart}

\usepackage{amsmath,upref,amssymb}
\usepackage{pdfsync}
\usepackage{tensor}

\swapnumbers

\input xypic
\xyoption{all}
\xyoption{arc}
\xyoption{2cell}%%%%%%%
\CompileMatrices

\RequirePackage{calrsfs}
\DeclareSymbolFont{rsfscript}{OMS}{rsfs}{m}{b}
\DeclareSymbolFontAlphabet{\mathrsfs}{rsfscript}
\renewcommand{\mathcal}{\mathrsfs}

\newcommand{\nc}{\newcommand}

\nc{\on}{\operatorname}

\nc{\DMO}{\DeclareMathOperator}
\numberwithin{equation}{subsection}

\theoremstyle{plain}
\newtheorem*{thm*}{Theorem}
\newtheorem*{lem*}{Lemma}
\newtheorem*{prop*}{Proposition}
\newtheorem*{cor*}{Corollary}

\theoremstyle{definition}
\newtheorem*{defn*}{Definition} 
\newtheorem*{conj*}{Conjecture}
\newtheorem*{quest*}{Question}
\newtheorem*{exmp*}{Example}
\newtheorem*{ass*}{Assumption}

\theoremstyle{remark}
\newtheorem*{rem*}{Remarks}
\newtheorem*{note*}{Note}
\newtheorem*{case*}{Case}
\makeatletter
\renewcommand{\p@enumi}{\thesubsection}
\renewcommand{\p@enumii}{\thesubsection\theenumi}
\makeatother

\newenvironment{num}{\renewcommand{\theenumi}{(\alph{enumi})}
                      
                                          \begin{enumerate} }
                    {\end{enumerate} }
 \newenvironment{conds}{\renewcommand{\theenumi}{(\roman{enumi})}
                       
                        \begin{enumerate} }
                     {\end{enumerate} }
 \newenvironment{subconds}{
                       
                        \begin{enumerate} }
                     {\end{enumerate} }
\nc{\be}{\begin{equation}}
\nc{\ee}{\end{equation}}

\nc{\bee}{\begin{equation*}}
\nc{\eee}{\end{equation*}}

\nc{\bs}{\begin{split}}
\nc{\es}{\end{split}}

\nc{\bc}{\begin{cases}}
\nc{\ec}{\end{cases}}

\nc{\bml}{\begin{multline}}
\nc{\eml}{\end{multline}}

\nc{\bmll}{\begin{multline*}}
\nc{\emll}{\end{multline*}}

\nc{\lb}[1]{\label{#1}\mar{#1}}

\nc{\belb}[1]{\mar{#1}\begin{equation}\label{#1}}
\newcommand{\mpair}[1]{\pair{\,#1\,}}

\newcommand{\mset}[1]{\set{\,#1\,}}

\newcommand{\pair}[1]{\langle #1\rangle}

\newcommand{\set}[1]{\{#1\}}

\nc{\mc}{{\mathcal}}
\nc{\mf}{{\mathfrak}}
\nc{\CA}{{\mathcal A}}
\nc{\CB}{{\mathcal B}}
\nc{\CC}{{\mathcal C}}
\nc{\CD}{{\mathcal D}}
\nc{\CE}{{\mathcal E}}
\nc{\CF}{{\mathcal F}}
\nc{\CG}{{\mathcal G}}
\nc{\CH}{{\mathcal H}}
\nc{\CI}{{\mathcal I }}
\nc{\CJ}{{\mathcal J }}
\nc{\CK}{{\mathcal K }}
\nc{\CL}{{\mathcal L}}
\nc{\CM}{{\mathcal M}}
\nc{\CN}{{\mathcal N}}
\nc{\CO}{{\mathcal O}}
\nc{\CP}{{\mathcal P}}
\nc{\CQ}{{\mathcal Q}}
\nc{\CR}{{\mathcal R}}
\nc{\CS}{{\mathcal S}}
\nc{\CT}{{\mathcal T}}
\nc{\CU}{{\mathcal U}}
\nc{\CV}{{\mathcal V}}
\nc{\CW}{{\mathcal W}}
\nc{\CX}{{\mathcal X}}
\nc{\CY}{{\mathcal Y}}
\nc{\CZ}{{\mathcal Z}}
\nc{\bbZ}{{\mathbb Z}}
\nc{\bbC}{{\mathbb C}}
\nc{\bbR}{{\mathbb R}}
\nc{\bbP}{{\mathbb P}}
\nc{\bbF}{{\mathbb F}}
\nc{\bbN}{{\mathbb N}}
\nc{\bbD}{{\mathbb D}}
\nc{\bbO}{{\mathbb O}}
\nc{\bbQ}{{\mathbb Q}}

\def\a{\alpha}
\def\b{\beta}
\def\g{\gamma}
\def\G{\Gamma}
\def\d{\delta}
\def\D{\Delta}
\def\e{\epsilon}

\def\l{\lambda}

\def\L{\Lambda}

\def\o{\omega}

\def\s{\sigma}

\def\th{\theta}
\def\Th{\Theta}

\def\to{\rightarrow}

\def\lsupp#1#2{\kern\scriptspace\vphantom{#2}^{#1}\kern-\scriptspace#2}

\def\lsubb#1#2{\kern\scriptspace\vphantom{#2}_{#1}\kern-\scriptspace#2}

\def\lsub#1#2{\tensor*[_{#1}]{#2}{}}

\def\lrsub#1#2#3{\tensor*[_{#1}]{#2}{_{#3}}}

\newcommand{\seq}{{\,\subseteq\,}}

\newcommand{\sneq}{{\,\subsetneq\,}}

\newcommand{\sm}{{\,\setminus\,}}

\newcommand{\eset}{{\emptyset}}

\nc{\wt}{\widetilde}

\nc{\wh}{\widehat}

\nc{\ol}{\overline}

\def\dotcup{\hskip1mm\dot{\cup}\hskip1mm}

\newcommand{\op}{^{{\rm op}}}

\newcommand{\mar}{\marginpar}

\newcommand{\join}{\bigvee}

\newcommand{\meet}{\bigwedge}

\DMO{\Supp}{{\mathrm{Supp}}}

\DeclareMathOperator{\Aut}{{\mathrm{Aut}}}

\DMO{\Rad}{{\mathrm{Rad}}}

\DMO{\Ann}{{\mathrm{Ann}}}

\nc{\Id}{\mathrm{Id}}

\DMO {\Add}{{\mathrm Add}}

\DMO {\add}{{\mathrm add}}

\DMO {\Img}{{\mathrm Im}}

\DMO {\coim}{{\mathrm Coim}}

\DMO {\coker}{{\mathrm Coker}}

\DMO {\colim}{\varinjlim}

\DMO {\plim}{\varprojlim}

\DMO {\mEnd}{{\mathrm End}}

\DMO {\mend}{{\mathrm end}}

\DMO {\Proj}{{\mathrm Proj}}

\DMO {\Ext}{{\mathrm Ext}}

\DMO {\ext}{{\mathrm ext}}

\DMO {\tor}{{\mathrm tor}}

\DMO {\Tor}{{\mathrm Tor}}

\DMO {\Hom}{{\mathrm Hom}}

\DMO {\HOM}{{\mathrm HOM}}

\DMO {\Modfg}{{\mathrm -Modfg}}

\DMO {\modulfg}{{\mathrm -modfg}}

\DMO {\modul}{{\mathrm -mod}}

\DMO {\Mod}{{\mathrm -Mod}}

\DMO {\pdim}{{\mathrm proj.dim.}}

\DMO {\PD}{{{\mathrm Proj.Dim.}}}

\DMO {\gldim}{{\mathrm gr.gl.dim.\ }}

\DMO {\grad}{{\mathrm rad }}

\DMO {\Sheaves}{{\mathrm Sh}}

\DMO {\Flab}{{\mathrm Fl}}

\DMO {\Poinc}{{\mathrm Poinc}}

\DMO {\Groth}{{{{\mathrm K}}_0}}

\DMO {\Mat}{{\mathrm Mat}}

\DMO \Sl{{\mathrm sl}}

\DMO \SL{{\mathrm SL}}

\DMO \Gl{{\mathrm gl}}

\DMO \GL{{\mathrm GL}}

\DMO \lcm{\mathrm{lcm}}

\DMO \rank{{\mathrm rank}}

\DMO \diag{{\mathrm diag}}

\nc{\bib}{\bibitem}

\nc{\pa}{\partial}

\DMO {\Addp}{{\mathrm Add}'}

\DMO {\addp}{{\mathrm add}'}

\DMO {\hatGr}{{\widehat{{\mathrm K}}_0}}

\DMO {\NExt}{{\mathrm NExt}}

\DMO {\nExt}{{\mathrm next}}

\DMO {\DExt}{{\mathrm DExt}}

\DMO {\dext}{{\mathrm dext}}

\newcommand{\Nat}{{\mathbb N}}
\newcommand{\real}{{\mathbb R}}

\DMO{\ob}{ob}
\DMO{\mor}{mor}
\DMO{\tr}{tr}
\DMO{\spec}{spec}
\newcommand\olp[1]{{\overline{#1}}^{\,\prime}}
\newcommand{\ab}{\mathfrak{A}}

\newcommand{\nts}{\negthinspace}
\newcommand{\snts}{\nts}
\DMO{\cov}{cov}
\DMO{\Sym}{Sym}
\DMO{\Dih}{Dih}
\DMO{\Spec}{Spec}
\DMO{\domn}{dom}
\DMO{\cod}{cod}
\DMO{\ord}{ord}

\newcommand{\nat}{\natural}
\newcommand{\cp}{\complement}
\newcommand{\gpdpreord}{\mathbf{Gpd\text{\bf -}PreOrd}}

\newcommand{\rlec}{\mathbf{RdE}}

\newcommand{\setc}{\mathbf{Set}}

\newcommand{\datc}{\mathbf{Dtm}}
\newcommand{\grpdc}{\mathbf{Gpd}}
\newcommand{\cgrpdc}{\mathbf{CGpd}}

\newcommand{\bringc}{\mathbf{BRng}}
\newcommand{\balgc}{\mathbf{BAlg}}

\newcommand{\prootc}{\mathbf{Prd}}

\newcommand{\rootc}{\mathbf{rd}}
\newcommand{\rtd}{\mathbf{Rd}}

\newcommand{\setproot}{\setc\text{\bf -}\prootc}
\newcommand{\grpdset}{\grpdc\text{\bf -}\setc_{\pm}}

\begin{document}

\title{Groupoids,  root systems   and weak order II}

\author{Matthew Dyer}
\address{Department of Mathematics 
\\ 255 Hurley Building\\ University of Notre Dame \\
Notre Dame, Indiana 46556, U.S.A.}
\email{dyer.1@nd.edu}
%\dedicatory{}

%\thanks{}
%\keywords{}
%\subjclass{Primary: 20F55, 17B22, 20L05, 06A12, 20J99; Secondary: 52C35, 20F05, 06E99}
%\subjclass{Primary: 20F55: Secondary: 17B22, 20L05, 06A12, 20J99,  20F05}
%\date{\today}
\begin{abstract}   
   This  is the second introductory paper    concerning  structures called rootoids and protorootoids, the definition of which is abstracted from 
   formal properties of Coxeter groups with their root systems and weak 
   orders.  The ubiquity of protorootoids is  shown by  attaching them   to  
   structures such as  groupoids with generators,  to simple graphs, to subsets of Boolean rings, to possibly infinite oriented matroids, and to 
   groupoids with a specified preorder on each set of morphisms with 
   fixed codomain; in each case, the condition that the structure give rise to a rootoid defines an interesting subclass of these structures. The 
   paper also gives non-trivial  examples of morphisms of rootoids  and 
   describes (without proof, and partly informally)  some  main ideas, results and questions  from subsequent papers of the series, including 
   the basic    facts about principal rootoids and    functor rootoids which together  provide the raison d'\^etre for these papers. 
   \end{abstract}

\maketitle
%%%%%%%%%%%%%%%%%%%%%%%%%%%%%%%%
%
%                           Paper 2
%
%
%%%%%%%%%%%%%%%%%%%%%%%%%%%%%%%%
  \setcounter{section}{7}
\section*{Introduction}\label{s0}
  This  is the second introductory paper  in a series of papers concerning  structures called rootoids and protorootoids, the definition of which is abstracted from formal properties of Coxeter groups with their root systems and weak orders.
  The basic definitions appear in \cite{DyGrp1}. This paper  provides  further basic facts about  and examples of  rootoids and protorootoids,  and formulates (often incompletely, in more limited generality or informally)   many of  the main results of subsequent papers.  Assuming familiarity with (at least the introduction to) \cite{DyGrp1},    its concerns may be  detailed  as follows.

 Section \ref{sa8} discusses    a construction of  a  protorootoid from  a 
  groupoid $G$ with a specified   generating 
   set $S$ (assumed  to be closed under inversion and to not 
   contain any identity morphism). This generalizes a 
   standard construction  of the  abstract  root system of a 
   Coxeter group  which arises naturally in the theory of 
   buildings.  The pair  $(G,S)$ is called a $C_{2}$-system if  the 
   associated 
   protorootoid is a rootoid. Some basic characterizations,   examples 
   and properties 
   of   $C_{2}$-systems are mentioned  in Section \ref{sa8} without proof. In 
   particular,  their braid presentations 
 and the fact that  even  $C_{2}$-systems (defined as those with a sign character) correspond 
      to abridged principal rootoids, up to isomorphism of 
   each, are discussed.  The braid relations may be specified by  Coxeter matrices (the entries of  which  equal  lengths of joins of pairs of simple generators, and  determine lengths of the corresponding braid   relation) together with    local  data which determine the successive simple generators on each side of the relation.
    For example, a  cyclic group $G=\mpair{x}$ considered with generating set 
   $S=\set{x,x^{-1}}$    gives a $C_{2}$-system $(G,S)$, which is 
   even if and only if  $G$ is of even order or of infinite order.  
   
    Section \ref{s9} describes several  other  constructions of 
  protorootoids from structures of various types.
 It attaches  protorootoids to 
  simple graphs,  to certain edge-labelled simple graphs which are called 
  rainbow graphs, to subsets of Boolean rings, and to
   possibly infinite oriented matroids.  In each case, as also in Section \ref{sa8}, the 
  condition that the resulting protorootoid is a rootoid (of some specified 
  type) distinguishes an interesting class of these structures. 
Notions of covering rootoids and covering quotient rootoids mentioned
 in  \cite[2.13 and 4.9]{DyGrp1}   enable (connected) rootoids in various  natural subclasses 
  to be described in subsequent papers  as covering   quotients of (connected, simply connected) 
  rootoids arising from the above 
  structures,  associated to a  suitable  group of automorphisms of the 
  structure.  For convenience of reference, a  table showing standard notations introduced in this paper and \cite{DyGrp1} for protorootoids, set protorootoids and signed groupoid sets attached to various structures  is provided   in 9.10.

Section \ref{sa10} describes  results and questions related to 
our original motivations for this work.
 It first    describes  a protorootoid  $\CR$ attached to the set    of initial  
  sections of reflection orders of  a Coxeter system.
 Conjectures from \cite{DyWeak} imply that $\CR$ is a  complete, saturated, regular  rootoid and we 
 conjecture further that it is pseudoprincipal (see \cite[Section 3]{DyGrp1} for the definitions).   Section \ref{sa10} also introduces notions of 
 completions of rootoids suggested by this conjecture and asks to  what 
 extent  they exist in general. The only non-trivial result of the section is  
 the fact (which would be  a  weak consequence of existence of a  
 completion)  that the weak orders of a rootoid are isomorphic to  order 
 ideals of complete ortholattices.   
 
  Section \ref{sa11} constructs a left adjoint $\mathfrak{Q}$, mentioned 
  already in  \cite[5.8]{DyGrp1},  to the forgetful functor $\mathfrak{P}$ 
  from the category of protorootoids to the category of groupoid-preorders. It also shows that the full subcategory of protorootoidal 
  groupoid-preorders (those underlying a protorootoid) is equivalent to a 
  full reflective subcategory  of the category  of protorootoids, a result which enables the development of an analogue of the theory of protorootoids and rootoids in terms of groupoid-preorders (though this is not done in this paper).

  The  notion of a square of a 
 protorootoid $\CR$ plays a crucial role in our main results, and   is 
 introduced in Section \ref{s12}.  
   A commutative square of a protorootoid $\CR=(G,\L,N)$ 
 is defined  to be  a commutative diagram in $G$ involving morphisms 
  $x,y,z,w$ with $xy=zw$ such that  $x(N(y))=N(z)$ and  
  $N(x) N(z)=0$. Equivalently, the requirement is  that the 
  morphisms at each corner, suitably oriented, are orthogonal 
  i.e the pairs $(x,z)$, $(y,x^{-1})$, $(w^{-1},y^{-1})$ and 
  $(z^{-1},w)$ are each orthogonal. In terms of signed groupoid-sets, 
  the condition is that $x(\Phi_{y})=\Phi_{z}$.  A 
   trivial  point which facilitates  
  certain local to  global  comparisons  is the following rigidity property of   squares: in a faithful protorootoid, a 
  square is determined by any two of its 
  adjacent sides.

Section \ref{s13}  gives examples of morphism of rootoids, which, 
  when  specialized to 
Coxeter groups, appear implicitly   in  \cite{Muhl} or  \cite{BrHowNorm}. 
   The general proofs are deferred, but are given here  in a particularly simple and instructive case which provides a prototype 
   for   proofs in subsequent papers. Namely, it is a 
well-known theorem of Tits that  if $G$ is a group of diagram 
automorphisms of a Coxeter system $(W,S)$, and   $W^{G}$ is the 
group   of $G$-invariants  of $W$, then  $W^{G}$ is a Coxeter group. 
 In  \ref{ss12.2}, it is  proved  that the  inclusion $\th\colon W^{G}\to W$
satisfies the AOP and shown how  AOP   implies
 both  that the simple 
 generators $R$ of $W^{G}$ are the  elements of 
 the form $\th^{\perp}(s)$ with $s\in S\cap \domn(\th^{\perp})$, and that $W^{G}$ is preprincipal. 
 Tits' theorem  is easily recovered from this.  The proof of AOP for
  $\th$ requires an 
 expression of $\th^{\perp}(w)$ (which, a priori, is defined as a 
 meet) as a join; in the special case of $\th^{\perp}(s)$, for $s$ as above, this expression  amounts to  the explicit description by Tits of the simple generators $R$.

 The results in \cite{Muhl} generalize the above one  of Tits to Coxeter 
 subgroups of $W$ generated by longest elements of finite parabolic 
 subgroups of $W$ and satisfying a length compatibility condition
 which, in more general contexts, is  the main requirement in the definition of preprincipal rootoid.
   In \ref{ss13.3},  a generalization of these results   to preprincipal rootoids  is stated   in
    terms of   order-theoretic hypotheses; it amounts roughly to a  description in semilocal terms  of the possible  images of atomic generators of preprincipal rootoids under rootoid local embeddings into a fixed preprincipal rootoid. The result is  illustrated by examples with Coxeter groups  in \ref{ss13.4}.

 Subsections  \ref{ss13.5}--\ref{ss13.7} discuss what are here called normalizer groupoids, which are obtained by a  generalization of a construction in \cite{BrHowNorm}.   To illustrate the definition,  let $(W,S)$  be a Coxeter system with root system $\Phi$ and simple roots $\Pi$. The  normalizer groupoid $\CN$     is a   groupoid,  the   objects of which  identify with the subsets of  $ \mset{\Phi_{w}\mid w\in W}$, two such objects being in the same component if and only if they are in the same $W$-orbit for the natural action of $W$ on sets of subsets of the root system,   and for which the vertex group at such a subset identifies with its setwise stabilizer (normalizer) in $W$.    The  groupoid $\CN$  has a subgroupoid (a union of components of $\CN$)  
with objects the subsets of    $\mset{\Phi_{w}\mid w\in S}\cong \Pi$, the components of which  were studied in \cite{BrHowNorm}.
   The  normalizer groupoid construction   underlies a construction with  protorootoids which preserves, amongst others,  the classes of  rootoids, interval finite rootoids, complete rootoids and preprincipal rootoids. Restricted to preprincipal rootoids, it provides natural examples of rootoid  local embeddings as discussed in the previous paragraph.

  The results about normalizer rootoids will be deduced in later papers from  more general facts about functor rootoids  
    which are stated  without proof in Section  \ref{s14}. 
The notion of  functor rootoids was discussed in the introduction to \cite{DyGrp1} in relation to signed groupoid-sets.  Here we   provide further details, in terms of rootoids.   For   a 
   rootoid $\CR=(G,\L,N)$  and a  
    groupoid $H$,  with $G$ and  $H$ assumed non-empty and connected
     for simplicity, one has  the category $G^{H}$ of functors 
    from $H$ to $G$, with natural transformations as morphisms. 
     This category is a groupoid, and it has a subgroupoid 
     $G^{H}_{\square}$ with 
     all objects but only morphisms for which each 
     commutative square in $G$ (from the definition of  natural 
     transformation) is a commutative square of $\CR$ as  
     defined above.
     
    Fix a groupoid homomorphism $F\colon H\to G$ and let
     $G^{H}_{\square}[F]$ be the component of  $G^{H}_{\square}[F]$
    containing $F$.
     For any $b\in \ob(G)$, one has  trivially
     an ``evaluation at $b$'' groupoid morphism
     $\e_{b}\colon G^{H}_{\square}[F]\to G$, a pullback  protorootoid 
     $\e_{b}^{\nat}(\CR)$ with underlying groupoid $G^{H}_
     {\square}[F]$ and a natural protorootoid morphism  $\e_{b}^{\flat }
     \colon\e_{b}^{\nat}(\CR)\to \CR$.
      The   main result concerning functor rootoids  is  that
     $\e_{b}^{\flat}$ is a morphism  in the category of rootoid local embeddings
     (in particular,  the  protorootoid
     $\e_{b}^{\nat}(\CR)$ is a rootoid, called a functor rootoid). 
          It follows  that     the abridged functor rootoid 
$(\e_{b}^{\nat}(\CR))^{a}$ is  
principal  or complete if $\CR$ is so; it turns out to be independent of choice of $b$ up
 to canonical isomorphism.    The morphism $\e_{b}$ above is  a  called 
      a stable local embedding (of groupoids).

     Consider the  category of based, connected  groupoids  with 
     objects the pairs $(G,a)$ of a connected groupoid $G$ and 
     an object $a$ of $G$, with morphisms basepoint preserving 
     groupoid homomorphisms. If we set $F(b)=a$, then the  
     above constructs from    a morphism
      $F\colon (H,b)\to (G,a)$ a   morphism
       $\e_{b}\colon (G^{H}_{\square}[F],F)\to (G,a)$ which we call the 
       dual of $F$.
    It will be shown that  $F$ factors  canonically  through its double dual,  and 
    that the construction restricts  to a 
    duality on a suitable category of based stable local 
    embeddings over $(G,a)$.
     The proofs of  these   results  involve local descriptions of the construction in terms of  
     certain
     symmetric Galois connections associated to the objects 
     of $G$.

The fact that   $(\e_{b}^{\nat}(\CR))^{a}$ is principal
if $\CR$ is principal   is  the most substantial result of these papers. 
 It is comparatively straightforward   to show that
if  $\CR$ is a rootoid, then $\e_{b}^{\nat}(\CR) $ is a rootoid, and then trivial  that if   $\CR$ is also  complete or interval finite, then so is  $\e_{b}^{\nat}(\CR)$. The crucial step is then to show that if $\CR$ is preprincipal, then so is 
$\e_{b}^{\nat}(\CR)$.  This follows, by  arguments similar to those appearing in regard to Tits' theorem (see \ref{ss12.2}),  from the fact that
(whether or not $\CR$ is preprincipal) $\e_{b}^{\flat}$ is a rootoid local embedding. In turn, the key step in the proof that   $\e_{b}^{\flat}$ is a rootoid local embedding  is the proof that it satisfies the AOP. This also follows the general lines  of the argument for the corresponding fact in \ref{ss12.2}; the main point  is  expressing certain meets explicitly as joins.
That is accomplished   by a (finite for preprincipal rootoids, transfinite in general)   repetition of certain ``zig-zag'' and ``loop'' constructions (defined in later papers in terms of basic   combinatorics of squares) in the setting of the previously mentioned Galois connections.  Similar but simpler arguments are involved in the proofs of other results showing that principal rootoids  are preserved by natural constructions. 

 To  oversimplify, the category of  rootoids is  defined so as to be able  to express in natural terms the proof that   $(\e_{b}^{\nat}(\CR))^{a}$ is principal
if $\CR$ is principal.   It  then turns out    that the category of rootoids has all small     limits, as does its full subcategory of complete rootoids, while the  full subcategories  of preprincipal and interval finite  rootoids     admit all limits from  small
     categories with finitely many objects. These results may be viewed as  general analogues of the above-mentioned theorem of Tits.

As an application of functor rootoids, Section \ref{s14} informally 
describes (the  part  involving underlying groupoids) of  a 
  construction which attaches  to any rootoid $\CR$ a
 structure $\CR^{\o}$, which we call
     a  symmetric cubical  $\o$-rootoid.
 It  has an underlying 
     cubical $\o$-groupoid in the sense of higher category 
     theory,  each of the $n$ groupoid $n$-compositions of 
     which  underlies a rootoid (reducing to $\CR$ for $n=1$). 
     The supremum of the set of $n$ for which $\CR^{\o}$ has non-trivial $n$-morphisms provides  a measure of the degree of non-triviality of the theory of functor rootoids for $\CR$.  For the rootoid $\CC_{(W,S)}$
     attached to a Coxeter system, it is the supremum of the ranks of finite standard parabolic subgroups.  
     
     It may be expected that study of functor rootoids and related structures will be of particular interest in  examples such as Coxeter 
     groups (and more particularly,  finite and affine Weyl groups and the 
     symmetric groups) and that significantly stronger properties than are given by the existing general theory may hold in such cases.  For 
     instance,   underlying groupoids of      functor rootoids from finite rank Coxeter systems $(W,S)$ inherit from $(W,S)$  not only root systems realizable in real vector spaces but also analogues of the 
     ``canonical automaton'' (\cite{BjBr}) from $(W,S)$ (it is   a natural question if the appropriate analogue of the finiteness of  the set of 
     elementary roots (\cite{BjBr}) of $(W,S)$ holds in them).
            \section{$C_{2}$-systems} 
 \label{sa8}
\subsection{Protorootoid of a groupoid with generating set}
\label{ssa8.1}
 Considerable use is made in this section of the results of  \cite[\S 5]{DyGrp1} and especially of  the functors    $\mathfrak{I}$, $\mathfrak{J}$,   $\mathfrak{K}$, and $\mathfrak{L}$ defined there (see \cite[5.8]{DyGrp1}). 

\begin{defn*}\begin{num}
\item
A $C_{0}$-system is  a pair $(G,S)$ where $G$ is a 
groupoid and    $S$ is  a set of generators of   $G$ such 
that $S=S^{*}$ and $S$ contains no identity 
morphism of $G$. 
\item   A $C_{0}$-system  $(G,S)$ is \emph{even} if  the pair $(G,S)$ admits a sign character in the sense of   \cite[3.8]{DyGrp1}. \end{num}
\end{defn*}

 \subsection{}\label{ssa8.2} Fix  a $C_{0}$-system $(G,S)$ and  let $l=l_{S}$ be the corresponding length function on $G$. 
Let $X$ be the left regular $G$-set; this is the $G$-set (functor $G\to \setc$) with 
$\lsub{a}{X}:=\lsub{a}{G}$ and action map 
$\lrsub{a}G{b}\times \lsub{b}X\to \lsub{a}{X}$ identified with the 
multiplication map 
$\lrsub{a}G{b}\times \lsub{b}G\to \lsub{a}{G}$ of $G$, for all $a,b\in \ob(G)$.  Let $Y\colon G\to \setc$ be the composite of  $\wp_{G}(X)\colon G\to \bringc$ with the forgetful functor $\bringc\to \setc$, so $\lsub{a}{Y}=\wp(\lsub{a}{X})$  is the {set} of subsets of $\lsub{a}{X}$.
 For any $a\in \ob(G)$ and  $s\in \lsub{a}{S}:=S\cap \lsub{a}{G}$, define 
 \begin{equation}\label{eqa8.2.1} G_{s}^{>}:=\mset{g\in \lsub{a}{G}\mid l_{S}(s^{*}g)>l_{S}(g)}\in\, \lsub{a}{Y}.\end{equation}
 Observe that \begin{equation}\label{eqa8.2.2} \lrsub{}{1}{a}\in G_{s}^{>}, \qquad s\not\in G_{s}^{>}\end{equation} since by assumption,
 $S$ contains no identity morphism.  Hence 
 \begin{equation}\label{eqa8.2.3} \eset \sneq G_{s}^{>}\sneq \lsub{a}{X}.\end{equation}
 
 Define  $\Psi_{G}=\Psi\colon G\to \setc$ to be the $G$-subrepresentation of $Y$
  generated by all the elements $G_{s}^{>}\in \lsub{a}{Y} $ for 
  $s\in \lsub{a}{S}$ and $a\in \ob(G)$. Thus,  for $a\in \ob(G)$,
 \begin{equation} \label{eqa8.2.4}\lsub{a}{\Psi}:=\mset{g( G_{s}^{>})\mid g\in \lrsub{a}{G}{b}, s\in \lsub{b}{S}, b\in \ob(G)},\quad g(G_{s}^{>}):=\mset{gx\mid x\in G_{s}^{>}}\end{equation} 
 and  for $g\in \lrsub{a}{G}{b}$ and $Z\in \lsub{b}{Y}$,  one has  $\Psi(g)(Z)=gZ:=\mset{gz\mid z\in Z}\in \lsub{a}{Y}$.
  For $a\in \ob(G)$, define   \begin{equation}\label{eqa8.2.5} \lsub{a}{\Psi}'_{G}=
   \lsub{a}{\Psi}':=\mset{A\in \lsub{a}{\Psi}\mid \lrsub{}{1}{a}\in A}\seq\lsub{a}{\Psi} .\end{equation}
 Viewing $\Psi$ as a functor $\Psi\colon G\to \setc$ with $\Psi(a)=\lsub{a}{\Psi}$, there is a  $G$-cocycle (in fact, a 
 coboundary) $N$ for $\wp _{G}(\Psi)$ defined  by $\lsub{a}{N}
 ({g}):=\lsub{a}{\Psi}'+ g(\lsub{b}{\Psi}'  )\in \wp(\lsub{a}{\Psi})$ for $a,b\in \ob
 (G)$ and $g\in \lrsub{a}{G}{b}$. 
  Note that  for $s\in \lsub{a}{S}$,
  \begin{equation}\label{eqa8.2.6} \set{G_{s}^{>},s(G_{s^{*}}^{>})}\seq \lsub{a}N(s), \qquad  G_{s}^{>}\in \lsub{a}{\Psi}', \qquad s(G_{s^{*}}^{>})\not\in \lsub{a}{\Psi}'\end{equation} by \eqref{eqa8.2.2}. Hence 
 \begin{equation}\label{eqa8.2.7}\vert {N}(s)\vert \geq 2, \quad s\in S.\end{equation}
 Simple direct calculation shows that for $a\in \ob(G)$,  $s\in \lsub{a}{S}$ and $g\in \lsub{a}{G}$,
 \begin{equation}\label{eqa8.2.8}
 G_{s}^{>}\in N(g)\iff l_{S}(s^{*}g)\leq l_{S}(g),\quad
 s(G_{s^{*}}^{>})\in N(g)\iff l_{S}(s^{*}g)< l_{S}(g). \end{equation}
 Hence 
 \begin{equation}\label{eqa8.2.9}
  s(G_{s^{*}}^{>})\in N(g)\implies G_{s}^{>}\in N(g). \end{equation} 
 \begin{defn*} Let $(G,S)$ be a $C_{0}$-system and notation be as above.
 \begin{num}\item  The   set protorootoid and protorootoid  corresponding to $(G,S)$ are  $\CJ'_
 {(G,S)}:=(G,\Psi,N)$ and  $\CJ_
 {(G,S)}:=\mathfrak{I}(\CJ'_
 {(G,S)})=(G,\wp_{G}(\Psi),N)$ respectively.  \item The signed groupoid-set of
 $(G,S)$ is $J_{(G,S)}:=\mathfrak{K}(\CJ'_{(G,S)})$.
\item  The $C_{0}$-system $(G,S)$ is called a
  \emph{$C_{1}$-system} or is said  to satisfy the \emph{weak exchange condition (WEC)} if    $\CJ_{(G,S)}$ is cocycle finite and  for all  $a\in \ob(G)$ and  $g\in \lsub{a}{G}$, one has
  $\rank(N(g))=2l_{S}(g)$ in $\wp(\lsub{a}{\Psi})$ (i.e. $\vert N({g})\vert =2l_{S}(g)$). 
  \item The $C_{0}$-system $(G,S)$ is called a
  \emph{$C_{2}$-system} if  it is a $C_{1}$-system and  $\CJ_{(G,S)}$ is a rootoid.
 \end{num}
 \end{defn*}
\subsection{}\label{ssa8.3}  Let $(G,S)$ be a $C_{0}$-system.
Write  $R:=J_{(G,S)}=(G,\L)$, $\CR':=\CJ'_{(G,S)}=(G,\Psi,N)$,
$\CR:=\CJ_{(G,S)}=(G,\wp_{G}(\Psi),N)$ as above.
 From the proof of   \cite[Proposition 5.5]{DyGrp1},   
$\lsub{a}\L:=\lsub{a}\Psi\times\set{\pm }$  with
  $\lrsub{a}\L{+}=\lsub{a}\Psi\times\set{+}$. 
  For $g\in \lrsub{a}{G}{b}$, $\L_{g}:=\lsub{a}{\L}_{+} \cap g(\lsub{a}{\L}_{-})$. From  \cite[5.7]{DyGrp1}, $\vert \L(g)\vert =\vert N(g)\vert$ for any $g\in \mor(G)$. 
  \begin{lem*}\begin{num}\item If $g\in\lrsub{b} {G}{a}$ and $g(G_{s}^{>})\in \lsub{b}{\Psi}'  $ for all $s\in \lsub{a}{S}$, then $g=1_{a}$.
  \item The protorootoid  $\CR=\CJ_{(G,S)}$ is  faithful. 
    \item If $(G,S)$ is a $C_{1}$-system, then for $a\in \ob(G)$ and $x,y\in \lsub{a}{G}$, one has $x \lsub{a}{\leq  }y$ if and only if  
    $l_{S}(y)=l_{S}(x)+l_{S}(x^{*}y)$. \end{num}
       \end{lem*}
      
       \begin{proof} Suppose that $g\in \lrsub{b}{G}{a}$ is as in (a). 
 Let $s\in \lsub{a}{S}$.      Then $g(G_{s}^{>})\in\lsub{b}{\Psi}'  $ i.e. 
 $1_{b}\in g(G_{s}^{>})$ or equivalently $g^{*}\in G_{s}^{>}$. This 
 means that $l_{S}(s^{*}g^{*})>l_{S}(g^{*})$ for all 
 $s\in \lsub{a}{S}$, which implies that $g=1_{a}$ since $S=S^{*}$ 
 generates $G$. This proves (a).

        To prove (b), suppose that $g\in \lrsub{b}{G}{a}$ with $N(g)=\eset$ i.e. $\lsub{b}{\Psi}'  = g(\lsub{a}{\Psi}'  )$. 
   By \eqref{eqa8.2.6}, $g(G_{s}^{>})\in \lsub{b}{\Psi}'  $  for all $s\in \lsub{a}{S}$ and by (a), $g=1_{a}$.

  Part (c) is immediate  from the definition of $C_{1}$-system and  \cite[Corollary 3.14(a)]{DyGrp1}. 
        \end{proof}

\subsection{} \label{ssa8.4}  This subsection provides, in the special case of even $C_{0}$-systems, simpler variants of the above conditions and constructions.  Retain the assumptions and notation of \ref{ssa8.3}, but assume in addition that $(G,S)$ is even. 

 Note  that for any $s\in \lrsub{a}{S}{b}$, $g\in \lsub{a}{G}$ one has
 $l_{S}(s^{*}g)=l_{S}(g)\pm 1\neq l_{S}(g)$ and therefore 
\begin{equation}\label{eqa8.4.1}
H_{s}^{<}:=\mset{g\in \lsub{a}{G}\mid l_{S}(s^{*}g)<l_{S}(g)}=\lsub{a}{G}\sm H_{s}^{>}. \end{equation}  
Also, 
\begin{equation}\label{eqa8.4.2}
H_{s}^{<}=
s(\mset{g\in \lsub{b}{G}\mid l_{S}(sg)>l_{S}(g)})=
s(H_{s^{*}}^{>})\in \lsub{a}{\Psi}. \end{equation}  
Now for any $h\in \lrsub{a}{G}{c}$ and $r\in \lsub{c}{G}$,
\begin{equation}\label{eqa8.4.3}
\lsub{a}{G}\sm h(H_{r}^{>})=h(H_{r}^{<})\in \lsub{a}{\Psi}.
\end{equation}

The above implies that for  each $a\in \ob(G)$, there is  a definitely signed  
 set $\lsub{a}{\wh\Psi}$ with underlying set $\lsub{a}{\Psi}$,
 such that   $-\a:=\lsub{a}{G}\sm \a$ for all $\a\in \lsub{a}{\Psi}$ and  the subset of positive elements of $\lsub{a}{\wh\Psi}$ is
 $\lrsub{a}{\wh\Psi}{+}:=\lsub{a}{\Psi}'  $. Hence there is a unique  functor   $\wh\Psi\colon  G\to \setc_{\pm}$
with the following properties:
 \begin{conds}
 \item The composite functor  $  G\xrightarrow{\wh \Psi} \setc_{\pm}\to \setc$ of $\wh \Psi$ with the forgetful functor 
 $ \setc_{\pm}\to \setc$ is equal to $\Psi$.
 \item For $a\in \ob(G)$,  $\wh \Psi(a)=\lsub{a}{\wh \Psi}$, 
 as defined above, in $\setc_{\pm}$.
 \end{conds}
 This defines a signed groupoid-set $\ol{R}=
 (G,\wh \Psi)$. Let
 $\olp{\CR}=(G,\ol{\Psi},\ol{N}):=
 \mathfrak{L}(\ol{R})$ be the associated set  protorootoid. Let   $\ol{\CR}:=\mathfrak{I}(\olp{\CR})=(G,\wp_{G}(\ol{\Psi}),\ol{N})$ be the protorootoid corresponding to the set protorootoid 
 $\olp{\CR}$. To indicate dependence on $(G,S)$, write  $\ol{R}=\ol{J}_{(G,S)}$,
 $\olp{\CR}=\olp{\CJ}_{\snts(G,S)}$ and 
 $\ol{\CR}=\ol{\CJ}_{\snts(G,S)}$.  Recall the definitions of  $R=(G,\L)$,
 $\CR'=(G,\Psi,N)$ and $\CR=(G,\wp_{G}(\Psi),N)$ from \ref{ssa8.3}.
 
 For $a\in \ob(G)$, one has by definition that
  $\lsub{a}{\ol{\Psi}}=\lsub{a}{\wh\Psi}/\set{\pm }$, the set of $\set{\pm}$-orbits on   $\lsub{a}{\wh{\Psi}}$. 
  There is a natural transformation $\pi\colon \Psi\to \ol{\Psi}$ of functors $G\to \setc$ such that its component  $\pi_{a}$ at $a\in \ob(G)$ is the orbit map  $\a\mapsto \set{\pm \a} \colon \lsub{a}{\Psi}\to \lsub{a}{\ol{\Psi}}$. 
 \begin{prop*} Let $(G,S)$ be an even $C_{0}$-system and let notation be as above. Then
 \begin{num}\item  
 $\mathfrak{K}(\,\olp{\CR})\cong \ol{R}$  in $\grpdset$.
\item There is a morphism $f:=(\Id_{G},\pi)\colon \olp{\CR}\to {\CR}'$ in $\setproot$.
\item Write  $\mathfrak{I}(f)=(\Id_{G},\wp_{G}(\pi))\colon \ol{\CR}\to\CR$ for  the morphism of protorootoids corresponding to $f$. Then  the abridgement 
$\ab(\mathfrak{I}(f))\colon\ab(\, \ol{\CR}\,)\to \ab(\CR)$ is an isomorphism in $\prootc^{a}$.
\item $(G,S)$ satisfies the WEC if and only if $\ol{\CR}$ is a principal protorootoid. In that case, $S$ is the set of  simple generators of $\ol{\CR}$.
  \item The underlying set representation of  $\L$ is 
  equivalent in $\setc^{G}$ to   the coproduct (disjoint union)  $\Psi\coprod\Psi$ of two copies of $ \Psi$. 
   \end{num}\end{prop*}
  \begin{proof} Part  (a) follows immediately  from the definitions and \cite[Proposition 5.5]{DyGrp1}. 
By the definitions, for $g\in \lsub{a}{G}$,  one has
$\wh{N}_{g}=N_{g}\cap \lsub{a}{\Psi}_{+}$,  $N_{g}=\mset{ \a,-\a\mid \a\in \wh\Psi_{g}}$ and $ \ol{N}_{g}=\mset{\set{\pm \a}\mid \a\in \wh\Psi_{g}}$. Hence  $N_{g}=\pi_{a}^{-1}(\ol{N}_{g})$, which proves (b).
Also,  $\set{\pm }$ acts freely on $N(g)$ with $\wh{N}_{g}$ as set of orbit representatives and $\ol{N}(g)$ as its set of orbits. Hence $\vert N(g)\vert =2\vert \ol{N}(g)\vert$, which implies (d) by the definitions.
 Note next that $\pi_{a}$ is surjective, so $(\wp_{G}(\pi))_{a}=\wp(\pi_{a})$ is an injective  homomorphism of Boolean rings $\wp(\lsub{a}{\ol{\Psi}})\to \wp(\lsub{a}{\Psi})$, which we regard for the proof of (c) as an inclusion. With this identification,  $\ol{N}_{g}=N_{g}$ and (c) follows immediately by definition of abridgement. 
  
Now for (e).  For any $a\in \ob(G)$ and  $\a\in \lsub{a}{\Psi}$, define  the \emph{sign} of $\a$  to be \begin{equation*}
\e(\a):=\begin{cases} +,&\text{if $\a\in \lsub{a}{\Psi}'$}\\
-,&\text{otherwise}.
\end{cases}
\end{equation*}
By definition,
$\lsub{a}{\L}=\lsub{a}{\Psi}\times \set{\pm}$. Using \cite[(5.5.1)]{DyGrp1},  the  $G$-action is  given by 
\begin{equation*} \L(g)(\a,\eta)=\bigl(\Psi(g)(\a),\eta \e(\a)\e(\Psi(g)(\a))\bigr), \qquad \eta\in \set{\pm},g\in G_{a}, \a\in \lsub{a}{\Psi}, a\in \ob(G).
\end{equation*}
The underlying set representation of $\L$ is the   disjoint union of  two subrepresentations  specified (imprecisely) by
\begin{equation*} \mset{(\a,\eta)\mid \eta\e(\a)=+},\qquad \mset{(\a,\eta)\mid \eta\e(\a)=-},\end{equation*} each of which is  isomorphic by projection $(\a,\eta)\mapsto \a$ to  $\Psi$ (and which are interchanged by the action of $-\in \set{\pm}$). From this, it is straightforward to   prove (e). 
 \end{proof} 
  \begin{rem*} Although $\L$ is not the coproduct in $(\setc_{\pm})^{G}$
   of two copies of $\wh \Psi$,  the proof of   (e) can easily be elaborated  to an explicit description of   $(G,\L)$ in terms of   $(G,\wh \Psi)$. \end{rem*}
 
\subsection{} \label{ssa8.5} The constructions discussed in \ref{ssa8.1}--\ref{ssa8.4} are modelled on an interpretation, in the theory of buildings,  of the abstract root system of a Coxeter system $(W,S)$ in terms of a $W$-action on a set of  subsets (called half-spaces) of $W$.   Though the following is a corollary  of both Theorems \ref{ssa8.8} and \ref{ssa8.14}, which are proved in subsequent papers,  it is  instructive to include here a direct proof   which makes the connection between the constructions of this section and  half-spaces  explicit.  
\begin{prop*} A  Coxeter system is an even $C_{2}$-system.  
\end{prop*}
\begin{proof}  Suppose that  
 $(G,S)=(W,S)$ is  a Coxeter system.  
 Then $(W,S)$ is certainly an even $C_{0}$-system, so 
 $\ol{J}_{(W,S)}=(W,\wh \Psi)$ is defined.  Write the standard signed groupoid-set of $(W,S)$ (see  \cite[6.5, 6.7]{DyGrp1}) as
 $C_{(W,S)} =(W,\Phi)$.
 Regard $\wh\Psi=\mset{w(H_{s}^{>})\mid s\in S,w\in W}$ and 
 $\Phi=T\times\set{\pm }$  as  $W\times\set{\pm}$-sets with distinguished positive subsets of $\set{\pm}$-orbit representatives
 $\mset{A\in \wh \Psi\mid 1_{W}\in A}$ and $T\times\set{+}$ respectively
 (see   \cite[6.5]{DyGrp1}).  We claim that    there is  a unique  morphism of signed $W\times\set{\pm}$-sets $\Phi\to \wh\Psi$ 
  such that
$(s,1)\mapsto G_{s}^{>}$ for all $s\in S$, and that it induces an isomorphism 
 \begin{equation} \label{eqa8.5.1}\Phi\xrightarrow{\cong} \wh \Psi\end{equation} between these $W\times \set{\pm }$-sets preserving their positive subsets. Though this is   well known (see \cite[Ch IV, \S1, Exercise 16(i)]{Bour}),
 an alternative direct proof is sketched below.
For $t\in T$,  define \emph{half-spaces}
\begin{equation*}\begin{cases} W_{(t,+)}:=\mset{w\in W\mid l(tw)>l(w)}\\W_{(t,-)}:=
\mset{w\in W\mid l(tw)<l(w)}=W\sm W_{(t,+)}\end{cases}\end{equation*}
Since $T\times \set{\pm}$ (resp., $\wh \Psi$) is the union of $W$-orbits  of the elements $(s,+)$ (resp., $H_{s}^{>}$) for $s\in S$,  the claim   follows from the claims  (a)--(c) below:
\begin{num}\item  for $w\in W$,
$ w(G_{s}^{>})=\mset{wx\mid x\in W, l(sx)>l(x)}=W_{w(s,+)}$.
\item for $(t,\e)$  in $T\times \set{\pm }$, $1\in W_{(t,\e)}$ if and only if $\e=+$.
\item for $(t,\e)$ and $(t',\e')$ in $T\times \set{\pm }$, 
$W_{(t,\e)}=W_{(t',\e')}$ implies $(t,\e)=(t',\e')$.
\end{num}

The claim (a)  can be proved by a simple  calculation involving the fact
 that for $w,x\in W$ and $s\in S$, $l(wsx)> l(wx)$ if and only if 
  $l(ws)-l(w)=l(sx)-l(x)$; the fact itself follows follows from the cocycle property of the reflection cocycle of $(W,S)$ and \cite[6.7.1]{DyGrp1}.
  The claim (b) follows immediately from the definitions.  To prove (c),
  suppose that  $W_{(t,\e)}=W_{(t',\e')}$. By (b), $\e=\e'$. Taking complements in $W$ shows that  $W_{(t,-\e)}=W_{(t',-\e')}$.
  Hence for any $w\in W$, $l(tw)>l(w)$ if and only if  $l(t'w)>l(w)$.
  Choose (see \cite{DyRef}) a palindromic reduced expression $t=s_{1}s_{2}\cdots s_{n+1}\cdots s_{2}s_{1}$ for $t$,  with $s_{i}\in S$ and $2n+1=l(w)$.
Let $w_{i} =s_{1}\cdots s_{i}$ for $i=0,1,\ldots, n+1$, so $l(w_{i})=i$
and $tw_{n}=w_{n+1}$. Note that $l(tw_{n})>l(w_{n})$, $l(w_{n}s_{n+1})>l(w_{n})$ and  $l(tw_{n}s_{n+1})<l(w_{n}s_{n+1})$.  The same is therefore true  with $t$ replaced by $t'$. The strong exchange condition 
(see \cite[6.3]{DyGrp1}) implies that $t'w_{n}=w_{n}s_{n+1}=w_{n+1}$  and so $t'=w_{n+1}w_{n}^{-1}=t$.

From the claim, it follows that $\ol{R}=(W,\wh \Psi)\cong (W,\Phi)$ as 
signed groupoid-set. By \cite[Theorem 6.7(b)]{DyGrp1},
$\ol{R}$ is  therefore principal and  rootoidal. So is $\ol{\CR}$ by the  terminological conventions in \cite[5.6]{DyGrp1}.   
 Proposition \ref{ssa8.4}(d)  implies that $(W,S)$ satisfies WEC; that is, it is a $C_{1}$-system. 
Also,  by 
 Proposition \ref{ssa8.4}(c)
   and \cite[Remark 4.3(2)]{DyGrp1},
$\CR$ is a rootoid. Hence $(W,S)$  is a $C_{2}$-system as required.   
\end{proof}

\subsection{}\label{ssa8.6} The following subsections \ref{ssa8.7}--\ref{ssa8.16} survey, 
 without proof and partly informally,  basic properties 
and examples of $C_{2}$-systems, especially even ones.  In particular, 
 Section \ref{ssa8.16} gives  some  simple  examples.  Full statements and  proofs will be given in subsequent papers.

 \subsection{WEC}\label{ssa8.7}   The following proposition lists some of many equivalent formulations of the WEC.  Compare
(d)--(e)   with EC in  \cite[Proposition 6.3]{DyGrp1}.
 
 \begin{prop*} The following five conditions on a $C_{0}$-system $(G,S)$ are equivalent:
\begin{conds}
\item  $(G,S)$ satisfies the WEC i.e. for all $g\in \mor(G)$,
$\vert N(g)\vert=2l_{S}(g)$.
\item For all $s\in S$,  $\vert N(s)\vert =2$.
\item For all $a\in \ob(G)$ and $g\in \lsub{a}{G}$, $\vert N(g)\cap \lsub{a}{\Psi}'\vert=l_{S}(g)$.
\item If $g\in \mor(G)$ and $r,s\in S$ with $l(gr)>l(g)$ and $l(sgr)\leq l(sg)$, then
$g(H_{r}^{>})=H_{s^*}^{>}$.
\item If  $g\in \mor(G)$ and $r,s\in S$ are such that $\exists sgr$, then $ g(H_{r}^{>})=H_{s^*}^{>}$ if and only if $(l(gr)> l(g)$  and   $l(sgr)\leq l(sg))$.
\end{conds}
 \end{prop*}

  \subsection{Principal rootoids and even  $C_{2}$-systems} \label{ssa8.8}    Abridged principal rootoids correspond bijectively  to even $C_{2}$-systems up to the natural notion of isomorphism of each.
  In particular, one has the following. 
\begin{thm*}\begin{num}\item 
    If $(G,S)$ is an even $C_{2}$-system, then  $\ab(\CJ_{(G,S)})\cong \ab(\ol{\CJ}_{(G,S)})$ is a principal rootoid with simple generators $S$.
    \item  If $\CR=(G,\L,N)$ is a  principal rootoid with simple generators $S$, then $(G,S)$ is an even $C_{2}$-system, and 
 $\CR$ and $\ol{\CJ}_{(G,S)}$   have isomorphic  abridgements.\end{num} 
 \end{thm*}
The easier part  (a) already follows using Proposition \ref{ssa8.4} and  \cite[Remark 4.3(2)]{DyGrp1}. Part (b) is deduced in a subsequent paper  from a 
similar characterization of  arbitrary $C_{2}$-systems (which involves  a condition related to \eqref{eqa8.2.9}).
  \subsection{Braid presentation}\label{ssa8.9}  An even $C_{2}$-system $(G,S)$ has a canonical presentation by generators $S$ subject to \emph{trivial relations}  and  \emph{ braid relations}.
 The trivial relations are the relations $s^{-1}=s^{*}$ for $s\in S$ (recall that $S=S^{*}$). The braid relations  may be specified by (uniquely determined)
 data consisting of a  Coxeter matrix $M^{a}$ indexed by $\lsub{a}{S}$ attached to each  $a\in\ob(G)$,  and certain   maps $\set{\pi_{r}}_{r\in S}$ between subsets of the sets $\lsub{a}{S}$ for $a\in \ob(G)$, as described below.

 For $r,s\in \lsub{a}{S}$, the corresponding  entry of the Coxeter  ${M}^{a}$ is ${m}^{a}_{r,s}:=l_{S}(r\vee s)$ (if the join exists in $\lsub{a}{\leq}$; if it doesn't exist, one interprets the length as $\infty$). If $r\neq s$ and $n:=m^{a}_{r,s}\neq \infty$,
 there are exactly two    reduced expressions of $r\vee s$.  One of them is of the form $r_{1}\cdots r_{n}$ with $r_{1}=r$  and  the other is of the form
$s_{1}\cdots s_{n}$ with $s_{1}=s$. The corresponding braid relation is
\begin{equation}\label{eq8.9.1} \lrsub{a}{[r_{1},\ldots, r_{n}]}{b}=\lrsub{a}{[s_{1},\ldots, s_{n}]}{b}\end{equation}
(using the more precise notation of \cite[2.6]{DyGrp1})  where $r\vee s\in \lrsub{a}{G}{b}$.
A braid relation is defined as a relation arising in this way.
A basic  fact is that  inverses and cyclic shifts of braid relations 
are again  braid relations. For instance,  
\begin{equation}\label{eq8.9.2} \lrsub{b}{[r_{n}^{*},\ldots, r_{1}^{*}]}{a}=\lrsub{b}{[s_{n}^{*},\ldots, s_{1}^{*}]}{a}\end{equation} and 
\begin{equation}\label{eq8.9.3} \lrsub{a'}{[s_{1}^{*},r_{1},\ldots, r_{n-1}]}{b'}=\lrsub{a'}{[s_{2},\ldots, s_{n},r_{n}^{*}]}{b'}\end{equation}  are braid relations, where $a':=\cod(s_{2})$ and $b':=\domn(r_{n-1})$ (in fact, \eqref{eq8.9.2} follows by repeated application of \eqref{eq8.9.3}).
Note that it is obvious that \eqref{eq8.9.2}--\eqref{eq8.9.3} are relations, but not that they are braid relations.
The above also  implies that  for any braid relation 
\eqref{eq8.9.1}, one has 
\begin{equation}\label{eq8.9.4}
m^{a}_{r_{1},s_{1}}=m^{b}_{r_{n}^{*},s_{n}^{*}}=m^{a'}_{s_{1}^{*},s_{2}}.
\end{equation}

Next, for all $r\in S$, say  $r\in \lrsub{a}{S}{c}$, there is one map $\pi_{r}=\pi_{r}^{(G,S)}$, 
giving a bijection (with inverse $\pi_{r^{*}}$) 
\begin{equation}\label{eq8.9.5} \pi_{r}\colon \mset{t\in \lsub{c}{S}\mid m^{c}_{r^{*},t}\neq \infty}\xrightarrow{\cong} \mset{s\in \lsub{a}{S}\mid m^{c}_{r,s}\neq \infty} \end{equation} It is defined by setting $\pi_{r}(r_{2}):=s_{1}$
  for each braid relation \eqref{eq8.9.1} with $r_{1}=r$ and by setting  $\pi_{r}(r^{*}):=r$.   Using the above  fact about
  cyclic shifts of braid relations,  one easily sees that the braid relations are completely determined as stated by the data consisting of the   family of Coxeter matrices $\set{M^{a}}_{a\in \ob(G)}$ and partially defined maps $\set{\pi_{r}}_{r\in S}$: informally, the partially defined maps constitute the local data necessary to   determine in turn successive 
 simple generators appearing on the two sides of a   braid relation, left to right, 
 beginning  with  the leftmost simple generator,  until each side of  the relation reaches the 
 length specified by the corresponding entry of the Coxeter matrix.

\begin{rem*}      An even,  
  $C_{2}$-system $(G,S)$ is  completely determined (up to isomorphism) by $\ob(G)$, the set  $S$ with its decomposition $S=\dotcup_{a,b\in \ob(G)}\lrsub{a}{S}{b}$ and the map $s\mapsto s^{*}\colon S\to S$, the Coxeter matrices $M^{a}=(m^{a}_{r,s})_{r,s\in \lsub{a}{S}}$ for $a\in \ob(G)$ and the bijections $\pi_{r}$ for $r\in \lrsub{a}{S}{b}$.
  Simple necessary and sufficient conditions for  arbitrarily specified data  of this type to determine an even $C_{2}$-system are not known. \end{rem*}

  \subsection{Groupoid representation on its simple generators}\label{ssa8.10}
  An interval finite rootoid $\CR=(G,\L,N)$ with atomic generators $A:=A_{\CR}$ is said to be  
  \emph{$n$-complete}, where $n\in \Nat$, if for any $a\in \ob(G)$ and any elements $r_{1},\ldots ,r_{n} \in \lsub{a}{A}$, the join $\join_{i} r_{i}$ exists in $\lsub{a}{G}$ in its weak order.  Note that  $n$-completeness holds automatically  for $n=0,1$,  and that  if $\CR$ is complete, it  is $n$-complete for all $n\in \Nat$.     
   Let $(G,S)$ be an even $C_{2}$-system and $\CR:=\CJ_{(G,S)}$. According to \cite[Lemma 3.7]{DyGrp1} and Theorem \ref{ssa8.8}, one has $A_{\CR}=S$.
     Define $(G,S)$  to be \emph{$n$-complete} (resp., \emph{complete}) if $\CR$ is $n$-complete (resp., complete).   Obviously if $(G,S)$ is complete, it is $n$-complete for all $n$. 
  
 \begin{exmp*} A Coxeter system $(W,S)$ is $n$-complete, where $n\in \Nat$,  if and only if for all subsets $J$ of $S$ with $\vert J\vert \leq n$, the standard parabolic subgroup $W_{J}$ is finite.  In particular, if $W$ is of infinite rank $\vert S\vert$ but all its  standard parabolic subgroups $W_{J}$ of finite rank are finite, then $(W,S)$ is $n$-complete for all $n\in \Nat$ but not complete. \end{exmp*} 
   
 Fix an even $C_{2}$-system
    $(G,S)$.    Note that  $(G,S)$ is   $2$-complete if and only if there are no infinite entries in any of the Coxeter matrices $M^{a}$ for $a\in \ob(G)$. This holds if and only if 
       for all
   $a,b\in \ob(G)$ and $r\in \lrsub{a}{S}{b}$, the  map
   $\pi_{r}$ is a bijection $\pi_{r}\colon \lsub{b}{S}\to \lsub{a}{S}$.
   One defines $(G,S)$ to be \emph{$5/2 $-complete} if it is $2$-complete
    and the maps $\pi_{r}$  determine  a representation (i.e. functor) $\pi\colon G\to \setc$
    such that for all $a\in \ob(G)$, $\pi(a)=\lsub{a}{S}$ and for  all $r\in S$,    $\pi(r)=\pi_{r}$. If  $\pi$ exists,  it is uniquely determined, and will be denoted as  $\pi^{(G,S)}$.  Note that if $G$ is connected and simply connected, then $\pi$ is trivial (if it exists) by \cite[Lemma 1.14]{DyGrp1}.
         \begin{thm*} \begin{num} \item A $3$-complete, even $C_{2}$-system is $5/2$-complete. 
         \item A  $2$-complete, even $C_{2}$-system $(G,S)$ with $G$ finite is both complete and  $5/2$-complete. \end{num}
    \end{thm*}
  
\begin{rem*} 
Define a $5/2$-completion of an even $C_{2}$-system $(G,S)$ to be a $5/2$-complete  even  $C_{2}$-system $(H,R)$ with the following additional properties:
\begin{conds}\item  $G$ is a subgroupoid of $H$ and  $S\seq R$. 
\item For $a\in \ob(G)$,  $\lsub{a}{G}$ is an order ideal of $\lsub{a}{H}$
(where $\lsub{a}{H}$ has the weak order from the associated rootoid $\CJ_{(H,R)}$).
\item For $a,b\in \ob{G}$, $r\in \lrsub{a}{S}{b}$ and $s\in \lsub{b}{S}$ such that $\pi^{(G,S)}_{r}(s)\in \lsub{a}{S}$ is defined, one has $\pi^{(G,S)}_{r}(s)=\pi^{(H,R)}_{r}(s)$.
\end{conds}
Here, (iii) follows from (i)--(ii) but is stated explicitly for emphasis.  It is an open question whether every even $C_{2}$-system has a $5/2$-completion. A special case,  open for infinite $G$, is the question of whether  every $2$-complete, even $C_{2}$-system $(G,S)$ is $5/2$-complete. 
 These questions have an obvious similarity to open questions involving  notions of  completions of rootoids as considered in Section \ref{sa10} (which would not in general give $C_{2}$-systems from $C_{2}$-systems), but their relationship is unclear.
   \end{rem*}

      \subsection{Coxeter groupoids and $C_{2}$-systems}\label{ssa8.11}
Coxeter (and Weyl) groupoids   are   classes of groupoids $G$  with  distinguished 
generating sets $S$ (closed under inverses) defined in \cite{HY}  and \cite{CH} (note  that \cite{HY} uses a different but equivalent notion of groupoid to that   used in these papers). They will be studied in relation to rootoids in subsequent papers, but it is appropriate to include some imprecise general remarks here to indicate  in what respects the resulting pairs $(G,S)$ are  similar to, and  differ from, even $C_{2}$-systems. As a first comment in that regard,  note that a Coxeter groupoid with one object is a Coxeter group, but many other groups  arise as underlying groupoids of even $C_{2}$-systems (see \ref{ssa8.14}--\ref{ssa8.16}). Correspondingly, Coxeter groupoids share more properties in common with Coxeter groups than general even $C_{2}$-systems (but they also do not form a closed class with respect to basic constructions of these papers; see the comments on covering quotients below, and also \ref{ss13.7}).

 A Coxeter   groupoid $G$ may be defined by a  
presentation by generators $S$ and relations    specified   in terms of certain defining data  consisting of  families of $I$-indexed Coxeter matrices,  for a fixed index set $I$, related by additional combinatorial data. See \cite{CH} and \cite{HY} for  precise definitions, which  provide  a specified  indexing of the canonical generating set $S$ of $G$  by $I\times \ob(G)$,  inducing bijections  $\lsub{a}{S}\cong I$ for $a\in\ob(G)$. This  indexing  may  not  be completely  determined   by the pair $(G,S)$ alone, and provides additional structure, when the Coxeter matrices contain infinite entries.

The theory of Coxeter groupoids  in  \cite{CH} and \cite{HY} is developed assuming  existence of a  root system,  realized in a family of real vector spaces, satisfying certain  conditions  in terms of the  defining data. In particular, the root systems satisfy an integrality condition (roots are integral linear combinations of simple roots), though \cite[Remark 5]{HY} raises the possibility  of relaxing the condition to allow real linear  combinations.
It can be shown, however, that root systems  even in the relaxed sense do  not  exist for all possible sets of defining data (i.e. the open question at the end of  \cite[Remark 5]{HY} has a negative answer),  and necessary and sufficient conditions   for their existence  are not known.   
On the other hand,  Remarks (2)--(3) in \ref{ssa9.6}
suggest  that there may be  natural examples  of   Coxeter groupoids
sharing many of  the formal properties developed in \cite{CH} and \cite{HY},  without root systems in the sense there, but  with   root systems in a more abstract sense than (even the relaxed sense)   considered in   \cite{CH} and \cite{HY} above.

   Fix  a  Coxeter groupoid $G$  with canonical generators $S$  and, for simplicity in explanation here, with a root system
  $\Phi$ in the sense of \cite{CH} and \cite{HY}.
  View $\Phi$ as a representation of $G$ in the category $\setc_{\pm}$ in a similar way as for root systems of Coxeter groups. Then  the pair
    $(G,\Phi)$  turns out to be  a principal, rootoidal 
  signed groupoid-set  with simple generators $S$,
  and  the braid 
  presentation  of $(G,S)$  coincides with the defining presentation of 
  $(G,S)$ as Coxeter groupoid.
  This can be proved    by verifying WEC and SLC  using   
  properties of $G,\Phi, S$ established in  \cite{CH} and \cite{HY}, though in these papers we shall give the result assuming only  existence of a root system in a more general sense. 
  Further,  as  a  consequence of the  existence of  the canonical indexing $S\cong \ob(G)\times I$ as part of the definition of Coxeter groupoid, one has also the following fact:  \begin{conds}\item [($*$)] 
  There exists a representation $\pi'\colon G\to \setc$ such that: 
  \begin{subconds} \item   For
  all $a\in \ob(G)$, one has $\pi'(a)=\lsub{a}{S}$. \item  For  all $a,b\in\ob{G}$, $r\in \lrsub{a}{G}{b}$, $s\in \lsub{b}{G}$ with $\pi_{r}(s)$ defined, one has $\pi_{r}(s)=(\pi'(r))(s)\in \lsub{a}{S}$.
  \item  The restriction of $\pi'$ to each component of $G$ is  a trivial representation.   \end{subconds}
     \end{conds}
 
Here, $\pi_{r}=\pi_{r}^{(G,S)}$ as in \ref{ssa8.9}. Note that by \cite[Lemma 1.14]{DyGrp1}, the condition (iii) can be omitted if $G$ is simply connected. These last  comments are closely related to the following  facts, which are stated imprecisely  here: the class of even $C_{2}$-systems is closed under  taking both coverings  and covering quotients, while the class of  Coxeter groupoids is  closed under  coverings but not covering quotients (see Example \ref{ssa9.5}(5)). 

The condition ($*$) depends only on $(G,S)$. The class of even $C_{2}$-systems $(G,S)$ which satisfy  ($*$) may be regarded as a natural extension of the class of Coxeter groupoids with root system as in 
 \cite{CH} and \cite{HY}, to a class of Coxeter groupoids with abstract root systems. Note that explicitly specifying a representation $\pi'$ as in   ($*$)
 provides additional structure beyond that from the pair $(G,S)$ itself, in general.
  
%  \begin{rem*} Some  obvious    extensions of  definitions of  Bruhat order and Iwahori-Hecke algebras of Coxeter groups  to even $C_{2}$-systems  either require, or are apparently  most interesting assuming,  the presence of (involutory) self-composable simple generators and the validity of conditions related to ($*$). It is not known if such definitions have any natural significance   and they will not be studied in these papers.\end{rem*}
%  

 \subsection{Solvability of the word problem} \label{ssa8.12}  Tits solution of the word problem for Coxeter groups (\cite[Ch 4, \S1, Ex  13]{Bour}) applies mutatis mutandis to even $C_{2}$-systems. Slightly more precisely, in an even $C_{2}$-system $(G,S)$, any non-reduced expression is braid equivalent to an obviously non-reduced one (i.e.  one with a consecutive subexpression $ss^{*}$ for some $s\in S$), and any two  reduced expressions of the same element differ by  a sequence of braid operations.  Of course,  as for Coxeter systems $(W,S)$, for this to give rise to an actual algorithm,   finiteness of $S$ is necessary.

 \subsection{Standard parabolic subgroupoids}\label{ssa8.13} 
 For a rootoid $\CR=(G,\L,N)$, a \emph{standard parabolic subgroupoid } is defined as a subgroupoid  $H$ of $G$
 such that for each $a\in \ob(H)$, $\lsub{a}{H}$ is a join-closed 
 order ideal of $\lsub{a}{G}$ in weak right order i.e. an order ideal such if $h_{i}$ is a non-empty family in $\lsub{a}{H}$ such that $g:=\join_{i} h_{i}$ exists in $\lsub{a}{G}$, then $g\in \lsub{a}{H}$. A standard parabolic subgroupoid  $H$ is \emph{saturated} if $\ob(H)=\ob(G)$; any  standard parabolic subgroup can be enlarged  to  a saturated one  by possibly adjoining some  trivial components.
   
 \begin{exmp*} The saturated  standard parabolic subgroupoids of  $\CC:=\CC_{(W,S)}$, where $(W,S)$ is a Coxeter system, are  the standard parabolic  subgroups ${W_{J}}$ for $J\seq S$, where $W_{J}$ is the standard parabolic subgroup of $W$ generated by $J$ (both $W_{J}$ and $W$ being regarded as one-object groupoids). The only non-saturated standard parabolic subgroupoid is the empty subgroupoid. \end{exmp*}
Certain subtleties arise, for instance   with the  notions of   rank of standard parabolic subgroupoids (see the  examples of  \ref{ssa8.16} below). 
However, basic results about longest 
 elements, standard parabolic subgroups and  shortest coset  representatives of Coxeter groups extend mutatis mutandis to even  $C_{2}$-systems $(G,S)$.  (The results about cosets apply  only to saturated standard parabolic subgroups.)  Some closely related results in the special case of Weyl groupoids can be found in \cite{HV}, but the notion of standard parabolic  subgroupoid used there is not the same as  the restriction to Weyl groupoids of the notion here. 
 
\subsection{Reflection systems} \label{ssa8.14}  Reflection systems $(G,X)$ in the sense of \cite{DyRef}   are groups $G$ with 
  presentations by generators $X$ satisfying relations   obtained 
  by suitably mixing the  relations occurring in standard  presentations of  
  Coxeter groups, and  coproducts of cyclic groups in the 
  categories of groups and abelian groups.     The results  of this subsection and \ref{ssa8.15} imply that  an extensive class of (even) $C_{2}$-systems may  be  constructed from reflection systems; the proofs, 
  given in   subsequent papers, involve verifying the  WEC and SLC.

  To define reflection systems precisely,  denote an 
  alternating product $aba\cdots $ of length $m$ of elements  $a$ 
  and $b$  in a group as  $(aba\cdots)_{m}$. Let $X$ be a set and 
  suppose given $m_{x}\in \Nat_{\geq 2}\cup\set{\infty}$ for $x\in X
  $, $m_{x,y}=m_{y,x}\in \Nat_{\geq 2}\cup \set{\infty}$ for $x\neq y
  $ in $X$ such that  $m_{x,y}\not\in \set{2,\infty}$ implies $m_{x}
  =m_{y}=2$.
  A reflection system is a pair of a group $G$ and subset $X\seq G$ such that  for some $m_{x},m_{x,y}$ satisfying the above conditions,
\begin{multline}\label{eq6.14.1}
G\cong  \mpair{X\mid  x^{m_{x}}=1 \text{ \rm if $x\in X$, $m_{x}\neq \infty$},\\ (xyx\cdots )_{m_{x,y}}=(yxy\cdots )_{m_{x,y}}\text{ \rm if $x,y\in X$, $x\neq y$ and  $m_{x,y}\neq \infty$}}
\end{multline}
(more precisely, the relations on the right hold in $G$ and the canonical map from the group with presentation on the right to $G$ is an isomorphism). 

These conditions imply that the canonical map $X\to G$ is injective, that $m_{x}=\ord(x)$ for $x\in X\seq G$, and that
for distinct $x,y\in X$, $m_{x,y}$ is the minimum $n\in \Nat_{\geq 2}$ such that $(xyx\cdots)_{n}=(yxy\cdots)_{n}$ if such an $n$ exists and otherwise $m_{x,y}=\infty$. 
\begin{thm*} Let $(G,X)$ be a reflection system in the sense of \cite{DyRef} and set 
 $S:=X\cup X^{*}$. Then  $(G,S)$ is a $C_{2}$-system, which is even if and only if  there exists no $x\in X$ for which  $m_{x}$ is a  finite,  odd integer.\end{thm*}
 
\begin{rem*} In particular,  arbitrary coproducts $G$ in the category of groups (resp., in the category of abelian groups)  $G$ of 
 cyclic groups
 $\pair{x_{i}}$ of \emph{even or infinite order}, for $i\in I$ give 
 \emph{even}  $C_{2}$-systems $(G,S)$ with 
 $S=\mset{x_{i},x_{i}^{*}\mid i\in I}$.   
  On the other hand, if one allows \emph{odd order} cyclic groups 
  $\pair{x_{i}}$ as well above, the Theorem gives only  that $(G,S)$
is a  $C_{2}$-system.  The class of  general  $C_{2}$-systems seems to be a natural extension of the class  of the even ones, but its theory   is substantially more involved and we leave it  largely open. Note that while the condition that $(G,S)$  be a $C_{2}$-system entails no 
restriction on the groupoid $G$ (see Example \ref{ssa8.16}(a)), it is still a 
 stringent restriction on the pair $(G,S)$.
  \end{rem*}

  \subsection{Semidirect products}\label{ssa8.15} Let $G$ be a groupoid and $H$ be a group.
Assume that there is a given group homomorphism $\a\colon H\to \Aut_{\grpdc}(G)$ denoted as $h\mapsto \a_{h}$.  For $h\in H$ and any object or morphism $x$ of $G$, write $h(x):=\a_{h}(x)$.
Then the semidirect product groupoid $K:=G\rtimes_{\a} H$ is defined as follows. One has \begin{equation}
\ob(K)=\ob(G),\qquad \mor(K)=\mor(G)\times H,
\end{equation}
 \begin{equation}
 \cod(\d,h)=\cod(\d),\qquad  \domn(\d,h)=h^{*}(\domn(\d))
 \end{equation} and 
 \begin{equation}
 (\d,h)(\g,h')=(\d h(\g),hh')
 \end{equation} when defined, where $\d,\g\in \mor(G)$ and $h,h'\in H$.

 Now assume in addition that  $(G,S)$  and $(H,R)$ are $C_{0}$-systems (regarding $H$ as one-object groupoid)  and that the action by $H$ on $G$ preserves $S$. We say that $(H,R)$ acts on $(G,S)$.
 Let \begin{equation} T:=\mset{(\d,1_{H})\mid \d\in S}\dotcup\set{(1_{a},r)\mid a\in \ob(G),r\in R}\seq \mor(K).\end{equation}  Then $(K,T)$ is a $C_{0}$-system called the semidirect product of $(G,S)$ and $(H,R)$ and denoted as $(K,T)=(G,S)\rtimes (H,R)$.

 \begin{thm*}  Let $i\in \set{0,1,2}$. If $(G,S)$ and $(H,R)$ are $C_{i}$-systems such that $(H,R)$ acts on $(G,S)$, then  the semidirect product $(G,S)\rtimes (H,R)$  is a $C_{i}$-system, which is even if and only if 
  $(G,S)$ and $(H,R)$ are both even. 
  \end{thm*}  
  \begin{rem*} Similarly, $C_{2}$-systems are closed under natural notions of infinite direct sums (special subgroupoids of infinite products) and  coproducts  (disjoint union). Those with underlying groupoid equal to  a group are also closed under  constructions giving   free products  (in the category of groups) and restricted wreath products of underlying groups.  Several of these constructions have versions for general rootoids, which are  compatible with those  for (even) $C_{2}$-systems; the compatibility amounts to a non-trivial closure property of rootoids from (even) $C_{2}$-systems under such general constructions. 
      The study of  $C_{2}$-systems in relation to amalgamated free products and HNN extensions, and, more generally,  arboreal group theory (fundamental groups and groupoids of graphs of groups, etc; see \cite{STrees}) may be of interest.  \end{rem*}
 \subsection{} \label{ssa8.16} To finish this section, here are    examples of $C_{0}$-systems with their underlying groupoid equal to a group. They  can 
be checked to be $C_{2}$-systems either directly or using the results quoted without proof above. The braid presentations, and in most cases the maps $\pi_{r}$,  of the even $C_{2}$-systems are  also given. 

 \begin{exmp*} 
 (1) Let $G$ be any groupoid and let $S$ be the set of all non-identity morphisms of $G$. Then $(G,S)$ is a  $C_{2}$-system.  For $a\in \ob(G)$, the weak order $\lsub{a}{\leq}$ on $\lsub{a}{G}$ has minimum element $1_{a}$, and the   elements of $\lsub{a}{G}\sm\set{1_{a}}$ are incomparable.

(2) Let $G$ be a non-trivial  cyclic group with generator $x$, and let 
$S:=\set{x,x^{*}}\seq G$. Then $(G,S)$ is a  $C_{2}$-system, 
which is even if and only if   the order $m$ of $x$ is  even  or infinite; the corresponding 
rootoid $\CJ_{(G,S)}$ is  complete if and only if  $x$ is of finite, even order.
Suppose from now that  $(G,S)$ is even.  If  $m=2$,    the   Coxeter matrix  at the (unique) object of $G$ is specified by the Coxeter graph
$\xymatrix{*+[o][F-]{}
}$ and the braid   presentation  $\mpair{x\mid x^{-1}=x}$ involves only a single  relation, which is trivial.
 Henceforward assume that  $m\geq 4$. Then the   Coxeter matrix of $(G,S)$  is specified by a Coxeter graph 
$
\xymatrix{*+[o][F-]{}\ar@{-}[r]^{m/2}&*+[o][F-]{}
}
$.
The braid presentation is 
  \begin{equation*}
  \mpair{x,x^{*}\mid x^{-1}=x^{*}, (xx\cdots )_{m/2}=(x^{*}x^{*}\cdots )_{m/2}}
  \end{equation*}
   where in the relation involving $m/2$,  there are $m/2$ factors in each product if $m$ is finite 
  and the   relation is omitted if $m=\infty$. One has $\pi_{x}(x)=x^{*}$ and $\pi_{x^{*}} (x^{*})=x$, while $\pi_{x}(x^{*})$ and $\pi_{x^{*}}(x)$ are defined if and only if $m$ is finite, in which case $\pi_{x}(x^{*})=x$ and $\pi_{x^{*}}(x)=x^{*}$. Note that
  although $(G,S)$ is an even $C_{2}$-system,  in this case $S$ is not a minimal generating set of $G$.

(3) Let $(W,S)$ be a dihedral Coxeter system i.e. one with $\vert S\vert =2$.
Write $S=\set{r,s}$.  Let $x=rs$, so $x^{*}=sr$. Let $R=\set{x,x^{*},r=r^
{*}}$, which is another set of generators of $W$ as group. Then $(W,R)$ is a $C_{2}$-system; it is even 
if and only if  the order $m$ of $rs$ in $W$ is 
even or infinite. The corresponding rootoid 
$\CJ_{(W,R)}$ is  complete if and only if  $m$ is finite and even. 
If $m=2$, then $(W,R)$ is a dihedral Coxeter system. Assume $m> 3$ is even. The Coxeter graph of the unique object of $(W,R)$ is then
\begin{equation*} 
\xymatrix{{x}\ar@{-}[r]^{m/2}&{x^{*}}&r 
}
\end{equation*} and the braid presentation is
 \begin{multline*}
  \mpair{x,x^{*},r\mid x^{-1}=x^{*}, \quad r^{-1}=r,\quad  rx=x^{*}r,\quad  rx^{*}=xr\\ (xx\cdots )_{m/2}=(x^{*}x^{*} \cdots )_{m/2} }.
  \end{multline*}
  The description of the maps $\pi_{r}$, for $r\in R$,  is left to the reader in this example.
  
  (4)  Suppose in (3) that $W$ is dihedral of order $8$, and let $t=rsr\in W$.
Set $V:=\set{r,s,t}$.  One can check that  $(W,V)$ is an even $C_{2}$-system. The  unique weak order of $\CJ_{(W,V)}$ is Boolean of rank $3$ i.e.  isomorphic to the lattice of subsets of $\set{r,s,t}$.  The Coxeter graph at  the unique object consists of three isolated vertices
\begin{equation*}\xymatrix{ {r}&{s}&{t}}\end{equation*} but it  still corresponds to a  non-trivial braid presentation
 \begin{multline*}
  \mpair{r,s,t\mid r^{-1}=r, \quad s^{-1}=s, \quad t^{-1}=t,\quad  ts=st,\quad  rt=sr, \quad rs=tr }
    \end{multline*} since the representation $\pi^{(W,V)}$ of $W$ on $V$  is non-trivial ($\pi_{s}=\pi_{t}=\Id_{V}$ and $\pi_{r}=(s,t)$ is the transposition on $V$ which interchanges $s$ and $t$).
  Since $(W,V)$ is not a Coxeter system,  this  shows that  WEC does not imply EC.
  Also, a $C_{2}$-system $(G,S)$ with $G$ a group and in which $S$ consists of involutions need not be a Coxeter system.
  
  (5) For a Coxeter system $(W,S)$, the Coxeter matrix of $(W,S)$ as $C_{2}$-system coincides with the Coxeter matrix of  $(W,S)$ as a Coxeter system. Denote the Coxeter matrix as $(m_{r,s})_{r,s\in S}$. 
  The braid presentations of $(W,S)$ as $C_{2}$-system and  as  Coxeter system also coincide.  One has, for $r,s\in S$, that $\pi_{r}(s)=s$ if $m_{r,s}\neq \infty$ and $\pi_{r}(s)$ is undefined if $m_{r,s}=\infty$. 
  
  (6) The example (4) (and  also (5)) is a special case of the following one. Let $(W,S)$ be a Coxeter system. Suppose that $S=I\dotcup J$ where no element of $I$ is conjugate to any element of $J$.  Let $\wt J:=\mset{wsw^{-1}\mid w\in W_{I},s\in J}$ and $\wt W$ be the subgroup of $W$ generated by $\wt J$.   It is shown in \cite{Gal} and in \cite{DySd} that $(\wt W,\wt J)$ is a Coxeter system  and $W$ is the semidirect product $W=W_{I}\ltimes \wt W$ . Set $V:=I\dotcup \wt J$. It follows from
Theorem \ref{ssa8.15} that $(W,V)$ is an even $C_{2}$-system.   The braid relations   are the braid relations of $(W_{I},I)$, those of $(\wt W,\wt J)$ and the relations $rs=\tilde s r$ for $r\in I$, $s,\tilde s\in \wt J$ with $\tilde s=rsr^{-1}$.  Let $(m_{r,s})_{r,s\in S}$ be the Coxeter matrix of $(W,S)$, let
$(\wt m_{r,s})_{r,s\in \wt J}$ be that of $(\wt W,\wt J)$ and let  $(m'_{r,s})_{r,s\in V}$ be the Coxeter matrix of $(W,V)$ as $C_{2}$-system. For any $r,s\in V$,  one has   $m'_{r,s}=2$ if one of $r$, $s$ is in $I$ and the other is in $\wt J$, while $m'_{r,s}=m_{r,s}$ if $r,s$ are both in $I$ and $m'_{r,s} =\wt m_{r,s}$ if both $r,s$ are in $\wt J$. 
Note that one could have $r\in I, s\in \wt J$ with $m'_{r,s}=2$
but $rs$ of infinite order as an element of the group $W$.
Recall that for $r,s\in V$,    $\pi_{r}(s)$ is defined if and only if $m'_{r,s}\neq \infty$.   The maps $\pi_{r}$ for $r\in V$ may be specified as restrictions of  maps $\pi'(r)$ where $\pi'$ is  the permutation  representation of the group $W$ on the set $V$ defined   as follows  First, the extended representation $\pi'$ is  trivial when restricted to $\wt W$, so it factors through a representation of $W_{I}\cong W/\wt W$.  On $V=I\dotcup \wt J$, $W_{I}$ acts trivially on $I$ and  by conjugation on $\wt J$. 
      \end{exmp*}

 \section{General constructions of protorootoids}
 \label{s9}
    \subsection{Protorootoid of a simple graph}\label{ssa9.1}
   We fix some notation and terminology concerning  simple (undirected) graphs. A \emph{simple graph} $\G$  is by definition a pair $(V,E)$ where $V$ is a set (the \emph{vertex set} of $\G$) and $E$ (the \emph{edge set} of $\G$) is a set of two-element subsets of $V$ (called edges of $\G$).
  A \emph{path} in
$\G$ is a sequence $(v_{n},\ldots, v_{0})$ in $V$, where $n\in \Nat$, such that $\set{v_{i},v_{i-1}}\in E$ for $i=1,\ldots, n$. This path is called a path of length $n$ from $v_{0}$ to $v_{n}$. The path is called a \emph{cycle} (of length $n$) if  $v_{0}=v_{n}$, and a \emph{simple cycle} (of length $n$) if  it is a cycle, $n\geq 3$ and $v_{i}\neq v_{j}$ for $0\leq i<j<n$.  
 The graph  $\G$ is said to be  \emph{connected} if it is non-empty and for any two vertices $x,y\in V$, there is a path from $x$ to $y$. An \emph{isomorphism} $(V,E)\to (V',E')$ of simple graphs is a bijection
 $\s\colon V\to V'$  such that for distinct  $x,y\in V$, one has $\set{x,y}\in E$ if and only if $\set{\s(x),\s(y)}\in E$. A \emph{subgraph} of $(V,E)$ is a graph
 $(V',E')$ with $V'\seq V$ and $E'\seq E$. The subgraph $(V',E')$ is \emph{full}
 if $E'=E\cap\mset{\set{x,y}\mid x,y\in V', x\neq y}$. A component of $(V,E)$ is a full, connected subgraph $(V',E')$ with $V'\neq \eset$ maximal under inclusion.
 
 A \emph{forest} is a simple graph with no simple cycles.  A \emph{tree} is a connected forest. Any simple graph $\G$ has, by Zorn's lemma,  a \emph{maximal subforest} $F$ i.e. a subgraph  $F$ of $\G$ which is a forest such that its sets of vertices and sets of edges are both maximal (under inclusion) amongst those of subgraphs of $\G$ which are forests.   It is easily seen (see e.g. \cite{STrees}) that the vertex set of a maximal subforest $F$ of $\G$ coincides with that of $\G$.

 Attach to the simple graph $\G=(V,E)$ the (unique up to isomorphism) simply 
 connected groupoid $G=G_\G$ with $\ob(G)=V$ in which there is a morphism 
 $b\to a$ in $G$ if and only if  there is a path from $b$ to $a$ in $\G$.  
 Concretely, we set 
 \begin{equation}\label{eqa9.1.1} \Hom_{G}(b,a)=\begin{cases}\set{(a,b)},&\text{if there is a path from $b$ to $a$ in $\G$}\\
 \eset,&\text{otherwise}\end{cases}\end{equation}
 with composition $(a,b)(b,c)=(a,c)$ when $(a,b)$ and $(b,c)$ are morphisms of $ G$. The groupoid $G$ has a generating set
 $S:=\mset{(a,b)\mid a,b\in V, \set{a,b}\in E}$. 
  Also, the groupoid $G$ is simply connected, and it is connected if  and only if $\G$ is connected.  More generally, the components of $G$ naturally correspond to those of $\G$.

   The definitions of Section \ref{sa8} for  the $C_{0}$-system $(G,S)$ can be transferred to the graph $\G$ as follows.
 
 \begin{defn*} \begin{num} \item  For $i\in\set{0,1,2}$, the graph  $\G$ is an (even) $C_{i}$-graph if $(G,S)$ is an 
 (even) $C_{i}$-system.

 \item The protorootoid (resp., set protorootoid, signed groupoid set) of $\G$ is   $\CJ_{\G}:=\CJ_{(G,S)}$ (resp., $\CJ'_{\G}:=\CJ'_{(G,S)}$, $J_{\G}:=J_{(G,S)}$). 
 \item For an even $C_{0}$-graph $\G$,  
 define  $\ol{\CJ}_{\G}:=\ol{\CJ}_{(G,S)}$,   $\ol{\CJ}_{\G}':=\ol{\CJ}_{(G,S)}'$ and   $\ol{J}_{\G}:=\ol{J}_{(G,S)}$.\end{num}
 \end{defn*}
 Note that $C_{0}$-graphs are just simple graphs, and saying that a $C_{i}$-graph is even  means it has no odd length cycles.  
 \begin{exmp*}   (1) Suppose that $n\in \Nat_{\geq 3}$ and  the 
 simple  graph $\G=(V,E)$ is an $n$-cycle. Write $V=\set{v_{1}, \ldots, v_{n}}$ and 
 $E=\mset{\set{v_{i},v_{i+1}}\mid i=1,\ldots n}$ where
  $v_{n+1}:=v_{1}$. Then (see Examples \ref{ssa9.4}(3)--(4))  $\G$ is a $C_{2}$-graph, which is 
  even if and only if  $n$ is even.  The corresponding rootoid $\CJ_{\G}$ is 
    complete if and only if  $n$ is even. The rootoids $\CJ_{\G}$ from the graphs here  are easily seen to be  the universal covering  rootoids of   the rootoids of   finite cyclic groups  of order $n\geq 3$ in Example 
    \ref{ssa8.16}(2).  (The universal covering rootoid of the rootoid from the
    infinite cyclic group is $\CJ_{\G}$ where $\G$ is a doubly infinite path).
 
 (2) If $\G$ is a  forest  then it is an even $C_{2}$-graph; further,  in that case,
 $\CJ_{\G}$ is complete if and only if   $\G$ contains no distinct edges with a  common vertex. 
 
 (3) More generally than (1), let  $(G,S)$ be a  $C_{0}$-system  such that $G$ is connected. 
Define the Cayley graph $(V,E)$ of  $(G,S)$ as follows.
Fix $a\in \ob(G)$. Set $V:=\lsub{a}{G}$ and
 \begin{equation}\label{eqa9.1.2} E:=\mset{\set{g,gs}\mid b\in \ob(G),g\in \lrsub{a}{G}{b}, s\in 
 \lsub{b}{S}}.\end{equation}  It is easy to see that, up to  isomorphism of simple graphs, this is independent of  the choice 
 of $a$.  Then  for $i\in \set{0,1,2}$, $(G,S)$ is an (even) $C_{i}$-system if and only if   its Cayley 
 graph $(V,E)$ is an (even) $C_{i}$-graph.   (The easy details of the argument will be given in subsequent papers).   This fact  implies  that the conditions for a general  $C_{0}$-system $(G,S)$ to be 
  a $C_{i}$-system  amount to  certain  metric conditions (with respect to the path length metric) on the  vertices of  the  Cayley graphs of the components of $G$. 
   \end{exmp*}

  \subsection{Protorootoid attached to a rainbow graph}
  \label{ssa9.2} A \emph{ rainbow 
 graph} $X$ is defined to be  a quadruple $X=(V,E,c,B)$ where 
 $\G:=(V,E)$ is a simple graph, $B$ is a Boolean ring  and 
 $c\colon E\to B$ is a function,
 subject to the requirement that for any cycle $(v_{n},\ldots, v_{0})$ in the graph,
 $c(e_{1})+\ldots +c(e_{n})=0_{B}$ where $e_{i}:=\set{v_{i-1},v_{i}}$. 
 
The rainbow graph $X$ is said to be defined over $B$. 
 A particular case which motivates the terminology, is 
 when $B=\wp(Y)$ for some set $Y$. The set $Y$ is  called the color set of $X$ and   elements of $Y$ are called colors.
  Then $c$   determines  a labelling of each edge $e\in E$
 by a set $c(e)$ of colors subject to the requirement
 that the symmetric difference of color labels of the edges  in any cycle is empty i.e. each color appears an even number of times in the color sets attached to the edges of the cycle.

 Let $G=G_{\G}$ be 
 the groupoid and $S$ its set of generators attached to the simple 
 graph $\G$ as in \ref{ssa9.2}.  Let $\L$ be the constant functor $G\to
  \bringc$ with constant value $B$.  Define a  cocycle $N$ of 
   $G$ for $\L$ as follows; for any $a,b\in G$ which are 
   joined 
   by a path in $\G$ with successive vertices $a=a_{n}, a_{n-1},
   \ldots, a_{0}=b$,  set  $\lsub{a}{N}((a,b)):=c(e_{1})+\ldots +c(e_{n})
   \in\lsub{a}{\L}$  where $e_{i}:=\set{a_{i-1},a_{i}}\in E $ (here 
   as in \ref{ssa9.1}, $(a,b)$ denotes  the 
   unique morphisms from $b$ to $a$ in $G$). This is well-defined 
   and 
   gives a cocycle, as required by  the definition of rainbow graph.   
   \begin{defn*} \begin{num} \item The  protorootoid of the rainbow graph $R$  is 
   defined as $\CJ_{R}:=(G,\L,N)$.
   \item The rainbow graph $R$ is rootoidal if $\CJ_{R}$ is a rootoid.\end{num}\end{defn*} 
   Note that if  $X$ is defined over a Boolean algebra $B=\wp(Y)$ of sets, then $\CJ_{X}$ is the associated protorootoid $\CJ_{X}=\mathfrak{I}(\CJ'_{X})$ of a set protorootoid $\CJ'_{X}:=(G,Y',N)$ where $Y'\colon G\to \setc$ is the constant functor with value $Y$. In that case, one also has an associated signed groupoid-set $J_{X}:=\mathfrak{K}(\CJ'_{X})$.

 \subsection{}  \label{ssa9.3} Though the following result  is   equivalent   to a  well known  description of the cycle space of a simple graph,  details of its  proof are given for completeness. It provides a convenient means of specifying examples of   rainbow graphs, but  is not  useful  in characterizing the more interesting  subclass of  rootoidal rainbow graphs.

  \begin{prop*}   Let $\G=(V,E)$ be a simple graph, $B$ be a Boolean ring and $T=(V,E')$ be a maximal subforest of $\G$, with edge set $E'\seq E$.
 Then  any function $c'\colon E'\to B$ extends uniquely to a function
 $c\colon E\to B$ such that $(V,E,c,B)$  is a rainbow graph.
\end{prop*}
 \begin{proof}
    Let $\G$ be as above. Let $U$ be a  vector space
   with $E$ as $\bbF_{2}$-basis, and $U'$, $U''$ be the subspaces of $U$ spanned by $E'$, $E\sm E'$ respectively, so $U=U'\oplus U''$. The \emph{cycle space}  of $\G$ is the subspace $C$ of $U$ spanned by all its  elements $e_{1}+\ldots +e_{n}$  such that there is a cycle $(v_{n},\ldots, v_{0})$ in $\G$ satisfying
   $e_{i}=\set{v_{i-1},v_{i}}$.       For any   function $c\colon E\to B$, let  $\wh c\colon U\to B$ be its   unique  extension to a
     $\bbF_{2}$-linear map (regarding $B$ as 
     possibly non-unital $\bbF_{2}$-algebra in the natural way).
        From the definitions,  $(V,E,c,B)$ is a rainbow graph if and only if $\wh c$ vanishes on $C$, or equivalently, if and only if it vanishes on $Z$  where  $Z$ is  some chosen basis of $C$.

    Choose  the basis $Z$  of $C$ according to  the following  well-known procedure involving $T$.    For each edge $f\in E\sm E'$, there is a  simple cycle   of the graph    $(V,E'\cup\set{f})$, by maximality of $T$. The set $\set{e_{f,1},\ldots, e_{f,n_{f}}}$ of edges of this cycle contains $f$,   say $f=e_{f_{1}}\in E\sm E'$ and $e_{f_{i}}\in E'$ for $i>1$ without loss of generality.
      Define $z_{f}:=e_{f,{1}} +\ldots +e_{f,n_{f}}\in C$.  Since $z_{f}+U'=f+U'\in U/U'\cong U''$ and the elements $f\in E\sm E'$ are a basis of $U''$, it follows that  the  elements $z_{f}$ for $f\in E\sm E'$ are pairwise distinct 
 and  the set $Z:=\mset{z_{f}\mid f\in E\sm E'}$ is a basis of  $C$ over   $\bbF_{2}$.     
     It is also well known and   easily shown (see for instance \cite{Dies}) that $Z$ spans $C$, so it is a basis of $C$. Hence $Z\cup E'$ is a basis of $U$.
    
   It follows immediately  that
     if $(V,E,c,B)$ is a rainbow graph,  then $c$ is uniquely determined by its restriction to a    function   $c'\colon E'\to B$ (since $\wh c$ vanishes on $Z$). Conversely,  an arbitrary function 
     $c'\colon E'\to B$     extends  uniquely    to a linear map 
     $c''\colon U\to B$ which is zero on $Z$. Let $c$ denote the restriction of $c''$  to a      function  $c\colon E\to B$.  Then  $(V,E,c,B)$ is a rainbow graph such that $c$ extends $c'$.     
      \end{proof}

   \subsection{} \label{ssa9.4} Let $\G:=(V,E)$ be a   simple graph.
  The associated protorootoid $\CJ_{\G}$ (and, if $\G$ is even,  $\ol{\CJ}_{\G}$) have as components  the associated  protorootoids
  $\CJ_{\G_{i}}$ (and, if $\G$ is even,  $\ol{\CJ}_{\G_{i}}$) for the components $\G_{i}$ of $\G$.  Assume henceforward in this subsection that $\G$ is connected. Rainbow graphs  which afford the associated protorootoid(s) are described below.
  
    Let $(G,S)$ be the $C_{0}$-system associated to $\G$, as defined in \ref{ssa9.1}.
  One has $\ob(G)=V$, $\mor(G)=\mset{(a,b)\mid a,b\in \ob(G)}$ and 
  $S=\mset{(a,b)\mid \set{a,b}\in E}$.
   Hence for $(a,b)\in S$   $l_{S}(a,b):=l_{S} ((a,b))$ is the minimum length of a path in $\G$ from $a$ to $b$. 
 
 For all $s=(a,b)\in S$, define 
 \begin{equation}\label{eqa9.4.1} X_{s}:=\mset{ x\in V\mid l_{S}(b,x)>l_{S}(a,x)}\seq V\end{equation}
Note that $a\in X_{s}$ and $b\not\in X_{s}$. Let $X:=\mset{X_{r}\mid r\in S}$ and $B:=\CP(X)$. Define an edge labelling $c\colon E\to B$ as follows:
  for an edge  $\set{a,b}\in E$,   $c(\set{a,b})$  is the set of all elements of $X$ which  contain exactly  one of the vertices of the edge i.e.
 \begin{equation}\label{eqa9.4.2} c(\set{a,b}):=\mset{A\in X\mid  \vert \set{a,b}\cap A\vert =1}.\end{equation} 
 
 Also,  if $\G$ is even, note that for $s=(a,b)\in S$,
  $Y_{s}:=\set{X_{s},X_{s^{*}}}$ is a partition of $V$ into two disjoint  non-empty sets.  Set $Y:=\mset{Y_{s}\mid 
 s\in S}$ and $B':=\CP(Y)$. Define an edge labelling $c'\colon E\to B'$ as follows:
  for an edge  $\set{a,b}\in E$,  
 $ c'(\set{a,b})$ is the subset of all elements $A=\set{A_{1},A_{2}}$ of 
 $Y$ such that $\vert A_{i}\cap \set{a,b}\vert=1$ for $i=1,2$.
 
  \begin{prop*} \begin{num}\item The quadruple   $R=(V,E,c,B)$  is a rainbow graph 
 whose associated protorootoid $\CJ_{R}$ is isomorphic to the protorootoid
 $\CJ_{\G}$.  \item If $\G$ is even,   $R'=(V,E,c',B')$  is a rainbow graph 
 whose associated protorootoid $\CJ_{R'}$ is isomorphic to the protorootoid
 $\ol{\CJ}_{\G}$. \end{num}
 \end{prop*}
  \begin{proof} We prove (a). In addition to the above notation, use notation concerning  $\CJ_{\G}=\CJ_{(G,S)}=(G,\Psi,N)$ as in Section \ref{sa8}.
    Suppose that $s=(a,b)\in S$ and $g=(e,a)\in \mor(G)$. By definition, 
 \begin{equation}\label{eqa9.4.3} G_{s}^{>}=\mset{(a,x)\in \lsub{a}{G}\mid l_{S}(b,x)>l_{S}(a,x)}=
 \mset{(a,x)\mid x\in X_{s}}\end{equation} and 
 therefore 
 \begin{equation}\label{eqa9.4.4} g(G_{s}^{>})=X_{e,s}:=\mset{(e,x)\mid  x\in X_{s}}.\end{equation}
 Hence $\lsub{e}{ \Psi}=\mset{X_{e,s}\mid s\in S} $ and 
 $ \lsub{e}{ \Psi}'=\mset{X_{e,s}\mid s\in S, e\in X_{s}}$. 
 So for $h=(e,d)\in \mor(G)$, $N(h):=\lsub{e}{ \Psi}'+h(\lsub{d}{ \Psi}'  )$ is given by
\begin{equation}\label{eqa9.4.5} N(h)= \mset{X_{e,s}\mid s\in S,\quad  \vert(\set{e}+\set{d})\cap  X_{s}\vert =1}\end{equation} Since $N$ is a cocycle, for any path
$(e=e_{n}, \ldots ,e_{0}= d)$ from $d$ to $e$, $N(e,d)$ consists of the elements $X_{e,s}$ such that $X_{s}\in M((e,d)):= c(e_{n},e_{n-1})+\ldots +c(e_{1},e_{0})$.

Let $\Th$ be the constant functor $G\to \setc$ with constant value 
$X$. The above implies that  $M$ is a cocycle for $\CP_{G}(\Th)$,   since $N$ is a cocycle. Therefore  $R$ is a  rainbow graph. 
  The set protorootoids underlying  $\CJ_{\G}$ and   $\CJ_{R}$ are 
 $\CJ'_{\G}=(G,\Psi,N)$ and   $\CJ'_{R}=(G,\Th,M)$.  It follows  readily from the above that there is an isomorphism $(\Id_{G},\nu)\colon \CJ'_{\G}\to \CJ'_{R}$ of set protorootoids where $\nu_{a}\colon \lsub{a}{\Th}\to \lsub{a}{\Psi}$ is the isomorphism  $X_{s}\mapsto X_{a,s}\colon X\to \lsub{a}{\Psi}$. Then $\mathfrak{I}((\Id_{G},\nu))\colon \CJ_{\G}\to \CJ_{R}$ is an isomorphism as required to prove  (a). 
The proof of (b) is similar from the above and the definitions, and it is omitted.
  \end{proof}
     
  \subsection{}\label{ssa9.5} Here are    examples illustrating the above constructions 
\begin{exmp*} (1)  Let $\G=(V,E)$ be  the even simple graph with vertex set $V=\set{p,q,r,s,t}$ indicated schematically  on the left below. 
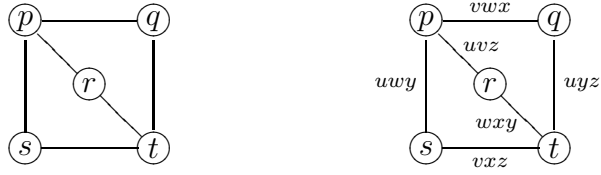
\begin{equation*} \xymatrix@!0{   &*+[o][F-]{p} \ar@{-}[dd]
 \ar@{-}[rr]\ar@{-}[dr]&   &*+[o][F-]{q}\ar@{-}[dd]&\\ 
   &   &*+[o][F-]{r}\ar@{-}[dr]&&\\
&*+[o][F-]{s}\ar@{-}[rr]&&
*+[o][F-]{t}&}\qquad \qquad
\xymatrix@!0{   &*+[o][F-]{p} \ar@{-}[dd]_{uwy}
 \ar@{-}[rr]^{vwx}\ar@{-}[dr]^>>>{\snts  \snts  \snts   uvz}&   
 &*+[o][F-]{q}\ar@{-}[dd]^{uyz}&\\ 
   &   
   &*+[o][F-]{r}\ar@{-}[dr]_<<<{wxy\snts  \snts  \snts  }&&\\
&*+[o][F-]{s}\ar@{-}[rr]_{vxz}&&
*+[o][F-]{t}&}\end{equation*} 
This gives rise to the rainbow graph $R'=(V,E,c',B')$ shown schematically on the right above where  $Y:=\set{u,v,w,x,y,z})$, $B':=\wp(Y)$ and 
the label ``$vwx$'' on the edge $\set{p,q}$ means that 
$c'(\set{p,q})=\set{v,w,x}$, etc.
In fact,   in the notation of \ref{ssa9.4},

\begin{align*} & X_{(p,q)} =\set{p,r,s}   & & X_{(q,p)} =\set{q,t}
 & & X_{(p,s)} =\set{p,r,q}   && X_{(s,p)} =\set{s,t}\\
& X_{(p,r)} =\set{p,s,q} &  & X_{(r,p)} =\set{r,t}
&& X_{(t,q)} =\set{t,r,s}   && X_{(q,t)} =\set{p,q}\\
& X_{(t,s)} =\set{t,r,q}   && X_{(s,t)} =\set{p,s}
&& X_{(t,r)} =\set{t,s,q}   &&H_{(r,t)} =\set{r,p}\\
\end{align*}

As color set, take  the  set $Y$ of partitions of $V$
where 
\begin{align*} &x:=\set{\set{p,r,s},\set{q,t}}
& &y:=\set{\set{p,r,q},\set{s,t}}
&&z:=\set{\set{p,s,q},\set{r,t}}\\
&u:=\set{\set{t,r,s},\set{p,q}}
&&v:=\set{\set{t,r,q},\set{p,s}}
&&w:=\set{\set{t,s,q},\set{r,p}}.
\end{align*}
This gives the rainbow graph $R'$ with edges  labelled by subsets of $Y$ as 
above  e.g. the partitions from $Y$  in which $p,q$ belong
 to complementary sets of the partition are precisely $v$, $w$ and $x$.
 Since   the color labels aren't all  singleton sets, $\G$ is  not  
 a $C_{1}$-graph.

(2)  Similarly,  the even graph $\G$ at the left below gives rise to the rainbow graph $R'$ at the right 
    
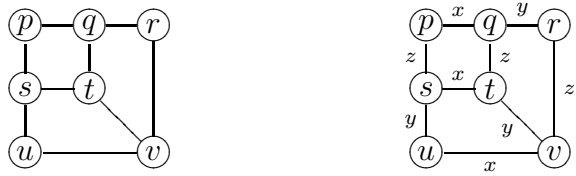
\begin{equation*} \xymatrix@!0{   &*+[o][F-]{p}
\ar@{-}[r]&
*+[o][F-]{q}\ar@{-}[r]
   &*+[o][F-]{r}\ar@{-}[dd]&\\ 
   &*+[o][F-]{s}\ar@{-}[u]\ar@{-}[d]\ar@{-}[r]
   &*+[o][F-]{t}\ar@{-}[u]\ar@{-}[dr]&&\\
&*+[o][F-]{u}\ar@{-}[rr]&&
*+[o][F-]{v}&}
 \qquad \qquad 
    \xymatrix@!0{   &*+[o][F-]{p}
\ar@{-}[r]^{x}
&*+[o][F-]{q}\ar@{-}[r]^{y}
   &*+[o][F-]{r}\ar@{-}[dd]^{z}&\\ 
   &*+[o][F-]{s}\ar@{-}[u]^{z}\ar@{-}[d]_{y}\ar@{-}[r]^{x}
   &*+[o][F-]{t}\ar@{-}[u]_{z}\ar@{-}[dr]_<<<{y\snts  \snts  }&&\\
&*+[o][F-]{u}\ar@{-}[rr]_{x}&&
*+[o][F-]{v}&}\end{equation*} 
Here,  $Y=\set{x,y,z}$ where 
\begin{align*}&x=\set{\set{p,s,u},\set{q,r,t,v}}&&y=\set{\set{p,q,s,t},\set{r,v,u}}&&
z=\set{\set{p,q,r},\set{s,t,u,v}}\end{align*}
Regarding $\G$ as a subgraph of the cube graph obtained by 
deleting one vertex and the incident edges, edges receive the same label if and only if  
they correspond to parallel edges of the cube. Write $\CR=\ol{\CJ}_{\G}=(G,\L,N)
$. For a morphism $(a,b)$ of $G$, $l_{S}(a,b)$ is the minimum length of a 
path from $a$ to $b$, whereas $l_{N}(a,b)$ is the number of elements of $X$ 
which appear an odd number of times as edge-label in some (equivalently, every) path from 
$a$ to $b$. One easily checks $l_{N}(a,b)=l_{S}(a,b)$ for all vertices $a,b$, 
so $G$ is an (even)  $C_{1}$-graph. 
 
  The Hasse diagram of  weak order at a vertex $a\in V=\set{p,q,r,s,t,u,v}$ is just the graph $\G$ directed by
edge-distance ($v\mapsto l_{S}(a,v)$) from the vertex $a$. There are three isomorphism types of weak orders, corresponding to the three orbits $\set{t}$,$\set{p,r,u}$, $\set{q,s,v}$, 
of the automorphism group of $\G$ on $V$. The corresponding Hasse diagrams are  
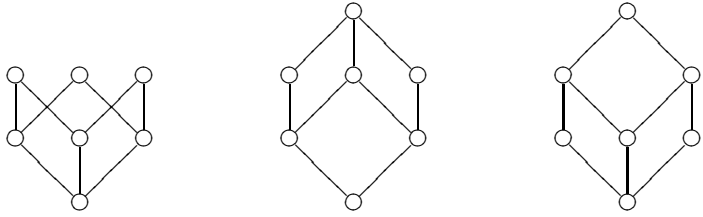
\begin{equation*} 
\xymatrix@!0{&&&\\
*+[o][F-]{}&*+[o][F-]{}
&*+[o][F-]{}
\\
*+[o][F-]{} \ar@{-}[u]\ar@{-}[ur]&
*+[o][F-]{} \ar@{-}[ul]\ar@{-}[ur]&
*+[o][F-]{} \ar@{-}[ul]\ar@{-}[u]\\
&*+[o][F-]{}\ar@{-}[ul]
\ar@{-}[u]
\ar@{-}[ur]&&
}\qquad \xymatrix@!0{&*+[o][F-]{}&&\\
*+[o][F-]{}\ar@{-}[ur]&
*+[o][F-]{}\ar@{-}[u]&
*+[o][F-]{}\ar@{-}[ul]\\
*+[o][F-]{}\ar@{-}[u]\ar@{-}[ur]&&
*+[o][F-]{}\ar@{-}[u]\ar@{-}[ul]\\
&*+[o][F-]{}\ar@{-}[ul]
\ar@{-}[ur]&&
}\qquad
\xymatrix@!0{&*+[o][F-]{}&&\\
*+[o][F-]{}\ar@{-}[ur]&&*+[o][F-]{}\ar@{-}[ul]\\
*+[o][F-]{}\ar@{-}[u]&
*+[o][F-]{}\ar@{-}[ul]\ar@{-}[ur]&*+[o][F-]{}
\ar@{-}[u]\\
&*+[o][F-]{}\ar@{-}[ul]
\ar@{-}[u]
\ar@{-}[ur]&&
}
\end{equation*}
respectively, so the weak orders are (complete) meet semilattices. However, $\CR$ is not a rootoid (and equivalently,
$\G$ is not a $C_{2}$-graph) since JOP (more precisely, its consequence \cite[Proposition 4.5(b)]{DyGrp1}) fails; in the third diagram, the middle atom is below the join of the leftmost atom  and the  rightmost atom. On the other hand, the edge graph of a cube (or any hypercube) is an even $C_{2}$-graph.

(3) Consider the even graph $\G$ (a $6$-cycle) shown at the left in the diagram below.
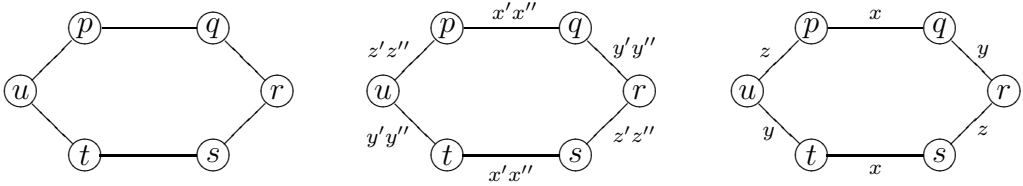
\begin{equation*}\xymatrix@!0{ 
&*+[o][F-]{p}\ar@{-}[rr]\ar@{-}[dl]&
&*+[o][F-]{q}\ar@{-}[dr]&\\
*+[o][F-]{u}&&&
&*+[o][F-]{r}\\
&*+[o][F-]{t}\ar@{-}[rr]\ar@{-}[ul]&
&*+[o][F-]{s}\ar@{-}[ur]&
}\qquad
\xymatrix@!0{ 
&*+[o][F-]{p}\ar@{-}[rr]^{\,\,x'x''}\ar@{-}[dl]_>>>{ z'z''\snts  \snts  \snts  }&
&*+[o][F-]{q}\ar@{-}[dr]^>>>{\snts  \snts  \snts   y'y''}&\\
*+[o][F-]{u}&&&
&*+[o][F-]{r}\\
&*+[o][F-]{t}\ar@{-}[rr] _{x'x''}\ar@{-}[ul]^>>>{y'y''\snts  \snts  \snts  }&
&*+[o][F-]{s}\ar@{-}[ur]_>>>{\snts  \snts   \snts   z'z''}&
}\qquad 
\xymatrix@!0{ 
&*+[o][F-]{p}\ar@{-}[rr]^{x}\ar@{-}[dl]_>>>{z\snts  \snts  }&
&*+[o][F-]{q}\ar@{-}[dr]^>>>{ \snts  \snts  \snts   y}&\\
*+[o][F-]{u}&&&&
*+[o][F-]{r}\\
&*+[o][F-]{t}\ar@{-}[rr] _{x}\ar@{-}[ul]^>>>{y\snts  \snts  \snts  }&
&*+[o][F-]{s}\ar@{-}[ur]_>>>{\snts  \snts  \snts   z}&
}
\end{equation*}

The corresponding rainbow graphs $R$ and $R'$ are shown to the right, 
using  as  colors for $R$ the set
 $X=\set{x',y',z',x'',y'',z''}$  of subsets  of $V:=\set{p,q,r,s,t,u}$
where 
\begin{align*} x'&=\set{q,r,s} &y'&=\set{r,s,t}
 & z'&=\set{s,t,u}\\ x''&=\set{t,u,p} & y''&=\set{u,p,q}
&z''&=\set{p,q,r}
\end{align*}
and for $R'$ the color set $X'=\set{x,y,z}$ where
\begin{align*} x&=\set{x',x''} & y&=\set{y',y''} & z&=\set{z',z''}\end{align*}
The Hasse diagram of weak order at any object (vertex of $\G$) is isomorphic to $\G$ directed by distance from that vertex:
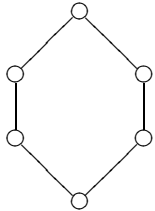
\begin{equation*}\xymatrix@!0{ &*+[o][F-]{}\ar@{-}[dr]\ar@{-}[dl]&\\
*+[o][F-]{}\ar@{-}[d]&
&*+[o][F-]{}\ar@{-}[d]\\
*+[o][F-]{}&&*+[o][F-]{}&\\
&*+[o][F-]{}\ar@{-}[ul]\ar@{-}[ur]&&
}\end{equation*}
It is easy to see that $\G$ is an even $C_{2}$-graph. The corresponding rootoids $\CJ_{\G}$, $\ol{\CJ}_{\G}$ are   complete,  connected and simply connected, and preprincipal,  and $\ol{\CJ}_{\G}$ is principal. 
The case of an arbitrary  cycle $\G$ of  even length $m\geq 4$  is analogous.

(4) Consider the graph $\G$ (a $5$-cycle) shown at left in the diagram below:
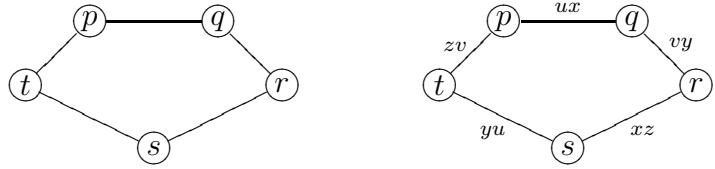
\begin{equation*}\xymatrix@!0{ 
&*+[o][F-]{p}\ar@{-}[rr]\ar@{-}[dl]&
&*+[o][F-]{q}\ar@{-}[dr]&\\
*+[o][F-]{t}&&&
&*+[o][F-]{r}\\
&&*+[o][F-]{s}\ar@{-}[ull]\ar@{-}[urr]&
}\qquad\qquad
\xymatrix@!0{ &*+[o][F-]{p}\ar@{-}[rr]^{ux}\ar@{-}[dl]_>>>{zv\snts  \snts  \snts  }&
&*+[o][F-]{q}\ar@{-}[dr]^>>>{\snts  \snts  \snts   vy}&\\
*+[o][F-]{t}&&&&*+[o][F-]{r}\\
&&*+[o][F-]{s}\ar@{-}[ull]^{yu}\ar@{-}[urr]_{xz}&
}
\end{equation*}
The corresponding rainbow graph $R$ is shown to the right
with color set $X=\set{x,y,z,v,w}$ where 
\begin{align*} x&=\set{q,r} &y&=\set{r,s}
 & z&=\set{s,t}& u&=\set{t,p} & v&=\set{p,q}
\end{align*}
Since $\G$ is not even, the  rainbow graph $R'$ is not defined.
The Hasse diagram of weak order at any object (vertex of $\G$) is 
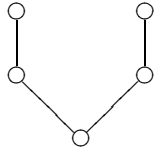
\begin{equation*}\xymatrix@!0{ 
*+[o][F-]{}\ar@{-}[d]&&*+[o][F-]{}\ar@{-}[d]\\
*+[o][F-]{}&&*+[o][F-]{}&\\
&*+[o][F-]{}\ar@{-}[ul]\ar@{-}[ur]&&
}\end{equation*}
It is easy to see that $\G$ is a $C_{2}$-graph.
The corresponding rootoid $\CJ_{\G}$ is    connected and
 simply connected, but not preprincipal and not complete. The case of an arbitrary cycle of odd length $m\geq 3$ is analogous.

 (5) Let  $\CT$  be the universal covering of  the rootoid $\CC_{(W,S)}$ of a  dihedral Coxeter system of order $2m$ ($m\geq 2$). As is easily checked, $\CT\cong \ol{\CJ}_{\G}$ where the simple graph  $\G$ is a cycle of length $2m$ if $m$ is finite and $\G$ is a doubly infinite path if $m=\infty$. From the definitions in \cite{CH} and \cite{HY}, it is easily seen  that $\CT$ has a Coxeter groupoid as underlying groupoid (with canonical generating set as Coxeter groupoid given by the simple generators of $\CT$ as principal rootoid). In turn, $\CT$  has a covering quotient isomorphic to the principal protorootoid  $\ol{\CJ}_{(G,X)}$ where $G$ is a cyclic group $G=\mpair{x}$ of order $2m$ and $X=\set{x,x^{-1}}$. However $G$ is not a Coxeter groupoid. 
  \end{exmp*}   

 \subsection{Protorootoid of an oriented matroid}\label{ssa9.6}
A (possibly infinite) oriented matroid in the sense of \cite{BF}  is a triple
$\mpair{E,*,c}$ where $E$ is a set, $x\to x^{*}\colon E\to E$ is a function and $c\colon \wp(E)\to \wp(E)$ is a closure operator, satisfying certain axioms which are not repeated here (in particular, $c$ is a convex closure operator in a suitable sense for signed sets).  For a  motivating example, see Remark (1) below.

Let $M=\mpair{E,*,c}$ be an oriented matroid in the sense of \cite{BF}.
Assume for simplicity that  the following extra conditions hold: $E\neq \eset$, $c(\eset)=\eset$, and  $x^{*}\neq x$ and $c(x)=\set{x}$ for $x\in E$.
 The axioms and extra conditions  imply that setting $-x:=x^{*}$  defines a free action of the sign group on $E$, making $E$ an indefinitely signed set.
A subset $H$ of $E$ is called a \emph{hemispace} if $H=c(H)=E\sm (-H)$.
It is shown in op. cit. that hemispaces exist.

Define a signed groupoid-set $(G,\Phi)$ as follows.
 Let $G$ be the connected, simply connected groupoid  with the set of 
 hemispaces as its set of objects.  For $H$, $K$ in $\ob(G)$, let
  $\lrsub{H}{f}{K}\in \lrsub{H}{G}{K}$ be the unique morphism $K\to H$ in $G$. Set $\lsub{H}{\Phi}:=E$ regarded as definitely signed set with the above action of $\set{\pm}$ and with the hemispace  $\lrsub{H}{\Phi}{+} :=H$ as its subset of positive elements.
 Define $\Phi\colon G\to \setc_{\pm}$ to be  a functor with 
 $\Phi(H):=\lsub{H}{\Phi}$ for $H\in \ob(G)$,  and such that  for any $x\in \lsub{K}{\Phi}=E$,  one has $(\Phi(\lrsub{H}{f}{K}))(x)=x$. Thus, the composite $\Phi\colon G\to \setc_{\pm}\to \setc$ of $\Phi$ with the forgetful functor $\setc_{\pm}\to \setc$ is the constant functor with  value $E$.
  
Note that  $\Phi_{\lrsub{H}{f}{K}}=H\cap K\seq H=\lrsub{H}{\Phi}{+}$.
  It is easily seen that this defines a faithful, complemented signed groupoid-set
  $J_{M}:=(G,\Phi)$.    Since $J_{M}$ is  complemented, if it is rootoidal then it is complete. The associated set protorootoid and protorootoid
are denoted as $\CJ'_{M}:=\mathfrak{L}(J_{M})$ and
$\CJ_{M}:=\mathfrak{I}(\mathfrak{L}(J_{M}))$

\begin{rem*} (1) Given a subset $E$ of a real vector space closed under multiplication by $-1$, define $*$ by $x^{*}=-x$, for $x\in E$, and $c$ by
$c(X):=\real_{\geq 0}X\cap E$ where $\real_{\geq 0}X$ is the cone of non-negative linear combinations of elements of $X$; then $M=(E,*,c)$ is an oriented matroid. The additional conditions hold if $E\neq \eset$,  $0\not\in E$ and $x,\l x\in E$ with $\l\in \real_{>0}$ implies $\l=1$.
The above construction  therefore attaches a signed groupoid-set $J_{M} $ to $(V,E)$ when these conditions all hold.

 (2)    Oriented geometries (see \cite{BjEZ}) are special oriented matroids $M$ with finite underlying sets $E$.  Examples include the oriented matroid   $M=(E,*,c)$ naturally associated (as in (1)) to the (finite)  set  $E$ of unit normals 
to the hyperplanes of a real, central hyperplane arrangement $\CH$ as in 
\cite[6.11]{DyGrp1}; oriented geometries of this type   are called realizable.  Using  facts established in \cite{BjEZ}, it can be shown   by an argument similar  to that in   \cite[6.11]{DyGrp1} that the signed groupoid-set $J_{M}$ attached to a (non-empty)  oriented geometry $M$   is rootoidal if and only if  the oriented geometry is simplicial. 

(3) If the non-empty oriented geometry $M$ is simplicial, then $J_{M}$  is  a finite, connected, simply connected, complete,  abridged, principal, signed groupoid-set. Also, any signed groupoid-set $(G,\Phi)$ with all these properties can be shown to   give  a 
Coxeter groupoid $G$ with root system $\Phi$  of a more 
    general type than considered in \cite{HY} and \cite{CH} (see \ref{ssa8.11}); further,   
    for $a\in \ob(G)$, $\lrsub{a}{\Phi}{+}$  admits a natural convex closure operator  in the sense of \cite{EdJ}. Note these are  closure operators for \emph{unsigned sets} and are not known to all be induced by a single oriented matroid closure operator. It  is an open question whether every  finite, connected, simply connected, complete,  abridged, principal, signed groupoid-set  $(G,\Phi)$ is isomorphic to $J_{M}$ for  some oriented simplicial geometry $M$ (i.e. whether $\Phi$ is realizable in oriented matroids).

  For non-realizable
simplicial oriented geometries $M$, it is expected  that in general  the  root system  of $J_{M}$
      is not   realizable in  real vector spaces in any 
      natural   way (even  more generally than  as in  \cite[6.6 or 6.11]{DyGrp1} or in \cite{HY}, \cite{CH}).  Finding an example to show this remains a non-trivial open problem, however, because  the most
       natural notions of realizability for  root systems 
      do not exactly correspond to that of realizability of  oriented matroids.      
 Detailed discussion and proof of the  claims in Remarks (2)--(3) is deferred to  subsequent papers. 
   
    (4)  In the setting of  \cite[Example 6.12(2)]{DyGrp1}, there is an  infinite oriented matroid $M=(E,*,c)$ naturally associated as in (1) to the unit sphere  $E=\mathbb{S}\seq V$. 
    The    hemispaces of $M$ are precisely the subsets of $\mathbb{S}$ which are the  positive elements of $\mathbb{S}$ with respect to some  vector space total ordering of $V$. It follows readily 
    that the  protorootoid $\CJ_{M}$ is isomorphic to the universal cover of    $\CJ_{O(V)}$  and hence  is a  rootoid by \cite{Pilk} and \cite[Lemma 4.9]{DyGrp1}.  The  class of possibly infinite oriented matroids $M$ for which $\CJ_{M}$  is    a (necessarily complete) rootoid  may therefore  be regarded as a non-vacuous   extension of the class of simplicial oriented geometries. Examples    of such $M$ satisfying stringent additional conditions  (e.g. such that $\CJ_{M}$ is  pseudoprincipal, saturated and regular, see \cite[Definition 3.3]{DyGrp1})  are known to exist,  and  should be particularly interesting. 
         \end{rem*}

        \subsection{Protorootoid of a protomesh}
        \label{ssa9.7}  
   Let $P=(R,L)$ be a protomesh (see \cite[2.8]{DyGrp1} for the definition).  Assume $L\neq \eset$ to avoid trivialities.  
         There is  an associated rainbow graph 
   $T=T_{P}=(V,E,c,R)$ as follows. 
   The graph $\G=(V,E)$ is the complete graph on vertex set $L$ i.e. 
   $V:=L$ and $E:=\mset{\set{A,B}\mid A,B\in L, A\neq B}$.
   The  labelling $c\colon E\to R$ is given by $c(\set{A,B})=A+B$ where $+$ denotes addition in $R$. Define the protorootoid $\CJ_{T}=(G,\L,N)$ of  the rainbow graph $T$.  

    Concretely, $\CJ_{T}=(G,\L,N)$ where $G$ is a connected, 
 simply connected groupoid with $\ob(G)=L$, $\L$ is the constant 
 functor $G\to \bringc$ with constant value $R$, and $N$ is the $G$-cocycle (in fact, coboundary) for $\L$ such that for $f\in \lrsub{A}{G}{B}$,
 where $A,B\in L$,  one has  $N(f)=A+B$. It is obvious from this description that $\CJ_{T}$ is faithful.

\begin{defn*} \begin{num}\item  The protorootoid of the protomesh $P=(R,L)$  is defined to be  $\CJ_{P}:=\CJ_{T_{P}}=(G,\L,N)$ as above. 
   \item  The protomesh $P$ is called  a \emph{mesh} if $\CJ_{P}$ is a  rootoid.
 \end{num}\end{defn*}  
 
We say that the protomesh $P=(R,L)$ is defined over $R$.  If  $R$ is  the Boolean ring of subsets of a set, one also defines the set protorootoid  $\CJ'_{P}:=\CJ'_{T_{P}}$ and signed groupoid-set $J_{M}:=J_{T_{P}}$ of  $P$.
 Terminology for rootoids and protorootoids may be  extended to meshes and protomeshes  as in \cite[5.6]{DyGrp1}, by saying that a protomesh $P$ has some property of rootoids if $\CJ_{P}$ has that property.
 For example, a mesh $P$ is said to be \emph{complete} (resp., \emph{principal}) if  $\CJ_{P}$ is a complete (resp., principal) rootoid.

  \subsection{Characterization of meshes}
  \label{ssa9.8}    
    The following is obvious from the definitions.
    \begin{prop*} \begin{num}\item    A  protomesh $(R,L)$ is a mesh  
if and only if  for all  $\G\in L$, $\G+L$ is a complete meet semilattice (in the 
order induced by the natural order of $R$) and the protomesh
 $(R,\G+L)$ satisfies the $JOP$.
 \item   A mesh $(R,L)$ is complete if for each
  $\G\in L$, $\G+L$ has a maximum element i.e. is a complete 
  lattice.\end{num}
\end{prop*}
    \begin{exmp*}  Let $R$ be a complete Boolean algebra.
    By \cite[Lemma 4.1]{DyGrp1}, the protomesh $(R,R)$ satisfies the JOP. Let $L$ be an ideal of $R$ as ring.  It is easy to see that the protomesh $(R,L)$ inherits the JOP from $(R,R)$.  Since $\G+L=L$ for all $\G\in L$, the proposition implies that   
         that $(R,L)$ is a mesh, and  so $\CT:=\CJ_{(R,L)}$ is a rootoid.
         
         It is trivial that $\CT$ is regular and  saturated, and that it is complete if $L=R$. It is also pseudoprincipal.  To prove this, it suffices to check that if $A,B\in L$ with $A\neq \eset$, there exists 
         $X\in L$ with $\eset\sneq X\seq A$ and either $X\seq B$ or $X\cap B=\eset$. Since $\eset\neq A=(A\cap B)\cup (A\cap B^{\cp})$, one may take $X$ to be whichever of $A\cap B$, $A\cap B^{\cp}$ is non-empty.
            \end{exmp*}

  \subsection{} \label{ssa9.9}
The definitions and \cite[Proposition 2.9]{DyGrp1}  immediately imply the following.
\begin{prop*} Let $\CR=(G,\L,N)$ be a faithful protorootoid with big weak order $\CL$. Then the following conditions are equivalent:
\begin{conds}\item $\CR$ is a (complete) rootoid.
\item For each $a\in \ob(G)$, $(\lsub{a}{\L},\lsub{a}{\CL})$ is a (complete) mesh.
\item For all $b\in \ob(G)$, there exists $a\in \ob(G[b])$  such that
$(\lsub{a}{\L},\lsub{a}{\CL})$ is a (complete) mesh.\end{conds}\end{prop*}

\subsection{} \label{ssa9.10}
This section concludes  with the following Table 1 which summarizes previously introduced  notation for objects of the  categories $\prootc$, $\setproot$ and $\grpdset$  attached to various special structures.   \begin{table}[htbp]
 \caption{Notation for  protorootoids, set protorootoids and signed groupoid sets
 attached to special structures.}
   \begin{center}
\begin{tabular}{| r | c ||  c | c | c |}\hline
\multicolumn{2}{ | c ||}{\vphantom{$\ol{S}$}Structure}&$\prootc$ &$\setproot$& $\grpdset$ \\ \hline\hline
\vphantom{$\underline{{\ol{{\CC}'_{(W,S)}}} }$}Coxeter system& $(W,S)$&$\CC_{(W,S)}$&$\CC'_{(W,S)} $ & $C_{(W,S)}$\\ \hline
\vphantom{$\underline{{{\ol{\CC_\CH'}}} }$}Real central  simplicial arrangement &$\CH$&$\CC_{\CH}$&$\CC_{\CH}'$& $C_{\CH}$\\ \hline
\vphantom{$\underline{{\ol{{\CC}'_{(O,V)}}} }$}Compact real orthogonal group &$O(V)$&$\CC_{O(V)}$&$\CC_{O(V)}'$& $C_{O(V)}$\\  \hline 
\vphantom{$\underline{{\ol{{\CJ}'_{(G,S)}}} }$}$C_{0}$-system
 & $(G,S)$&$\CJ_{(G,S)}$&$\CJ_{(G,S)}'$& $J_{(G,S)}$\\ 
 \hline \vphantom{$\underline{{\wh{{\CJ}'_{(G,S)}}} }$}Even $C_{0}$-system  &$(G,S)$&$\ol{\CJ}_{(G,S)}$&$\ol{\CJ}\,'_{(G,S)}$& $\ol{J}_{(G,S)}$\\ \hline
\vphantom{$\underline{{\ol{{\CJ}_{(G,S)}}} }$}$C_{0}$-graph& $\G$&$\CJ_{\G}$&$\CJ_{\G}'$& $J_{\G}$\\ \hline
\vphantom{$\underline{{\ol{\ol{\CJ}_{(G,S)}}} }$}Even $C_{0}$-graph &$\G$&$\ol{\CJ}_{\G}$&$\ol{\CJ}\,'_{\G}$& $\ol{J}_{\G}$\\ \hline
\vphantom{$\underline{{\ol{{\CJ}'_{X}}} }$}Rainbow graph& $X$&$\CJ_{X}$&($\CJ_{X}'$)& ($J_{X}$)\\ \hline
\vphantom{$\underline{{\ol{{\CJ}'_{P}}} }$}Protomesh& $P$ &$\CJ_{P}$&($\CJ_{P}'$)& ($J_{P}$)\\ \hline
\vphantom{$\underline{{\ol{{\CJ}_{(M}}} }$}Possibly infinite  oriented matroid& $M$&${\CJ}_{M}$&${\CJ}'_{M}$& ${J}_{M}$\\ \hline
\end{tabular}\end{center}
\end{table}

 Those objects  denoted by notation of the form  $\CC_{U}$, $\CC'_{U}$, $C_{U}$ are rootoidal for any structure $U$ of the indicated type, while those denoted $\CJ_{U}$, $\CJ'_{U}$, $J_{U}$,
$\ol{\CJ}_{U}$, $\ol{\CJ}\,'_{U}$ or  $\ol{J}_{U}$
  may or may not be rootoidal depending on the specific structure $U$.
  The abridgement of a  protorootoid  denoted $\ol{\CJ}_{U}$ is isomorphic to the abridgement of  $\CJ_{U}$.
  Other regularities in the notation should be apparent from the table. The bracketed entries for a rainbow graph or  protomesh are  only defined  if   it  is defined over a Boolean ring $\CP(Y)$ of sets.

 %
%  \begin{center}
%\begin{tabular}{| c || c | c | c |}\hline
%Structure&$\prootc$ &$\setproot$& $\grpdset$ \\ \hline
%Coxeter system $(W,S)$&$\CC_{(W,S)}$&$\CC'_{(W,S)} $ & $C_{(W,S)}$\\ \hline
%Real central  simplicial arrangement $\CH$&$\CC_{\CH}$&$\CC_{\CH}'$& $C_{\CH}$\\ \hline
%Orthogonal group $O(V)$&$\CC_{O(V)}$&$\CC_{O(V)}'$& $C_{O(V)}$\\ \hline
%$C_{0}$-system $(G,S)$&$\CJ_{(G,S)}$&$\CJ_{(G,S)}'$& $J_{(G,S)}$\\ \hline
%Even $C_{0}$-system $(G,S)$&$\ol{\CJ}_{(G,S)}$&$\ol{\CJ}\,'_{(G,S)}$& $\ol{J}_{(G,S)}$\\ \hline
%$C_{0}$-graph $\G$&$\CJ_{\G}$&$\CJ_{\G}'$& $J_{\G}$\\ \hline
%Even $C_{0}$-graph $\G$&$\ol{\CJ}_{\G}$&$\ol{\CJ}\,'_{\G}$& $\ol{J}_{\G}$\\ \hline
%Rainbow graph $X$&$\CJ_{X}$&($\CJ_{X}'$)& ($J_{X}$)\\ \hline
%Protomesh $P$ &$\CJ_{P}$&($\CJ_{P}'$)& ($J_{P}$)\\ \hline
%Possibly infinite  oriented matroid $M$&${\CJ}_{M}$&${\CJ}'_{M}$& ${J}_{M}$\\ \hline
%\end{tabular}\end{center}
%\vskip .8in

\section{Completion}\label{sa10}
\subsection{Conjectural complete rootoid of a Coxeter system} 
\label{ssa10.1}
      Suppose $(W,S)$ is a Coxeter system and 
   $\CC=\CC_{(W,S)}=(W,\L,N)$ is the associated rootoid.
    The $W$-set corresponding to the functor $\L$ is $\wp(T)$, 
    with $W$-action induced by the   action of $W$ on $T$ by conjugation. The  (unique)  
    weak order is  $\CL=\mset{N(w)\mid w\in W}$. One therefore 
    has an associated mesh $P:=(\wp(T),\CL)$. In turn, this 
    mesh has an associated rootoid $\CR:=\CJ_P$, 
    which is complete if and only if  $W$ is finite.  The rootoid $\CR$ is easily seen to be the universal covering rootoid of $\CC$.    
 Let    $\wh \CL\seq \wp (T)$ be  the set of all initial sections of 
 reflection orders of the reflections $T$ of $(W,S)$,  
    as defined in \cite{DyHS1}. Then from \cite{DyHS1},
    \begin{equation}\label{eqa10.1.1} \CL=\mset{\G\in \wh{\CL}\mid \vert 
    \G\vert \text{ \rm is finite}}.\end{equation} One has  $\wh{\CL}=\CL$ if and only if  $W$ 
    is finite.  Define the  protomesh $\wh P:=(\wp(T),{\wh{\CL}}\ )$ and the 
    corresponding protorootoid  $\wh \CR:=\CJ_{\wh P}$. The identity 
    map $i:=\Id_{\wp(T)} $ is a morphism of protomeshes $i\colon P\to \wh P$.
    
\begin{conj*}  The protorootoid  $\wh \CR$ is a   complete, 
regular, saturated,   pseudoprincipal 
rootoid. Equivalently,  $\wh P$ is a complete, regular, saturated, pseudoprincipal mesh.
\end{conj*}    
Except  for the statement that $\wh\CR$ is pseudoprincipal,  all parts of the conjecture    follow from the unproven conjectures   \cite[(2.2) and  (2.5)]{DyWeak}, which motivated many of the  concerns of  these papers.  The conjecture holds trivially for
 finite $(W,S)$ from the results in \cite{DyGrp1} since then $\wh \CR=\CR$, and it is not difficult to check   it holds   if $(W,S)$ is infinite dihedral. It is not known for any  other irreducible Coxeter systems. 
 Subsequent papers will show that  complete,  regular, saturated,   pseudoprincipal  rootoids have some  very favorable properties.
%     \begin{rem*}    It can be shown that for infinite $W$ (even of finite rank) $\wh \CR$ need  not be order isomorphic to  the protorootoid  of any oriented matroid. However, this  does not rule out, for instance,   the possibility that   $\wh \CR$ could be described in terms of  a   family of such protorootoids. 
%     \end{rem*}
     \subsection{Completion of rootoids} \label{ssa10.2}
         Let $\CR=(G,\L,N)$ be a rootoid. 
A \emph{completion} of $\CR$ is a complete rootoid $\CT=(H,\G,M)$ together with a morphism $(\a,\nu)\colon \CR\to \CT$ of rootoids such that the following conditions (i)--(iii) hold:
\begin{conds} \item  $\a=\a'\a''$  where $\a'\colon G'\to H$ is the inclusion morphism  of a subgroupoid  $G'$ of $H$ and $\a''\colon G\to G'$ is an isomorphism. 
\item   For all $a\in \ob(G)$, $\nu_{a}\colon \lsub{a}{\L}\to \lsub{\a(a)}{\G}$ is injective.
\item  For each $a\in \ob(G)$ the  inclusion map $\lsub{a}{\a}\colon \lsub{a}{G}\to \lsub{\a(a)}{H}$ induced by $\a$ is an embedding of the weak right order of $G$ at $a$ as an order ideal of the  weak right order of $H$ at $\a(a)$. 
\end{conds}
\begin{rem*} A (different) definition of completion, more similar to standard notions of completions of semilattices, would be  obtained by replacing ``order ideal'' by ``join-closed meet subsemilattice'' in (iii).\end{rem*}

       \subsection{Completions of meshes}
     \label{ssa10.3}  
    A variant  for meshes of the notion of completion seems particularly simple. Define a completion of a mesh
     $P=(\L,\CL)$ to be a morphism $i\colon P\to \wh P$ of meshes  satisfying the conditions (i)--(iii) below:       \begin{conds}\item $\wh P=(\wh\L,\wh \CL\ )$ is a complete mesh
     \item $i$ is injective as a function $i\colon \L\to \wh \L$. Let $\CL':=\mset{i(\G)\mid \G\in \CL}\seq \wh \CL$.
     \item For each $\G\in \CL' $, the subset $\G+\CL'$ of  $\wh \L$ is an
     order ideal of $\G+\wh \CL$ (both in the order induced by the Boolean ring $\wh \L$). \end{conds}
         It is easily seen that  a completion $i$ as above induces a completion  $\CJ_{P}\to \CJ_{\wh P}$ of the associated rootoids.

 \begin{exmp*} (1) Let $P=(\L,\CL)$ where $\L=\wp(\set{x,y,z})$ and $\CL:=\set{\eset, \set{x},\set{y},\set{z}}$. It is easy to see that $P$ is a mesh 
 which has no completion $i\colon P\to \wh P$  such that 
 the map of Boolean rings underlying $i$ is an isomorphism.
 However, let $\wh \L:=\wp(\set{x,y,z,w})$, \begin{equation*} \wh \CL:=\CL\cup\mset{\set{x,y,z,w}\sm \G\mid \G\in \CL}\end{equation*} and $i\colon \L\to \wh\L$ be the natural inclusion homomorphism of Boolean rings (which doesn't preserve identity elements, but isn't required to).   It is easily checked that  $\wh P:=(\wh L,\wh \CL)$ is a mesh and  $i\colon P\to \wh P$ is a completion of $M$.    
 
  (2)    It is not  even conjectural  whether, for general Coxeter systems $(W,S)$,  $\CC_{(W,S)}$ has a completion.  However,   Conjecture \ref{ssa10.1} and \eqref{eqa10.1.1}   imply that the morphism $i\colon P\to \wh P$ defined in \ref{ssa10.1} is a completion of $P$ (with the additional stringent property that the map of Boolean rings underlying $i$ is an isomorphism). In particular, the conjecture implies that  the universal covering  rootoid $\CJ_{ P}$ of $\CC_{(W,S)}$ has a completion.  
        \end{exmp*}

  \begin{rem*} (1) Conjecture \ref{ssa10.1} would imply  that for any  non-empty tree $\G$, $\ol{\CJ}_{\G}$ has a completion, by embedding $\G$ as a subtree of the Cayley graph of  a universal Coxeter system $(W,S)$   (i.e. one with no braid relations) of sufficiently high  rank $\vert S\vert$ and taking a completion of the  universal covering  rootoid $\CJ_{ P}$ of $\CC_{(W,S)}$. 
  Note that the rootoids  $\ol{\CJ}_{\G}$ for non-empty trees $\G$ are the ``least complete'' of connected, simply connected, principal rootoids, in the sense that the only non-empty subsets of their weak orders which have upper bounds are finite chains (totally ordered subsets).          These  observations suggest the question of whether   an arbitrary   principal  mesh $P$ has a   completion $\wh P$.
  
     (2)  If rootoid completions as defined above do no exist in useful generality, it may be an indication that a weaker or different notion of completion is more appropriate or  even that  the axioms for rootoids should be  strengthened (in some  unknown way). 

  (3)  Subsequent papers will make use  of  a construction  which will be called the pseudocompletion of a rootoid, the definition of which  is suggested by the above notions of completions. In general, the pseudocompletion of a rootoid is not a rootoid, but only a protorootoid
  with self-dual weak orders; however, it enables one to use for general rootoids  a weak version of a useful partial duality in the theory of  rootoids which admit a completion.   
   \end{rem*}

\subsection{}  \label{ssa10.4} By definition,  the weak orders of a rootoid $\CR$  with a completion $\CT$ embed as  order ideals of weak orders of $\CT$, which by \cite[Proposition 7.1]{DyGrp1}, are complete ortholattices. 
 In \ref{ssa10.8}, it will be shown  that the weak orders of any rootoid $\CR$ embed as  order ideals of complete ortholattices. The proof uses  constructions which are described in \ref{ssa10.5}--\ref{ssa10.7}. 
The first of these  glues two complete lattices related by a Galois connection to give a complete lattice, which is a complete ortholattice under a stringent symmetry condition.
 
Recall that a Galois connection between two posets $X$ and $Y$ is a pair $(\a,\b)$ of order reversing maps
$\a\colon X\to Y$ and $\b\colon Y\to X$ such that for all $x\in X$, $y\in Y$ one has $y\leq \a(x)$ if and only if $x\leq \b(y)$. 
This is equivalent to saying that, regarding $X$, $Y$ as categories  and $\a,\b$ as (covariant) functors between $X$ and $Y\op$ in the natural way,  $\a$ is left adjoint to $\b$.
The subset $X':=\mset{x\in X\mid \b\a(x)=x}$ is called the set of stable elements of $X$; similarly, one defines the set $Y'$ of  stable  elements of $Y$. 

The following well-known facts concerning such a Galois connection are recorded here  for convenience of reference
(additional properties holding by virtue of the symmetry interchanging $X$ and $Y$, and $\a$ and $\b$, are not listed).
\begin{lem*} Let $x,x'\in X$ and $y\in Y$.
\begin{num}\item $x\leq x'$ implies $\a(x')\leq \a(x)$.
\item $y\leq\a(x)$ implies  $x\leq \b(y)$.
\item $x\leq \b\a(x)$ and $\a(x)=\a\b\a(x)$.
\item $X'=\mset{\b(u)\mid u\in Y}$.
\item If $X$, $Y$ are complete  lattices and $x_{i}\in X$ for $i\in I$, then
$\a(\join_{i}x_{i})=\meet_{i}\a(x_{i})$.
\item   If $X$, $Y$ are complete  lattices, then
 $X'$  and $Y'$ are complete lattices in the induced order (with meet given by meet in $X$, $Y$ respectively) and $\a$, $\b$ restrict to order reversing bijections between $X'$ and $Y'$.
\end{num}\end{lem*}

\subsection{} \label{ssa10.5}Suppose that $\a\colon X\to Y$ is an order reversing map between posets $X$ and $Y$.
Choose a set  
$V= V_{0}\dotcup  V_{1}$, where   $V_{0}$ (resp., $V_{1}$) is in bijection with $X$ (resp., $Y$) by a map $i_{0}\colon X \to V_{0}$ (resp., $i_{1}\colon Y\to V_{1}$). Define a relation
$\leq$ on $V$ by the following conditions:   for all $x,x'\in X$ and $y,y'\in Y$,
\begin{conds}\item   $i_{0}(x)\leq i_{0}(x')$ if and only if $x\leq x'$ in $X$.
\item $i_{1}(y)\leq i_{1}(y')$ if and only if $y'\leq y$ in $Y$.
\item $i_{1}(y)\not \leq i_{0}(x)$.
\item $i_{0}(x')\leq i_{1}(y)$ if and only if $y\leq \a(x')$ in  $Y$.
\end{conds} 
One readily checks  that $\leq$ is a partial order on
$V$. Obviously,  $V_{0}$ is an order ideal of $V$ isomorphic to $X$
  and   $V_{1}$ is an order coideal of $V$ order 
  isomorphic to $Y\op$. 
 \begin{lem*} Suppose above  that $X$, $Y$ are complete lattices and the map $\a$ is part of a Galois connection  $(\a,\b)$ between $X$ and $Y$. Then
\begin{num}\item The poset $V$ is a complete lattice.
\item  Suppose further that $X=Y$, $\a=\b$ and $x\wedge \a(x)=0_{X}$ for all $x\in X$. Then $V$ becomes a complete ortholattice with orthocomplement defined by $(i_{0}(x))^{\cp}:=i_{1}(x)$ and $(i_{1}(x))^{\cp}:=i_{0}(x)$ for all $x\in X$ 
\end{num}
\end{lem*}
\begin{proof} For any non-empty subset of $V_{1}$, its join in $V_{1}$ is also its join in $V$, since $V_{1}$ is an order coideal of $V$. We claim that for any non-empty  subset $A$ of $V_{0}$, its join in $V_{0}$ is also its join in $V$. To see this, write $A=\mset{i_{0}(a)\mid a\in A'}$ where $A'\seq X$ has join $x:=\join A'$. It will suffice to show that for any $y$ in $Y$ with 
 $i_{0}(a)\leq  i_{1}(y)$ for all $a\in A'$, one has  
  $i_{0}(x)\leq  i_{1}(y)$. By definition,
  $i_{0}(a)\leq  i_{1}(y)$  if and only if $y\leq \a(a)$ in $Y$.
  This holds if and only if $a\leq \b(y)$. But $a\leq \b(y)$ for all $a\in A'$ if and only if $x=\join_{a\in A'}a\leq \b(y)$, which in turn holds if and only if $y\leq \a(x)$ i.e.  $i_{0}(x)\leq  i_{1}(y)$.
     Hence any non-empty subset 
  of $V_{0}$ (resp., $V_{1}$) has a join in $V$, lying in 
  $V_{0}$ (resp., $V_{1}$). 
  To show that $V$ is a complete lattice,
  it will therefore suffice to show that the join $i_{0}(x)\vee i_{1}(y)$ exists in $V$ for any $x\in X$ and $y\in Y$.  By  the definitions, for $z\in Y$, $i_{0}(x)\leq i_{1}(z)$ if and only if $z\leq \a(x)$ in $Y$. This   holds if and only if $i_{1}(\a(x))\leq i_{1}(z)$ in $V_{1}$. 
Writing $i_{1}(z)=w$,  this means that 
    $i_{1}(\a(x))=\min(\set{w\in V_{1}\mid i_{0}(x)\leq w})$.  Since any upper bound  of $\set{i_{0}(x), i_{1}(y)}$ in $V$ must be in $V_{1}$, it follows that in $V$, 
  \begin{equation*}i_{0}(x)\vee i_{1}(y) =i_{1}(\a(x))\vee i_{1}(y)=i_{1}(\a(x)\wedge y).\end{equation*} This proves (a).
  
  Next,  make the assumptions as in (b) and define $z^{\cp}$ for $z\in V$ as there. The definition of $V$ and \ref{ssa10.4}(a)--(b) readily imply that $z\mapsto z^{\cp} $ is order-reversing and 
  $(z^{\cp})^{\cp}=z$ for $z\in V$. 
  Further, $z\vee z^{\cp}=1_{V}$ since for all $x\in X$,
  \begin{equation*}
  i_{0}(x)\vee i_{1}(x)=i_{1}(x\wedge \a(x))=i_{1}(0_{X})=1_{V}
  \end{equation*}
   from above. This implies that $z\wedge z^{\cp}=(z^{\cp}\vee z)^{\cp}=0_{V}$. This completes the proof. \end{proof}
 \subsection{}\label{ssa10.6}  The above   gluing construction will be applied with $X=Y= L'$ where $ L'$  is the  following lattice completion of  a complete meet semilattice $L$.
   Let $L'$ be the set of all non-empty join-closed order ideals of $L$, and partially order $L'$ by inclusion. Let $i'\colon L\to L'$ map each element of $L$ to the principal order ideal of $L$ it generates i.e.  $i'(y):=\mset{x\in L\mid x\leq y}$ for all $y\in L$. 
\begin{lem*} Ordered by inclusion,  $L'$ is a complete lattice and $i'$ is an order isomorphism of $L$ with an order ideal of $L'$. \end{lem*} 
 \begin{proof} To avoid trivialities, assume $L\neq \eset$. By definition, $L'$ consists of the  order ideals $I\neq \eset$ of $L$ such that if a subset
$X$ of $I$ has a join $\join X$ in $L$,  then $\join X\in I$. The intersection of any family of elements  of $L'$ is an element of $L'$. It follows that $L'$, ordered by inclusion, is a complete lattice with minimum element $\set{0_{L}}$ and maximum element $L$. 
 This map $i'$ is clearly  an order-isomorphism onto its image $i'(L)$, regarded as subposet of $L'$. Moreover, $i'(L)$ is an order ideal of  $L'$. For suppose that $x\in L$ and $y\in L'$ with $y\seq i'(x)$. Then as a subset of $L$, $y$ is bounded above by $x$, so $u:=\join y$ exists in $L$, $u\in y$ since $y$ is join-closed, and necessarily, $y=\mset{z\in L\mid z\seq u}=i'(u)$ since $y$ is an order ideal of $L$.
\end{proof}

\subsection{}\label{ssa10.7} Let $L$ be a non-empty complete meet semilattice.
 For the purposes of this subsection,   a  relation $P$ on  $L$ (i.e.  a subset $P$ of $L\times L$) is called  an \emph{orthogonality relation} on $L$  if it satisfies the 
  conditions (i)--(iii) below: \begin{conds}
\item As a relation on $L$, $P$ is symmetric  i.e. $(x,y)\in P$ implies $(y,x)\in P$.
\item For any $x\in L$, $\mset{y\in L\mid (x,y)\in P}$ is a non-empty  join-closed order ideal of $L$.
\item If $a\in L$ with $(a,a)\in P$, then $a=0_{L}$.  \end{conds}

For example, if $L$ is a non-empty order ideal of a complete ortholattice with orthocomplementation $z\mapsto z^{\cp}$, then the relation $P$ on $L$ defined by $(x,y)\in P$ if $x,y\in L$ and   $x\leq y^{\cp}$ is an orthogonality relation.

\begin{thm*} Let $L$ be a non-empty complete meet semilattice with an orthogonality relation $P$ as above. Then there is a canonically associated  order isomorphism of $L$ with an order ideal of a  complete 
ortholattice $V$.\end{thm*}
\begin{proof} Let $\th\colon \wp(L)\to \wp(L)$ be the function
defined by setting, for all $z\seq L$,  
\begin{equation*}
\th(z)=z^{\dag}:=\mset{b\in L\mid (a,b)\in P\text{ \rm for all $a\in z$}}.
\end{equation*}
 The definitions immediately imply that $(\th,\th)$ gives a Galois connection from $\wp(L)$ to itself. Note that  (ii) states that $\set{x}^{\dag}\in L'$ for all $x\in L$, where $L'$ is as in Lemma \ref{ssa10.6}. By Lemma \ref{ssa10.4},  for any $z\seq L$,
$z^{\dag}=(\cup_{x\in Z}\set{x})^{\dag}=\cap_{x\in x}\set{x}^{\dag}\in L'$, since meet in  $L'$ is given by intersection. This proves
that $\th$ restricts to a map $\th'\colon L'\to L'$.
The definitions readily imply that $(\th',\th')$ define a Galois connection from $L'$ to itself. We assert that for $x\in L'$,
$x\wedge \th'(x)=0_{L'}=\set{0_{L}}$.  Let $a\in x\wedge \th'(x)=x\cap \th'(x)$. Then $a\in x$ and $a\in \th'(x)$, so $(a,a)\in P$.
By (iii), $a=0_{L}$. Since clearly $0_{L}\in x\wedge \th'(x)$, the assertion is proved.

Now apply Lemma \ref{ssa10.5} with $X=Y= L'$ (which is a complete meet semilattice by Lemma \ref{ssa10.6}) and $\a=\b=\th'$.
It gives  a complete ortholattice $V$ with a canonical embedding of $L'$ as an order ideal of  $V$. Since  $L$ canonically embeds as an order ideal of $L'$,  the theorem follows.
\end{proof}
\subsection{} \label{ssa10.8} Finally, we may prove the following, in  
which  (c) is just one of  many possible illustrative applications of (b).
 \begin{cor*}\begin{num}\item  Let $(R,L)$ be a standard protomesh such that $L$ is a 
complete meet semilattice and $(R,L)$ satisfies the JOP. Then there
is a complete ortholattice $\wh L$ and an order-isomorphism
$i$ of $L$ with an order ideal of $\wh L$.
\item Any weak right order of a rootoid $\CR$ is order isomorphic to an order ideal of a complete ortholattice.
\item The weak right order of a Coxeter system $(W,S)$ is order isomorphic to an order ideal of a complete ortholattice.
\end{num}\end{cor*} 
\begin{rem*} The conclusion of (b) is much weaker than 
that which would follow from existence of a completion $\wh \CR$ of $\CR$, which would give \emph{compatible} embeddings of the weak orders  of $\CR$ in complete ortholattices, and  additional complete ortholattices  arising as weak orders of  objects of $\wh{\CR}$  which don't come from objects of $\CR$.  \end{rem*} 
 \begin{proof}  Take \begin{equation*}P:=\mset{(x,y)\in L\times L\mid x\cap y=\eset}.\end{equation*} in
Theorem \ref{ssa10.7}.  The condition \ref{ssa10.7}(i) on $P$ is trivial. 
In  \ref{ssa10.7}(ii), the indicated set is an order ideal because for any $x,y,z\in R$ with $z\seq y$,  $x\cap y=\eset$ implies  $x\cap z=\eset$, and it is join closed by the JOP. 
 Further,   if $(x,y)\in P$, then $x\wedge y\seq x\cap y =\eset=0_{L}$ in $L$. Taking $x=y$ shows that  \ref{ssa10.7}(iii) holds.
 This proves the hypotheses of Theorem \ref{ssa10.7}, and (a) follows.
 Part (b) follows by applying (a) to the protomeshes $(\lsub{a}{\L},\lsub{a}{\CL})$ of a rootoid  $\CR=(G,\L,N)$ with big weak order $\CL$, and (c) follows from (b) by taking $\CR:=\CC_{(W,S)} $.
 \end{proof}

 \section{Protorootoids and groupoid-preorders} \label{sa11}
    \subsection{Unital completion}\label{ssa11.1}
This subsection and the next recall some additional background about Boolean rings which will be used in several places subsequently in these papers. 

   The inclusion functor $I\colon \balgc\to \bringc $ has a left adjoint $U\colon \bringc \to \balgc$, which sends each  Boolean ring  to its \emph{unital completion} $U(B)$. Concretely, for a Boolean ring $B$,  $U( B):=B\oplus \bbF_{2}$ as abelian group with multiplication
\begin{equation*}
(b,\e)(b',\e')=(bb'+\e b'+\e' b,\e\e')
\end{equation*} (using that $B$ is naturally a possibly non-unital $\bbF_{2}$-algebra).  For any morphism $f\colon B\to C$ in $\bringc $, $f':=U(f)\colon U(B)\to U(C)$ is given by $f'(b,\e)=(f(b),\e)$. 
Identify $B$ as an ideal of $  U(B)$ via $b\mapsto (b,0)$ and let $1_{U(B)}:=(0_{B},1)$. 

Every element of $  U(B)$ is either in $B$, or of the form $b^{\cp}=1_{U(B)}+b$ for unique $b\in B$. 
This implies that $B$ is an order ideal of $  U(B)$ (in the natural order
of $U(B)$) and $B^{\cp}:=\mset{b^{\cp}\mid b\in B }$ is an order coideal of $  U(B)$.
\subsection{Free Boolean rings and algebras}\label{ssa11.2}
The forgetful functor $F\colon \bringc\to \setc$  has  a left adjoint
$G\colon \setc\to \bringc$ sending  each set $X$ to the free Boolean ring $G(X)$ on $X$.  This gives rise to a diagram  \begin{equation*}
       \xymatrix{
       {\balgc}\ar@<.6ex>[r]^{I}
      &{\bringc}\ar@<.6ex>[r]^-{F}
      \ar@<.6ex>[l]^{U}&{\setc}
      \ar@<.6ex>[l]^-{G}        	   }    
       \end{equation*} 
in which each functor to the left is left adjoint to the functor to the right which is above it in the diagram.

It is convenient to describe $G(X)$ as a subring of  its unital completion $U(G(X))$, which is  the free Boolean algebra $U(G(X))$ on $X$. Here, $U(G(X))$  naturally identifies with the quotient of the (commutative) polynomial ring $\bbF_{2}[X]$ on $X$ by the ideal $J$ generated by all elements $x^{2}+x$ for $x\in X$. The natural map $X\to U(G(X))$ given by $x\mapsto x+J$ is injective, and we use it to    identify $X$ with a subset of $U(G(X))$.  Then 
$U(G(X))$ has a $\bbF_{2}$-basis consisting of the monomials
$x_{1}\cdots x_{n}$ for $n\in \Nat$ and pairwise distinct $x_{i}\in X$. Two such  monomials are equal only  if they contain the same factors up to order, and the empty monomial is the identity element $1$ of $U(G(X))$. The ring $G(X)$ naturally identifies with the subring (not containing $1$) of $U(G(X))$  generated by $X$. Note that the monomials as above with $n>0$ form a $\bbF_{2}$-basis of $G(X)$.  An arbitrary function $X\to B$ where $B$ is a Boolean ring (resp., Boolean algebra) has a unique extension to a ring homomorphism (resp., unital ring homomorphism) $G(X)\to B$ (resp.,  $U(G(X))\to B$). It is easy to see that the map $G$ just defined on objects extends naturally to a functor $\setc \to \bringc$;
 the left adjointness of $G$ to $F$
(and of $UG$ to $FI$) follows using  these universal properties.

Note that if $X$ is finite, of cardinality $n$, then $G(X)$ (resp., $U(G(X))$)
is a finite Boolean ring of cardinality $2^{2^{n}}$ (resp., $2^{2^{n}-1}$)
with atoms equal to the products \begin{equation}
e_{Y}:= \Bigl(\prod_{y\in Y}y\Bigr) \Bigl(\prod_{y\in X\sm Y}(1+y)\Bigr)=\sum_{Z\colon Y\seq Z\seq X}\Bigl(\prod_{z\in Z}z\Bigr)
\end{equation}
 for $Y\seq X$ (resp., $\eset\sneq Y\seq X$). In particular, $G(X)$ has an identity element   \begin{equation} 1_{G(X)}=\sum_{Y\colon \eset\sneq Y\seq X} e_{Y}=1+e_{\eset}=\sum_{Y\colon\eset\sneq Y\seq X} \Bigl(\prod_{y\in Y}y\Bigr)\end{equation} if $X$ is finite, since $1=\sum_{Y\colon Y\seq X}e_{Y}$.
\begin{rem*} The Boolean ring  $\mpair{X\mid R}$  generated by generators $X$  subject to relations $R$ is by definition the quotient  $B/I$
of the free Boolean ring $B=G(X)$ on the set $X$, as constructed above, by the ideal $I$ of $B$ generated by the subset $R$ of $B$.
  \end{rem*}
 
          \subsection{} \label{ssa11.3} 
Recall from
 \cite[2.15 and 5.8]{DyGrp1}  
 the definition of the  functors  $\mathfrak{P}\colon \prootc\to \gpdpreord$  and $\mathfrak{P}'\colon \prootc^{a}\to \gpdpreord$. 
  \begin{thm*}  The  functors   $\mathfrak{P}$ and $\mathfrak{P}'$  
 have  left adjoints   $\mathfrak{Q}\colon  \gpdpreord \to \prootc$
and  $\mathfrak{Q}' \colon  \gpdpreord \to \prootc^{a}$ respectively.   \end{thm*} 
 \begin{proof}
 Let $(G,\leq)$ be an object of $\gpdpreord$.  Attach to $(G,\leq)$ 
 the following protorootoid $\CR:=(G,\L,N)$.
  The Boolean ring $\lsub{a}{\L}$ for $a\in \ob(G)$  is defined to be  the Boolean ring generated by a set $X_{a}$
 of indeterminates  subject to  relations $R_{a}$, both specified below. 
 One takes $X_{a}:=\mset{(a,x,y)\mid x\in \lrsub{a}{G}{b}, y\in \lsub{b}{G} }$  where one regards the indeterminate $(a,x,y)$ as corresponding to  $(\L(x))(N(y))$. The relations $R_{a}$ are of the following types: \begin{conds}\item for $x\in \lrsub{a}{G}{b}$, $y\in \lrsub{b}{G}{c}$, $z\in \lrsub{c}{G}{d}$,  
 $(a,x,yz)+(a,x,y)+(a,xy,z)=0$.
 \item for $x\in \lrsub{a}{G}{b}$, $y,z\in \lsub{b}{G}$ with $y\leq z$,
 $(a,x,y)(a,x,z)+(a,x,y)=0$. 
\end{conds}
In this, (i) (resp., (ii)) is chosen to correspond to the result of acting by $\L(x)$ on the  cocycle condition $N(yz)=N(y)+\L(y)N(z)$ (resp.,  the order relation $N(y)\seq N(z)$ i.e. 
$N(y)N(z)=N(y)$).
 Formally, the  Boolean ring $\lsub{a}{\L}$ defined above   is the evident quotient $\lsub{a}{\L}:=B(X_{a})/I_{a}$  of  the free Boolean ring $B(X_{a})$ generated by   $X_{a}$ by the ideal  $I_{a}$  of $B(X_{a})$ generated by the left hand sides of the relations of type (i)--(ii) above.   
 The image in the quotient $\lsub{a}{\L}$ of an element $(a,x,y)\in X_{a}\seq B(X_{a})$ will  still be  written  as $(a,x,y)$.

There is evidently a representation $\L$ of $G$ determined  by $\L(a):=\lsub{a}{\L}$
for $a\in \ob(G)$ and $\L(g)(a,x,y)=(b,gx,y)$ for $g\in \lrsub{b}{G}{a}$. 
There is a cocycle $N\in Z^{1}(G,\L)$ given by $N(g)=(a,1_{a},g)\in \L(a)$ for $g\in \lsub{a}{G}$. This defines the protorootoid $\CR=(G,\L,N)$.
Set $\mathfrak{Q}(G,\leq):=\CR$. It is straightforward to check that $\Id_{G}$ defines a morphism $I_{(G,\leq)}\colon (G,\leq)\to \mathfrak{P}\mathfrak{Q}(G,\leq)$ in $\gpdpreord$. 

By a standard criterion for existence of left adjoints, 
there is a unique extension of  the map $\mathfrak{Q}$ on objects  to a functor $\mathfrak{Q}$ left adjoint to $\mathfrak{P}$, with $I_{(G,\leq)}$ as component at $(G,\leq)$ of its unit,   provided that $I_{(G,\leq)}$
 is universal from $(G,\leq)$ to $\mathfrak{P}$ i.e. provided that
for any protorootoid $\CR=(H,\G,M)$ and for every morphism
$ \th\colon (G,\leq)\to \mathfrak{P}(\CR)$ in $\gpdpreord$  (specified by the groupoid homomorphism $\th\colon G\to H$), there is a unique morphism
$f=(\a,\nu)\colon (G,\L,N)\to (H,\G,M)$ in $\prootc$  such that $\th=\mathfrak{P}(f) I_{(G,\leq)}$.
It is easy to check that there is a unique such $f$ as required, given by
$\a=\th\colon G\to H$ and $\nu_{a}((a,x,y))=\bigl(\G(\th(x))\bigr)(M_{\th(y)})$.

The image of $\mathfrak{Q}$ is contained in $\prootc^{a}$, so  $\mathfrak{Q}$ factorizes uniquely as $\mathfrak{Q}=\mathfrak{B}\mathfrak{Q}'$ as in 
\cite[5.8]{DyGrp1}, with $\mathfrak{Q}'$ clearly left adjoint to $\mathfrak{P}'=\mathfrak{P}\mathfrak{B}$.
 \end{proof}

 \subsection{}   \label{ssa11.4}
      Let  $(G,\leq)$ be a groupoid-preorder
in which $G$ is (empty or) a connected, simply connected groupoid. Assume that the big weak right  preorder  is an anti-chain i.e. none of the elements of any of the weak right preorders of $G$ are   comparable.
The proposition below provides  a  concrete description,  which will be useful in a subsequent paper,  of a protorootoid $\CR$
  which is isomorphic to  $\mathfrak{Q}(G,\leq)$.

Define  $\CR:=(G,\L',N')$  as follows. Let $V:=\ob(G)$ and for $x,y\in V$, denote the unique morphism in $\lrsub{x}{G}{y}$ as $(x,y)$. 
 Let $\wh B$ be the   Boolean ring freely generated by $V$,
  and let $B$ denote the (possibly non-unital) subring of $\wh B$ generated by all the elements $x+y$ of $\wh B$ where $x,y\in V\seq \wh B$. Let $\L'\colon  G\to \bringc$ be the constant functor with value $B$, and let $N\in Z^{1}(G,\L')$ be the cocycle determined by $N'((x,y))=x+y\in \lsub{x}{\L}'=B$ for all $(x,y)\in \mor(G)$.
  
  \begin{prop*}  For $(G,\leq)$ and $\CR$ as  above, there is an isomorphism $\mathfrak{Q}(G,\leq)\xrightarrow{\cong}\CR$ in $\prootc$.\end{prop*}
  \begin{proof}
   Use notation for   $\mathfrak{Q}(G,\leq)=(G,\L,N)$ as in the proof of Theorem \ref{ssa11.3}.
   It will be shown   that there is an isomorphism of protorootoids 
    $(\Id_{G},\nu)\colon \mathfrak{Q}(G,\leq)\xrightarrow{  \cong} \CR$ such that  for
    $a,b,c\in \ob(G)$, $g:=(a,b)\in \lrsub{a}{G}{b}$ and $h:=(b,c)\in \lrsub{b}{G}{c}$, one has  \begin{equation}\label{eqa11.4.1}\nu_{a}((a,g,h))=\nu_{a}(a,(a,b),(b,c)):=b+c\in \lsub{a}{\L}'=B.\end{equation} 
  
To prove this, note first that for fixed $a\in \ob(G)$, there is a bijective correspondence $(a,(a,b),(b,c))\mapsto (b,c)$, for $b,c\in \ob(G)$,  from the set of generators of $\L(a)$ to  $\mor(G)$. So $\L(a)$ is naturally isomorphic to the Boolean ring $B$ generated by $\mor(G)$ subject to relations $(b,d)+(b,c)+(c,d)$ for $b,c,d\in V$ (these correspond to the  relations of type \ref{ssa11.3}(i) for $\L(a)$, and by the assumption on $(G,\leq)$, there are no other relations). These relations immediately imply that
$(b,b)=0$ (take $c=d=b$) and then $(b,c)=(c,b)$ (take
$d=b$). Let $\G:=(V,E)$ be the complete graph on $V$, with vertex set
$ E:=\mset{\set{b,c}\mid b,c\in V, b\neq c}$. The above implies that  
$\L(a)$ identifies with the  Boolean ring $B'$  generated by $E$ subject to relations $e_{1}+e_{2}+e_{3}=0$ for the edges $e_{1},e_{2},e_{3}$ of any triangle in $\G$, by an isomorphism $\a\colon \L(a)\to B'$ with
$\a(a,(a,b),(b,c))= \set{b,c}\in E\seq B'$ if $b\neq c$ and $\a(a,(a,b),(b,c))= 0$ otherwise. 

One can readily see 
(compare Proposition 
      \ref{ssa9.3})  that for any fixed $a\in V$,  $B'$ is the Boolean ring  freely generated by     the set $\mset{e\in E \mid a\in e}$ of  edges of $\G$ which have  $a$ as an endpoint.  Similarly, for fixed $a$,  $B$ identifies naturally with  the  Boolean ring  freely generated by its subset  $\set{a+b\mid b\in V, b\neq a}$. Hence there is an isomorphism $\b\colon B'\to B$ such that
      $\set{a,b}\mapsto a+b$ for all $b\in V\sm\set{a}$. By the relations for $B'$, $\b$ satisfies $\b(\set{b,c})= b+c$ for all $\set{b,c}\in E$. Clearly $\nu_{a}:=\b\a$ is an isomorphism of Boolean rings $\L(a)\to \L'(a)=B$ satisfying \eqref{eqa11.4.1}.
      It is now trivial to check that $(\Id_{G},\nu_{a})$ is an isomorphism of protorootoids as required. 
               \end{proof}

 \subsection{} \label{ssa11.5}  Let $\eta\colon \Id\to \mathfrak{P}\mathfrak{Q}$ and $\e\colon \mathfrak{Q}\mathfrak{P}\to \Id$ denote the unit and counit respectively  of the  adjoint pair $(\mathfrak{Q},\mathfrak{P})$.
 Recall the definition of the  category   $\gpdpreord_{P}$ of protorootoidal groupoid-preorders  from \cite[2.15]{DyGrp1}. Also, denote   the full subcategory of $\prootc$  with objects $\CR$ such that $\e_{\CR}$ is an isomorphism as $\prootc_{F}$. Part (a) of the following corollary provides a  characterization of protorootoidal groupoid-preorders which will be made more explicit in a subsequent paper.
 \begin{cor*}
 \begin{num} \item An object $a$ of $\gpdpreord$ is in $\gpdpreord_{P}$ if and only $\eta_{a}$ is an isomorphism in $\gpdpreord$.
 \item $\mathfrak{Q}$ and $\mathfrak{P}$ restrict to an adjoint equivalence between  $\gpdpreord_{P}$ and $\prootc_{F}$.
   \item A protorootoid is in $\prootc_{F}$ if and only of it is isomorphic to one of the form  $\mathfrak{Q}(a)$ for some object $a$ of $\gpdpreord_{P}$. 
   \end{num}\end{cor*}
   \begin{proof}   To prove (a), let $a$ be an object of $\gpdpreord$. Suppose first   that $\eta_{a}\colon a\to \mathfrak{P}\mathfrak{Q}(a)$ is an isomorphism.
 Then $a\cong \mathfrak{P}(b)$ where $b:=\mathfrak{Q}(a)$, so $a$ is an object of $\gpdpreord_{P}$. 
 On the other hand, suppose that $ a= \mathfrak{P}(b)$  where $b$ is a protorootoid.
 By one of the  triangular identities for the adjunction, the composite
 \begin{equation}\label{eqa11.5.1}
 \xymatrix{
  {\mathfrak{P}(b)}\ar[r]^-{\eta_{\mathfrak{P}(b)}}&{\mathfrak{P}\mathfrak{Q}\mathfrak{P}(b)}\ar[r]^-{\mathfrak{P}(\e_{b})}& {\mathfrak{P}(b)} }
  \end{equation}  is the identity. Writing $\mathfrak{P}(b)=(G,\leq)$ and $\mathfrak{P}\mathfrak{Q}\mathfrak{P}(b)=(G,\leq')$, this  means that the composite $(G,\leq )\to (G,\leq ')\to (G,\leq)$ is $\Id_{(G,\leq)}$, which 
  as a morphism of groupoids is $\Id_{G}$. From the proof of Theorem \ref{ssa11.3}, $\eta_{\mathfrak{P}(b)}=I_{(G,\leq)}$  has $\Id_{G}$ as underlying groupoid morphism. Hence $\mathfrak{P}(\e_{b})$ also has 
  $\Id_{G}$ as underlying groupoid morphism. This shows that $\Id_{G}$ induces a preorder isomorphism between  $(G,\leq)$ and $(G,\leq')$, so
the preorders   $\leq$ and $\leq'$ on $G$ coincide. Hence in \eqref{eqa11.5.1}, the two maps are inverse isomorphisms, and 
$\eta_{a}\colon a\xrightarrow{\cong}
{\mathfrak{P}\mathfrak{Q}(a)}$ is an isomorphism.  This proves (a).

      If $F\colon A\to B$ and $G\colon B\to A$ are  functors with $F$ left adjoint to $G$, with unit $\eta$ and counit $\e$,  then, using the triangular identities for the adjunction,  $F$ and $G$ restrict to an adjoint equivalence between the full subcategory of $A$
 with objects $a$   such  that $\eta_{a}$ is an isomorphism,  and the full subcategory of $B$ with objects $b$ such that $\e_{b}$ is an isomorphism.   
 Part (b) follows from (a)  by applying this remark to the adjoint pair $(\mathfrak{Q},\mathfrak{P})$, and (c) follows from (b).  
   \end{proof}
\subsection{} \label{ssa11.6} Let $\CR=(G,\L,N)$ be a protorootoid. Denote the component   at $\CR$ of the counit of the adjunction in Theorem \ref{ssa11.3} as   $\e_{\CR}=(\a,\nu)\colon \CR'\to \CR$.
Write $\CR'=(G,\L',N')$ and let $\CL$, $\CL'$ denote the big weak orders of $\CR$ and $\CR'$ respectively.   
 \begin{cor*}\begin{num}
 \item  $\CR'$ is abridged.
 \item $\mathfrak{P}(\e_{\CR})\colon \mathfrak{P}(\CR')\to \mathfrak{P}(\CR)$ is an identity  map  in $\gpdpreord$, with underlying groupoid map  $\a=\Id_{G}$. 
 \item For any $a\in \ob(G)$, the weak right preorders of $\CR$ and $\CR'$ at $a$ are equal.
  \item For  $a\in \ob(G)$ $\nu_{a}\colon \lsub{a}{\L}'\to \lsub{a}{\L}$ restricts to an order isomorphism   $\lsub{a}{\CL}'\to \lsub{a}{\CL}$. 
  \item   If ${\CR}$ is abridged, then  for $a\in \ob(G)$, $\nu_{a}$ is a surjection of Boolean rings.  \end{num}
 \end{cor*}  
 \begin{proof}  
By construction,  for each $a\in \ob(G)$, $\lsub{a}{\L}'$ is generated as ring  by  the elements
$(a,g,h)$ for $g\in \lrsub{a}{G}{b}$ and $h\in \lrsub{b}{G}{}$.
Since $(a,g,h)=(\L'(g))(N'(h))=N'(gh)+N'(g)$, (a) follows. Part (b)
 follows from  the proof of Corollary \ref{ssa11.5}(a) (taking $b=\mathfrak{P}(\CR)$) and 
   (c) follows from (b).  For (d), observe that the map
   $\lsub{a}{\CL}'\to \lsub{a}{\CL}$ given by restriction of $\nu_{a}$ identifies with the induced map (of  associated 
    partial orders) of the preorder isomorphism
   $\Id_{G}\colon \lsub{a}{G}\to \lsub{a}{G}$ from (c), where the left (resp., right) side has the weak right preorder from $\CR'$ (resp., $\CR$).

    Finally (e) follows from surjectivity of the map $\lsub{a}{\CL}'\to \lsub{a}{\CL}$ in (d),  because 
 $\lsub{a}{\CL}$  generates $\lsub{a}{\L}$ since $\CR$ is abridged  and  $\lsub{a}{\CL}'$ generates  $\lsub{a}{\L}'$  by (a).  \end{proof} 
   \subsection{} \label{ssa11.7} This subsection   shows  that the right hand part of the diagram in \cite[5.8]{DyGrp1} can be refined to the following diagram by inserting the category $\gpdpreord_{P}$:
        \begin{equation*}
       \xymatrix{
       {\prootc}\ar@<1ex>[r]^{\mathfrak{A}}
       \ar@/^3.0pc/[rr]_-{\mathfrak{P}_{P}}
       &{\prootc^{a}}\ar@<1ex>[r]^>>{\mathfrak{P}'_{P}}
       \ar@/^3.0pc/[rr]_-{\mathfrak{P}'} \ar@<1ex>[l]^{\mathfrak{B}}&
       {\gpdpreord_{P}}\ar@{^{(}->}[r]_{\iota}\ar@<1ex>[l]^<<{\mathfrak{Q}'_{P}} \ar@/^3.0pc/[ll]
	_-{\mathfrak{Q}_{P}}
       &
       {\gpdpreord.} \ar@/^3.0pc/[ll]
	_-{\mathfrak{Q}'}	   }    
       \end{equation*}
            In this, $\mathfrak{A}$, $\mathfrak{P}'$, $\mathfrak{Q}'$,
       $\mathfrak{B}$ have been previously defined 
       ( the functors $\mathfrak{P}=\mathfrak{P}'\mathfrak{A}$ and $\mathfrak{Q}=\mathfrak{B}\mathfrak{Q}'$ from \cite[5.8]{DyGrp1}  are  not shown here).
 Let $\iota\colon \gpdpreord_{P}\to \gpdpreord$ denote the inclusion functor. Note that $\mathfrak{P}'$ uniquely factorizes as 
 $\mathfrak{P}'=\iota \mathfrak{P}_{P}'$ for some functor 
 $\mathfrak{P}_{P}'$ as shown,  by definition of $\gpdpreord_{P}$.
 Set $\mathfrak{Q}'_{P}:=\mathfrak{Q}'\iota$. A straightforward argument using Theorem \ref{ssa11.3} and  the fact that $\iota$ is full and faithful shows that $\mathfrak{Q}'_{P}$ is left adjoint to $\mathfrak{P}_{P}'$. Since $\mathfrak{B}$ is left adjoint to $\mathfrak{A}$, it follows that $\mathfrak{Q}_{P}:=\mathfrak{B}\mathfrak{Q}_{P}'$ is left adjoint to $\mathfrak{P}_{P}:=\mathfrak{P}_{P}'\ab$.
 Therefore each explicitly named functor except $\iota$ in the diagram is part of an adjoint pair, with the lower functor left adjoint to the symmetrically corresponding upper functor.
 
   \begin{rem*}
      Note that if $\CR$ is  an object of $\prootc_{F}$,  then $\CR$ is  abridged by Corollary \ref{ssa11.6}(a). 
         By Corollary \ref{ssa11.5}, the functor $\mathfrak{Q}_{P}'$ (resp., $\mathfrak{Q}_{P}$)   defines an equivalence of $\gpdpreord_{P}$ with the full coreflective subcategory  $\prootc_{F}$ of $\prootc^{a}$ (resp., of $\prootc$).      It will be  shown in
  a  later paper that   abridged principal rootoids are objects of $\prootc_{F}$. It  is an open problem to  determine   conditions  for an  abridged rootoid or protorootoid to be in $\prootc_{F}$. \end{rem*}

     \section{Squares}\label{s12}
\subsection{}\label{ss12.1}
Let  $(W,S)$ be a   Coxeter system. Suppose that $x\in W$ and  $r,s\in S$ are  such that  $l(sx)> l(x)$ and $l(sxr)\leq l(xr)$.   Then EC  implies that $sxr=x$ so, setting $v:= rx^{*}s=x^{*}$, one has  $vsxr=1$, and  each of the expressions   $xr$, $sx$,  $vs$ and $rv$ is a compatible expression:
\begin{equation*}
\xymatrix{
{}\ar[r]^{x}&{}\ar[d]^{s}\\
{}\ar[u]^{r}&{}\ar[l]^{v}
}
\end{equation*}
\subsection{} \label{ss12.2} The following  simple  notion, of which the previous subsection provides natural examples, plays a fundamental   role in our main results.
\begin{defn*} Let $\CR=(G,L,N)$ be a  protorootoid.
\begin{num}\item A \emph{oriented square} of $\CR$ is a quadruple of morphisms
$(v,y,x,w)$ of $G$ as in the diagram 
\begin{equation*}\xymatrix{
{}\ar[r]^{x}&{}\ar[d]^{y}\\
{}\ar[u]^{w}&{}\ar[l]^{v}
}
\end{equation*}  such that the composite $vyxw$ is defined and equal to an identity morphism of $G$
and  each expression
$vy$, $yx$, $xw$ and $wv$ is  compatible  for $\CR$.

\item A commutative square diagram in $G$ 
\begin{equation*}\xymatrix{
{}\ar[r]^{x}&{}\\
{}\ar[u]^{w}\ar[r]_{u}&{}\ar[u]_{z}
}
\end{equation*} is said to be a \emph{commutative square of $\CR$}
if $(x,w,u^{*},z^{*})$ is an oriented square of $\CR$. 

\end{num}

\end{defn*}
If $(v,y,x,w)$ is an oriented square of $\CR$, then  $(y,x,w,v)$ and $(w^{*},x^{*},y^{*},v^{*})$
are also oriented squares of $\CR$, so there is a natural action of the dihedral 
group of order $8$ on the set of oriented squares of $\CR$. This implies that  
the notion of commutative square of $\CR$ is well-defined, since in (b), 
  $(x,w,u^{*},z^{*})$ is an oriented  square of $\CR$ if and only if  $(z,u,w^{*},x^{*})$ is 
  an oriented square of $\CR$. 
  
  Sometimes   the phrase ``square of $\CR$''  (or even just ``square'') will be used  refer to either an oriented square of $ \CR$   or a commutative square of $\CR$ when context makes it clear which is meant.  Note that the sets of oriented and commutative squares of $\CR$ depend only on the underlying groupoid-preorder of $\CR$.
\subsection{}\label{ss12.3} The remainder of this section gives basic properties of squares,  and  examples.  
\begin{lem*} Let $\CR=(G,\L,N)$ be a protorootoid. Fix a 
 quadruple $s=(x,w,v,y)$ of morphisms of $G$ with $xwvy$ defined and equal to an identity morphism of $G$. Then:
\begin{num}\item  $s$ is  an oriented  square of $\CR$  if and only if 
both $x(N_{w})=N_{y^{*}}$ and $N(w)\cap N(x^{*})=\eset$.
\item If $s$ is an oriented square of $\CR$, it is uniquely determined by $(x,w,v)$.
\item If   $\CR$ is faithful, and $s$ is an oriented square, then $s$ is uniquely determined by $(x,w)$ (this   property will be referred to as \emph{rigidity of squares}).
\item Suppose that $\CR$ is the protorootoid attached to a signed groupoid-set $(G,\Phi)$. Then $s$ is an oriented square of $\CR$  if and only if  
$x(\Phi_{w})=\Phi_{y^{*}}$. \end{num}\end{lem*}
\begin{rem*} (1)
 Many results involving squares have an analogue 
involving a   dual  notion  of cosquares, in which, for instance,  the 
condition $x(\Phi_{w})=\Phi_{y}$ in (c) is replaced by
 $x(\Phi^{\cp}_{w})=\Phi^{\cp}_{y}$ where
 for $z\in \lsub{a}{G}$, $\Phi_{z}^{\cp}:=\lrsub{a}{\Phi}{+} \sm \Phi_{z}$ (see Remark \ref{ssa10.3}(3), and \cite{DyWeak} for some examples in Coxeter groups).   

 (2) Much of the combinatorics   involving  squares and cosquares in subsequent papers  applies in  the framework of pairs
    $(\CR,\CL)$ where $\CR=(G,\L,N)$ is a protorootoid and 
    $\CL=(\,\lsub{a}{\CL})_{a\in \ob(G)}$,  
     $\lsub{a}{\CL}\seq {\L}(a)$ is a  subset  containing $0_{\L(a)}$ of $\L(a)$,   endowed with the 
     induced partial order, and the family ${\CL}$ is  stable under the    dot $G$-action in the natural sense.  \end{rem*}
\begin{proof} For (a), first suppose that $s$ is a square of $\CR$.
From $N(y)\cap N(v^{*})=\eset$, it follows that 
\begin{equation*} N(xw)=N_{x}+x(N_{w})=N(y^{*}v^{*})=N_{y^{*}}\cup y^{*}(N_{v^{*}}).\end{equation*}
Using $N_{x}\cap N_{y^{*}}=\eset$, this implies that $N_{y^{*}}
\seq x(N_{w})$.  By symmetry, $N_{w}\seq x^{*}(N_{y^{*}})$.
One has $N_{w}\cap N_{x^{*}}=\eset$ by definition, and  the ``only if'' 
implication in (a) follows. To prove  the ``if'' in (a), it will suffice by symmetry 
to show that if $x(N_{w})=N_{y^{*}}$ and $N(w)\cap N(x^{*})=\eset$, then 
$y(N_{x})=N_{v^{*}}$ and $N_{x}\cap N_{y^{*}}=\eset$ (for then two more 
repetitions of the same argument shows that the defining conditions for $s$ to be 
an oriented square are satisfied). Now the assumptions imply that
\begin{equation*} N(xw)=N_{x}\dotcup x(N_{w})=N_{x}\dotcup N_{y^{*}}.\end{equation*}
But also, $N(xw)=N(y^{*}v^{*})=N_{y^{*}}+y^{*}(N_{v^{*}})$. 
Comparing the last two equations shows   that $N_{x}=y^{*}(N_{v^{*}})$ as 
required. This completes the proof of (a). Part  (b) holds 
since $xwvy=1_{a}$ where $a:=\cod(x)$.  If also  $\CR$ is faithful, then $y$ 
is uniquely determined by the condition  $N_{y^{*}}=x(N_{w})$ from (a).  
Then  $v$  is determined by  (b),  since $(y,x,w,v)$ is 
a square of $\CR$. This proves (c).

For (d), write $\mathfrak{L}(G,\Phi)=(G,\Xi,N)$ where $\Xi\colon G\to \setc$
with $\CP_{G}(\Xi)=\L$. From the constructions in the proof of \cite[Proposition  5.5]{DyGrp1}, it follows that
$x(N_{w})=N_{y^{*}}$ if and only if  $x(\Phi_{w})$ is a set of $\set{\pm}$-orbit 
representatives on $\Phi_{y^{*}}\cup -(\Phi_{y^{*}})$. Also, $N_{w}\cap N_{x^
{*}}=\eset $ if and only if  $\Phi_{w}\cap \Phi_{x^{*}}=\eset$ if and only if   $x(\Phi_{w})\seq \lrsub{a}{\Phi}{+}
$ where $a:=\cod(x)$.
Hence  $x(\Phi_
{w})=\Phi_{y^{*}}$ if and only if  ($x(N_{w})=N_{y^{*}}$ and $N_{w}\cap N_{x^{*}}=\eset$). This shows that  (d) follows from (a).
\end{proof}
 \subsection{} \label{ss12.4} Consider a commutative diagram in $G$ as follows:
\begin{equation}\label{eq12.4.1}\xymatrix{&\ar[l]_{a}&\ar[l]_{e}\\
\ar[u]^{d}&\ar[u]^{b}\ar[l]^{c}&\ar[l]^{g}\ar[u]_{f}
}\end{equation}
Define the quadruples   $s_{1}:=(a,b,c^{*},d^{*})$, $s_{2}:=(e,f,g^{*},b^{*})$
and $s_{3}:=(ae,f,g^{*}c^{*},d^{*})$ of morphisms of $G$.
\begin{lem*} If any two of  $s_{1},s_{2},s_{3}$ are oriented squares  of $\CR$, then they are all  
oriented squares  of $\CR$.\end{lem*}
\begin{rem*} If $\CR$ is preorder-isomorphic  to the rootoid attached to a signed groupoid-set, this  is obvious from  Lemma \ref{ss12.3}(d).  Although the general case can be deduced from this special case,      a direct  proof is indicated  for completeness. \end{rem*}
\begin{proof} 
 Suppose for example that $s_{1}$ and $s_{3}$ are squares of $\CR$. Then using
Lemma \ref{ss12.3}(a) it follows  that 
\begin{equation*} a(e( N_{f}))=(ae)( N_{f})= N_{d}=a( N_{b})\end{equation*} and hence $e( N_{f})= N_{b}$.
Also,  $N(e)\cap N(b)=\eset$ since \begin{multline*}a(N(e)\cap N(b))=aN(e)\cap aN(b)=(N(ae)+N(a))\cap N(d)=\\ (N(ae)\cap N(d))+(N(a)\cap N(d))=\eset+\eset=\eset.\end{multline*}
This shows that $s_{2}$ is a square of $\CR$ by Lemma \ref{ss12.3}(a) again. The  other cases can be proved  similarly, or  deduced from the case treated using  the symmetry properties of oriented squares. 
\end{proof}

\subsection{Examples of squares}
\label{ss12.5} This subsection gives    examples  of squares   from Coxeter groups. \begin{exmp*} Let
 $(W,S)$ be a Coxeter system, $\CR:=\CC_{(W,S)}$ and $\Phi$ be the abstract root system of $\Phi$. 
 
  (1) Assume that 
     $(W,S)$ is  of type $A_{3}$, with
  $S=\set{r,s,t}$ and Coxeter graph $\xymatrix{r\ar@{-}[r]&s\ar@{-}[r]&t}$. 
  Concretely, $W=S_{4}$, the symmetric group on four letters, with
  $r=(1,2)$, $s=(2,3)$ and $t=(3,4)$ (adjacent transpositions).
 In the commutative  cubical diagrams below, each  face gives  a commutative square of  $\CR$. 
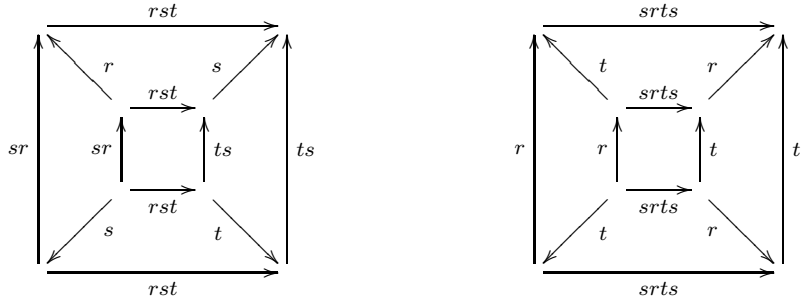
   \begin{equation*}\xymatrix{\ar[rrr]^{rst}&&&\\
   &\ar[r]^{rst}\ar[ul]_>>>>>>>>{r}&\ar[ur]^>>>>>>>>{s}&&\\ &\ar[u]^{sr}\ar[r]_{rst}\ar[dl]^>>>>>>>>{s}&\ar[u]_{ts}\ar[dr]_>>>>>>>>{t}&&\\
\ar[uuu]^{sr}\ar[rrr]_{rst}&&&\ar[uuu]_{ts}&}\qquad\qquad
\xymatrix{\ar[rrr]^{srts}&&&\\
   &\ar[r]^{srts}\ar[ul]_>>>>>>>>{t}&\ar[ur]^>>>>>>>>{r}&&\\ &\ar[u]^{r}\ar[r]_{srts}\ar[dl]^>>>>>>>>{t}&\ar[u]_{t}\ar[dr]_>>>>>>>>{r}&&\\
\ar[uuu]^{r}\ar[rrr]_{srts}&&&\ar[uuu]_{t}&} 
\end{equation*} 
Applying the previous Lemma to the above diagrams also gives some other squares of $\CR$, such as the oriented square $(sr,tsrt,st, srst)$.

 (2) Assume that $W_{J}$ is a finite standard parabolic subgroup of 
 $(W,S)$. Denote the  longest element of $(W_{J},J)$ as $w_{J}$.
 It is well known that $w_{J}^{2}=1$ and that for any $z\in W_{J}$, the expression $z z'$ is compatible where $z':=z^{*}w_{J}$. Hence if  $x\in W_{J}$, then  $(x, x^{*}w_{J},w_{J}xw_{J},w_{J}x^{*})$ is an oriented square of $\CR$. 

(3)  Suppose $J,K\seq S$ and $d\in W$  with $dJd^{-1}=K$ and 
$l(dr)=l(d)+1$ for all $r\in J$ (equivalently,  $l(sd)=l(d)+1$ for all $s\in K$).
Let $x\in W_{J}$ and $y=dxd^{*}\in W_{K}$. It is well known and easily seen (see \cite{Bour} or \cite{Hum}) that $d(\Phi_{x})=\Phi_{y}$. It follows that 
  $(d^{*},y^{*},d,x)$ is an oriented  square of  $\CR$.
  \end{exmp*}

 \section{Examples of morphisms of rootoids} 
 \label{s13} This section gives examples of morphisms of rootoids. Familiarity with the definitions of the categories $\rtd$ and $\rlec$ (and in particular, with JOP, AOP and   the notation   $\th^{\perp}$ for the partially defined adjoint of $\th$ as in 
 \cite[4.6--4.8]{DyGrp1}) is assumed.
\subsection{} \label{ss13.1} 

Let $\CR=(G,\L,N)$ be a protorootoid, and $i\colon H\to G$ be the inclusion morphism for  a subgroupoid $H$ of $G$. Recall the notion of the  restriction $\CR_{H}$ of $\CR$ to $H$ (see \cite[2.11--2.12]{DyGrp1}). Trivial 
examples show that $\CR_{H}$ need not be a rootoid if $\CR$ is a 
rootoid (and so, a fortiori,  pullbacks of rootoids by arbitrary groupoid homomorphisms need not be rootoids either).  This subsection  discusses  properties of the natural morphism  $i^{\flat }\colon \CR_{H}\to \CR$ in two  very   simple examples.

   \begin{exmp*} 
 (1)   Suppose $(W,S)$ is a dihedral Coxeter system of order $2n$  where $n\in 
   \Nat_{\geq 2}\cup\set{\infty}$.  Let $\CR:=\CC_{(W,S)}=(W,\L,N)$. 
  Write  $S=\set{r,s}$.    Let $x:=rs$, let  $G=\mpair {x}$  be  the 
  rotation subgroup of $W$ and let $i\colon G\to W$ be the inclusion. Note that $G$ is a cyclic group  of order $n$.   Set $X:=\set{x, x^{*}}\seq G$.
   Then $(G,X)$ is a $C_{2}$-system discussed in Example 
   \ref{ssa8.16}(2), and there is an 
   associated rootoid $\CT:=\CJ_{(G,X)}$ defined there. There is also the protorootoid
   $\CR_{G}:=i^{\nat}(\CR)$ obtained by restriction of $\CR$ to $G$.   One can  easily check that  
    $\CR_{G}\cong \CT$,  and that 
      $i^{\flat}$ is a morphism  in $\rtd$. It is a morphism in $\rlec$  if and only if 
    $n$ is even or infinite, since only then is $G$  join-closed in $W$. 
    If this holds, then $\CR_{G}$ is a  preprincipal rootoid (and otherwise, it is not).

    (2) Suppose $(W,S)$ is a Coxeter system and  $J\seq S$. Write 
    $\CR:=\CC_{(W,S)}=(W,\L,N)$ and let $i\colon W_{J}\to W$ be the 
    inclusion. Recall that $(W_{J},J)$ is a Coxeter system.
    Then $i^{\nat}(\CR)$ is a principal rootoid, the abridgement of which
    is easily seen to be  isomorphic to the abridgement of $\CC_{(W_{J},J)}$.  The unique morphism $\th$ of 
    weak  right orders induced by $i^{\flat}$ identifies with $i$ (as a map of sets)
    and induces an order  isomorphism between  $W_{J}$ (in weak  right order) and its image $W_{J}$ in $W$, which is a join closed order ideal
    of $W$ in its weak right order.
    The adjoint  $\th^{\perp}$  is defined only on $W_{J}$  and reduces to the identity map $W_{J}\to W_{J}$.     
    It follows  that  $i^{\flat}$ is a morphism in  $\rtd$ and in $\rlec$. 
        \end{exmp*} 
     \subsection{Fixed subgroups of automorphism groups of Coxeter groups}\label{ss13.2} Let $(W,S)$  be a Coxeter system and  $G$ be a group of 
      automorphisms of $(W,S)$. It is a well-known Theorem of Tits that
      the  fixed subgroup
       \begin{equation}\label{eq13.2.1} W':=W^{G}=\mset{w\in W\mid \s(w)=w \text{ \rm for 
      all $\s\in G$}}\end{equation}  of $G$ on $W$ is a Coxeter group. Its set $S'$ of standard  Coxeter 
      generators is the set of   longest elements $w_{J}$ of finite  standard
      parabolic subgroups $W_{J}$ of $W$ such that $J$ is a single $G$-orbit 
      on $S$.  It is known how to associate  a root system of $(W',S')$ to one of $(W,S)$, either as a linearly realized root system (see e.g. \cite{Hee}) or abstractly (\cite{DyEmb1}--\cite{DyEmb2}).
      
        \begin{prop*} Let  $\CR:=\CC_{(W,S)}$  denote the standard rootoid attached to $(W,S)$,  
      let    $i\colon W'\to W$ be the inclusion homomorphism
      and  set $\CT:=\CR_{W'}=i^{\nat} \CR$.
      Then  $\CT$ is a preprincipal  rootoid with simple generators $S'$,  and $ i^{\flat}\colon \CT\to \CR$ is a morphism in $\rlec$       
        \end{prop*}
              \begin{proof} Let $\leq$ denote the weak right order of $\CR$  and $\leq'$ denote that of  $W'$ from $\CT$.   Let $I$ denote the order ideal of $W$ generated by $W'$. Meets and joins in $(W,\leq)$ are denoted 
 by $\vee$ and $\wedge$ respectively.  
 
 By \cite[Lemma 2.10]{DyGrp1}, $ i\colon W'\to W$ is order preserving.  Moreover, since $i$ is injective, the definitions imply that  $W'$ has the induced order as a subposet of $W$. Note  that $G$ acts as a group of order automorphisms of $(W,\leq)$. Hence any join  or meet which exists in $W$ of a  subset of $W'$ is itself in $W'$ 
 and so $W'$ is a join-closed meet subsemilattice of $W$.
 This already implies that $\CT$ is a rootoid (the JOP for $\CT$ is inherited trivially from that for $\CR$), that $i^{\flat}$ is a morphism in $\rootc$, 
and   that     $i^{\flat}$ is  a morphism in $\rlec$ provided that it is one in $\rtd$ i.e.  if  $i^{\flat}$  satisfies the AOP.
 Recall that for $w\in I$,
 \begin{equation}\label{eq13.2.2} i^{\perp}(w):=\meet   \mset{u\in W'\mid w\leq u}\end{equation} is the minimum of the elements $u$  of $W'$ (in either $\leq '$ or $\leq$) such that $w\leq i(u)$.
There is  an   expression for  $i^{\perp}(w)$ as a join instead of a meet as follows:
  \begin{equation}\label{eq13.2.3}  i^{\perp}(w)=\join_{\s\in G}\s(w).\end{equation} To prove this, 
  note first  that since $w\leq i^{\perp}(w)$, it follows that
 $\s(w)\leq \s(i^{\perp}(w))=i^{\perp}(w)$ for all $\s\in G$. Therefore 
$u:=\join_{\s}\, \s(w)\leq i^{\perp}(w)$ as well. 
 Clearly, one has $w\leq u$. Since 
 $\mset{\s(w)\mid \s\in G}$ is $G$-stable, its join $u$ is in $W'$
 and so $i^{\perp}(w)\leq u$.   This proves \eqref{eq13.2.3}.

   Now   AOP requires  that for $w'\in W'$ and $w\in I$ such that  $N(i(w'))\cap N(w)=0$, 
 one  has $N(w')\cap N(i^{\perp}(w))=\eset$.
But for all $\s\in G$, $N(w')\cap N(w)=\eset$ implies that   
  $N(w')\cap N(\s(w))=\eset$  and hence 
  \begin{equation}\label{eq13.2.4} N(w')\cap N(i^{\perp}(w))=N(w')\cap N(\join_{\s\in G}\s(w))=\eset\end{equation}
  by the JOP for $\CR$. Hence $i^{\flat}$ is a morphism in $\rlec$.
  
  Next, note  that $W'$ is interval finite since intervals in $(W',\leq ')$ are subsets of intervals in $(W,\leq)$  which are finite.  By \cite[Lemma 3.4]{DyGrp1},  the set $A'$  of atoms of $W'$ in its weak order generates $W'$.  We claim that
  \begin{equation}\label{eq13.2.5} {A'}=\mset{i^{\perp} (s)\mid s\in I\cap S}.\end{equation}
  To see this, suppose first that $s'\in A'$. Then there is some $s\in  S$ with $s\leq s'$, since $W$ is interval finite with atoms $S$. Note that $s\in I$.
  Then $1_{W}<i^{\perp}(s)\leq s'$ and hence $1_{W'}<'i^{\perp}(s)\leq' s'$.
  This implies that $s'=i^{\perp}(s)$ since $s'$ is an atom of $W'$.
  
  Conversely, suppose that $r\in S\cap I$ and let $r':=i^{\perp}(r)$.
  There is an $s'\in{A'}$ with $s'\leq ' r'$, since $W'$ is interval 
  finite. From the argument above, $s'=i^{\perp}(s)$ for some $s\in I\cap S$.
 Since $W$ is preprincipal (in fact, principal), either  $N({r})\cap N({s'})=\eset$  or
 $N(r)\seq N(s')$.  In the first case, AOP  implies that $N(r') \cap N(s')=
 \eset$ contrary to  $N(s')\cap N(r')=N(s')$ (from $s'\leq 'r'$). In the second 
 case, 
 $r'=i^{\perp}(r)\leq s'$ which implies $r'\leq ' s'$. This gives  $r'=s'\in A'$.
 
 To prove that $\CT$ is preprincipal, consider $w\in W'$ and $r'\in A'$.
 It is required to show that either $N(r')\seq  N(w)$ or $N(r')\cap N(w)=\eset$.
Write $r'=i^{\perp}(r)$ where $r\in I\cap S'$. Since $\CR$ is preprincipal, 
either $N(r)\seq N(w)$ or $N(r)\cap N(w)=\eset$. In the first  case, 
$N(r')\seq N(w)$ by definition of $r'=i^{\perp}(r)$.  In the contrary case,
  $N(r')\cap N(w)=\eset$ by AOP. Hence $\CT$ is preprincipal as claimed. From \eqref{eq13.2.3} and \eqref{eq13.2.5}, it follows that
  $A'=S'$.
 \end{proof}
 
    \begin{rem*} (1)  In \cite{Muhl}, it is shown that $\CT$ preprincipal  implies (in the circumstances above) that $(W',S')$ is a Coxeter system. It   can also  be shown from \cite{DyEmb2} that $\CT^{a}\cong (\CC_{(W',S')})^{a}$.
       
        (2) The proposition and its proof extend mutatis mutandis  to the context of a group of automorphisms of  any preprincipal rootoid
        (but Theorem \ref{ss14.1} is more general). \end{rem*}
 
  \subsection{}\label{ss13.3}  We now state without proof a generalization of Proposition \ref{ss13.2} (quite different from \ref{ss14.1}).    Let $\CR$ be a preprincipal rootoid. Denote $\CR$ as $(W,\L,N)$
 and its set of  atomic morphisms as  $S:=A_{\CR}$ for notational consistency with the examples below. Let $W'$ be a groupoid and $\th\colon W'\to W$ be a groupoid homomorphism such that for each $a\in \ob(W')$, the induced map $\lsub{a}{W}'  \to \lsub{\th(a)}{W}$ is injective (in examples below, $W$ is a Coxeter  group and  $\th$ is  the inclusion of a subgroup).

   Suppose that $R=R^{*}$ is a set of non-identity morphisms generating  $W'$.  Write $\CT:=\th^{\nat}(\CR)=(W',\Psi,M)$.  The assumptions on $\th$ imply that $\CT$ inherits faithfulness and interval finiteness  from $\CR$.
   Consider the following conditions:  \begin{conds}\item   For each  $a\in \ob(W')$, $w\in \lsub{a}{W}'  $ and $r\in R_{a}$, either the expression
  $\th(r)\th(w)$ is compatible in $\CR$ or the expression $\th(r^{*})\th(w')$ is compatible  in $\CR$ where $w':=rw$. 
  \item For any $a\in \ob(W')$ and  distinct $r,s\in \lsub{a}{R}$ for which $\th(r)\vee \th(s)$ exists in weak right order of  $\lsub{a}{W}$, one has $\th(r)\vee \th(s)=\th(u)$ for some $u\in \lsub{a}{W}'  $.
  \item  Let $a\in \ob(W')$ and  $s\in \lsub{\th(a)}{S}$ be such that there is an element $w\in \lsub{a}{W}'  $ with $N(s)\seq N(\th(w))$. Then  there is an element $\hat s\in \lsub{a}{R}$ with the following properties (1)--(2):  (1) if  $w\in \lsub{a}{W}'  $ with  $N(s)\seq N(\th(w))$,
 one has $N(s)\seq N( \th(\hat s))\seq N(\th(w))$  (2) if $w\in \lsub{a}{W}'  $ with $N(s)\cap N(\th(w))=\eset$, then $N(\th(w))\cap N(\th(\hat s))=\eset$. 
  \end{conds}
  
 For $a$, $s$ as in (iii), the element $\hat s$ is uniquely determined. Let $R'$ be the subset of $R$  consisting of all elements $\hat s$ arising as in (iii) from $a$ and  $s$ satisfying the conditions there. 
 
   \begin{thm*} If the  conditions $\text{\rm (ii)--(iii)}$ hold, then $\CT$ is a preprincipal rootoid, $R'$ is the set of atomic generators of $\CT$, 
   $\text{\rm (i)--(iii)}$ hold with $R$ replaced by $R'$,
  and $\th^{\flat} $ is a morphism in $\rlec$.   Further, if $R$ satisfies 
 $\text{\rm (i)--(iii)}$  or if   no proper subset $R_{0}=R_{0}^{*}$ of $R$ generates $W'$ and satisfies  $\text{\rm (ii)--(iii)}$, then $R'=R$. 
      \end{thm*}
      \begin{rem*} Theorems  \ref{ss13.5}  and \ref{ss14.2} provide extensive classes of (special) morphisms in $\rlec$. The underlying morphism $\th\colon W'\to W$  of groupoid-preorders of any morphism in $\rlec$ between preprincipal rootoids arises as Theorem \ref{ss13.3}, with $R=R'$ equal to the atomic generators of $W'$.
         The  version of the theorem proved in subsequent papers  includes also another characterization of such morphisms  with weaker hypotheses  which are more amenable  to verification in interesting cases (e.g. for Coxeter groups, using (generalizations of) the canonical automaton \cite[Section 4]{BjBr}).   \end{rem*} 
  \subsection{}\label{ss13.4}  
   In this subsection,  the theorem is  used     to describe additional  examples of   morphisms in $\rlec$ with  the rootoid $\CR=\CC_{(W,S)}$ of a Coxeter system $(W,S)$ as codomain. Fix $R=R^{*}\seq W\sm\set{1}$, let $W'$ be the subgroup generated by  $R$
 and let  $\th$ denote the inclusion $\th\colon W'\to W$.  
  \begin{exmp*}  
    (1)   Suppose above that $R$ consists of longest elements of non-trivial
  finite  standard parabolic subgroups of $(W,S)$, so the elements of $R$ are   involutions.  Conditions of various degrees of generality (and ease of verification)  under which (i) is satisfied can be found in \cite{Muhl}, 
  \cite{DyEmb1}, \cite{DyEmb2} and \cite{DyAdm}.  Assume (i) holds.
  Then  the sets of atoms in $W$ of weak right intervals $[1,r]$ for $r\in R$ are pairwise disjoint subsets of $S$.  
    It is easy to see  using this that  (i) then implies (ii) and (iii).
  Assume that (i)--(iii) hold, so the Proposition applies with $R=R'$.  It is    a  result of
   \cite{Muhl}  that $(W',R)$  is  a Coxeter system, and   \cite{DyEmb2}  implies that  $\CT^{a}\cong (\CC_{(W',R)})^{a}
    $; these  facts may be deduced alternatively from the conclusions of   Theorem \ref{ss13.3}.    The results in \ref{ss13.2} are a special case. 
        
(2) Suppose that $(W,S)$ is irreducible and  $S$ is finite.
Fix a  standard Coxeter element $c$ of $(W,S)$ i.e. an element of the form
$c=s_{1}\cdots s_{n}$ where $s_{1},\ldots, s_{n}$ are the distinct elements of
 $S$ in some  order.   Consider the following condition
 \begin{enumerate}\item[($*$)]
  there exists $n'$ such that
$1\leq n'\leq n$ and  $s_{1},\ldots,s_{n'}$  commute pairwise and  $s_{n'+1},\ldots, s_{n}$  commute pairwise. \end{enumerate}

Assume first that $W$ is finite and   ($*$) holds. Let $R_{0}:=\set{s_{1}\cdots s_{n'}, s_{n'+1}\cdots s_{n}}$ and $R_{1}:=\set{c,c^{-1}}$.
Set $W_{i}:=\mpair{R_{i}}$ and let $\th_{i}\colon W_{i}\to W$ be the inclusion.
Using  \cite{Bour}, (i)--(iii) are easily shown   to hold taking $R=R_{0}$, $W'=W_{0}$, and  $W'$ is a finite dihedral group (this  is   generalized in \cite{DyEmb2}). 
Assume  further that 
 $n$ divides 
 the length $N$ of the longest element of $(W,S)$.   Then \cite{Bour}  implies that 
 $l_{S}(c^{m})=ml_{S}(c)$ for $-N/n\leq m\leq N/n$. It is easily checked then 
 that (i)--(iii) also hold with $R=R_{1}$; in this case, $W'=W_{1}$ is a finite cyclic 
 group. The   morphism $(\th_{1})^{\flat}$   is the 
 composite  in $\rlec$  of  a morphism $i^{\flat}$ as in Example  \ref{ss13.1}(1) with the morphism $(\th_{0})^{\flat} $.  

Now consider the case that $W$ is infinite. It is proved in \cite{Spey} that  $l(c^{m})
 =ml(c)$ for all $m\in \Nat$. Taking  $R=\set{c,c^{*}}$, this implies that (i) 
 holds and that $W'$ is an infinite cyclic group. It is  not known
    if (ii)--(iii) hold 
 in general. However, they follow easily from $l(c^{m})
 =ml(c)$  if ($*$) holds.
 In that case, taking $R=R_{0}$ and $W'=W_{0}$ as defined   in the  case of finite $W$, (i)--(iii) also hold, and  $W'$ is infinite dihedral. Note however  that  not all infinite irreducible Coxeter systems have  Coxeter elements $c$ satisfying ($*$). 

(3)  Assume 
that $(W,S)$ is a universal Coxeter system i.e. the product
of  any two distinct elements of $S$ has infinite order, so $W$ is 
isomorphic to the  free product of the cyclic subgroups of order two generated 
by the elements of $S$. Assume also that there is a given free action of the  cyclic  group $C_{2}$ of order two on $S$. Write the action on $S$ of the non-trivial element of $C_{2}$ as $s\mapsto s'$ and let $S'$ be a set of $C_{2}$-orbit representatives.
Let $R=\mset{ss'\mid s\in S}$. Then $W'$ is a free group
(freely generated by  $\mset{ss'\mid s\in S'}$). It is easy to verify that conditions (i)--(iii) hold.
 \end{exmp*}   
  
   \subsection{Normalizer rootoids}\label{ss13.5}
 The remainder of this section  describes another class  of morphisms in $\rlec$  and,  in the case of Coxeter groups,  indicates  its connection to results of \cite{BrHowNorm}.

    Fix a (faithful, for simplicity)  protorootoid $\CR=(G,\L,N)$.  
  Define  a groupoid $ L$ as follows. The objects of $ L$ are pairs $A=(a,X)$ where $a\in \ob(G)$ and $X\seq \lsub{a}{G}$.   Given another object $B=(b,Y)$ of $ L$, a morphism $A\to B$ in $G$
  is defined to be   a triple $(B,g,A)$ where $g\in \lrsub{b}{G}{a}$ has the following property:
  \begin{enumerate}\item[($*$)] there is a bijection  $\s\colon X\to Y$ such that for each $x\in X$, there is a commutative square of $\CR$ of the form 
  \begin{equation*}\xymatrix{
{}\ar[r]^{g}&{}\\
{}\ar[u]^{x}\ar[r]&{}\ar[u]_{\s(x))}
}
\end{equation*}
(note $\s$  is uniquely determined, by  rigidity of squares).
 \end{enumerate}  
 The composition in $ L$ is induced by that in $G$, according to the obvious  formula
 $(C,h,B)(B,g,A):=(C,hg,A)$. It follows easily from Lemma \ref{ss12.4} that $ L$ is a groupoid as claimed.  
  There is a natural functor  $\th\colon  L\to G$  defined as follows.
 On objects, $\th(a,X)=a$ and on morphisms $\th(B,g,A)=g$.
 Define the protorootoid $\CN:=\th^{\nat}(\CR)$ and the morphism $\th^{\flat}\colon  \CN\to \CR$ in $\prootc$.

 \begin{thm*} Assume above that $\CR$ is a rootoid. Then
 \begin{num}
 \item  $\CN$ is a rootoid.
 \item If $\CR$ is complete (resp.,  interval finite, cocycle finite or  preprincipal), then $\CN$ has that same  property.
 \item $\th^{\flat}$ is a morphism in $\rlec$.
 \end{num} In particular, if $\CR$ is a principal rootoid, then $\CN^{a}$ is also a principal rootoid. \end{thm*}
 
 \begin{rem*} (1) Suppose that $\CR$ is the protorootoid of a (strongly faithful) signed groupoid-set $(G,\Phi)$. Let $A=(a,X)$,
$B=(b,Y)$ be objects of 
 $ L$. Then \begin{equation}\label{eq13.5.1} \lrsub{B}{ L}{A}=\mset{(B,g,A)\mid g\in \lrsub{b}{G}{a},  \mset{g(\Phi_{x})\mid  x\in X} =\mset{\Phi_{y}\mid y\in Y}}.\end{equation} In particular, the
 vertex group $\lrsub{A}{ L}{A}$ is isomorphic to the normalizer (i.e. setwise stabilizer)  in  the vertex group $\lrsub{a}{G}{a}$ of $ \mset{\Phi_{h}\mid h\in X}$.    For this reason,   $L$ (resp., $\CN$) will be referred to as the \emph{normalizer groupoid} (resp., \emph{normalizer rootoid}) of 
$\CR$. An analogous result applies to similarly defined \emph{centralizer rootoids} which are coverings of  the normalizer  rootoids, and for which  the vertex groups are the centralizers (pointwise stabilizers) of  sets $\mset{\Phi_{h}\mid h\in X}$ for $X\seq \lsub{a}{G}$ and $a\in \ob(G)$. 

 (2) In addition to the above construction, there is  a dual version  involving  a rootoid $\CN'$ with underlying groupoid  
 groupoid $ L'$ where  $\ob( L')=\ob( L)$ and with  composition given by the same formula as for $L$, but in which the analogue  of \eqref{eq13.5.1} 
 is
      \begin{equation}\label{eq13.5.2} \lrsub{B}{ {L'}\nts\nts}  {A}=\mset{(B,g,A)\mid g\in \lrsub{b}{G}{a},  \mset{g
      (\Phi_{x}^{\cp})\mid  x\in X} =\mset{\Phi_{y}^{\cp}\mid y\in Y}}\end{equation} 
      where 
       $\Phi_{z}^{\cp}:=\lrsub{c}{\Phi}{+} \sm \Phi_{z}$ for $z\in \lsub{c}{G}
       $.   
       These matters for Coxeter groups  are closely related to our initial motivations 
       in \cite{DyWeak} for this work and probably  more generally to the  questions on completions raised  in \ref{ssa10.3}. Similar remarks to those here apply, for instance, to Theorem \ref{ss14.2}.
           \end{rem*}      
       
        \subsection{} \label{ss13.6} This subsection  discusses in more detail applications of the preceding theorem  in the case in which $\CR=\CC_{(W,S)}$ is the standard (principal) rootoid of a Coxeter system $(W,S)$. The groupoid $ L$ may be described as in Remark \ref{ss13.5}, with $(G,\Phi)=(W,\Phi)$.
 The full subgroupoid  $K$ of $ L$  containing   all objects $(a,J)$ of $L$ 
 with $J\seq S$ (rather than $J\seq W$) is clearly  a union of components of $ L$.
 Identify subsets of $S$ with subsets of  the simple roots
   $\Pi=\dotcup_{s\in S} \,\Phi_{s}$ in the obvious way, by
    $J\mapsto \Pi_{J}:=\cup_{s\in J} \Phi_{s}$. Then $K$ may be described 
    (up to isomorphism) as follows.   The objects of $K$ are subsets 
    $\Gamma$ of $\Pi$. For $\G,\D\seq \Pi$, one has
    $\lrsub{\G}{K}{\D}=\mset{(\G,w,\D)\mid  w(\D)=\G}$ and the composition is
    given by $(\G,w,\D)(\D,w',\L)=(\G,ww',\L)$. The vertex groups of $K$
    are isomorphic to  the setwise stabilizers in $W$ of subsets of simple roots.
    
  The groupoid $K$ (or   its components) has been studied implicitly in  \cite{How}, \cite{Deod} and explicitly in \cite{BrHowNorm}.   Natural groupoid generators were found for $K$ if   $W$  is finite in  \cite{How}, and  in general   in \cite{Deod}.  A natural  presentation    for   $K$ with these groupoid generators was described  in \cite{BrHowNorm}.
  The generators and relations are both defined in terms of finite  standard parabolic subgroups of $W$.

    Since $\CR$ is preprincipal, Theorem \ref{ss13.5} implies that $\CN$ is  a 
    preprincipal  rootoid and hence  the restriction $\CN_{K}$ is also 
    preprincipal. For each object $a$ of $L$, one has the  map
     $\lsub{a}{\th}\colon \lsub{a}{L}\to W$ and its partially defined adjoint 
     $\lsub{a}{\th}^{\perp}$.  The atomic generators of $L$ (resp., $K$)
     are those elements $\lsub{a}{\th}^{\perp}(s)$ for $s\in S\cap \domn(\lsub{a}{\th}^{\perp})$  and $a\in \ob(L)$ (resp., $a\in \ob(K)$). The construction of   these simple generators for $L$ in general involves   repeatedly taking joins of certain recursively defined families of elements.  However, in the special case of the simple generators of $K$, it reduces to constructions with joins of elements of $S$. These are the longest elements of finite parabolic subgroups, and the     construction is seen to reduce  to that in \cite{Deod}. Thus,       
    the generators for $K$ in \cite{Deod} turn out to be  the atomic generators of $\CN_{K}$ (or equivalently, the simple generators  $S$ of the abridgement). Similarly, the presentation obtained in \cite{BrHowNorm}  is the braid presentation discussed in \ref{ssa8.8}.
     
    The parts of the general theory  described in this paper now  imply  additional  facts about $K$.  For example, 
    $K$ has a natural abstract root system which may be described in several ways (e.g. as that  of  $\ol{J}_{(K,S)}$), the number of positive roots made negative by a groupoid element is its length, the weak orders of 
    $K$ are complete meet semilattices which embed naturally in the 
    weak orders of $W$ satisfying AOP, and have the simple generators as their atoms,  the solution by  Tits of the  word problem for Coxeter groups extends to $K$, etc.
    Additionally, all these general consequences  apply not only to $K$, but also to the (in general much  larger) groupoid   $ L$; similar conclusions hold for any preprincipal rootoid  in place of that from a Coxeter system.
    
 \subsection{} \label{ss13.7}   The rootoids $ L$, $K$ also  have  additional
special     properties which can be established with the aid of the general theory, but 
    can't be proved entirely within  it.  Consider first the case in which $W$ is finite.
   Then (components of) $K$ turn out to be  Coxeter groupoids with  root systems realizable  in  real     vector spaces as in \cite{HY} and  \cite{CH}  (in a way similar to that of  Coxeter groups).
 Once one knows this,
     the     presentations of components of $K$ from \cite{BrHowNorm} can be alternatively written 
      as  Coxeter groupoid presentations as in \cite{HY} and  \cite{CH}  (in terms of  a family 
      of   Coxeter matrices attached to the objects   and an object 
      change diagram).   Note that components of $K$ (or $L$) need not simply be   coverings   of    Coxeter groups.   For example, 
      taking $(W,S)$ of type $D_{4}$,  the component $K'$ of 
      $K$ with objects the pairs $(1_{W},\set{s})$ for $s\in S$ is a (connected)  Coxeter groupoid with four objects  such that each of its four stars has thirty elements,  contains three  atomic generators  and has a  longest element of length  seven.
      
            It is not known for finite $W$  whether  all components of $L$ are Coxeter groupoids  or whether their abstract root systems have  realizations similar to those for $K$ in real vector spaces; in fact, it is not even known whether the self-composable    atomic  generators of $K$ are  involutions.

         If  $W$ is infinite, then components of $K$ (and a fortiori, $L$) need not be Coxeter groupoids (since simple examples  show that  the cardinalities of the sets of  simple generators with codomains in a fixed component of $K$       need not be constant). 
    However,   the natural geometric realizations in real vector spaces  of   the abstract  root system of $W$ induces  a  realization $\Psi$ in a real vector space of the abstract root systems of  $L$ (and therefore $K$)  in a more general, and weaker, sense;
    for example, simple generators need not act as pseudoreflections, and a component   $\lsub{a}{\Pi}$ for $a\in \ob(L)$ of  the simple roots may span a proper subspace of the span of the corresponding component  $\lsub{a}{\Phi}_{+}$ of the  positive roots. 
    \begin{rem*} (1) Preservation of realizability in this weak sense  of   root systems of principal rootoids under several natural constructions  is a general phenomenon, as will be shown in a subsequent paper. Realizability of root systems in the stronger sense (similar to that in \cite{CH}, \cite{HY})  is  only preserved in far more restricted circumstances, but has stronger consequences.

(2)      A final point  worth emphasizing here is that Theorem \ref{ss13.5} not only applies to produce, say, a preprincipal   rootoid $\th^{\nat} (\CR)$ from a preprincipal rootoid $\CR$, but  the theorem may be applied  again to $\th^{\nat} (\CR)$ instead of $\CR$ to obtain, for example, the  analogue for     $\th^{\nat} (\CR)$  of many results of \cite{BrHowNorm}.  \end{rem*}
  \section{Functor rootoids}
  \label{s14}
  \subsection{Limits of rootoids} \label{ss14.1}
  After describing in previous sections some background and  examples of rootoids and morphisms between them, we finally discuss the main results of these papers, which involve  the categories $\rtd$ and $\rlec$.  An easily stated first result is the following.
   \begin{thm*}   
  \begin{num}
  \item  The category $\rtd$ has  all small limits, as does its full subcategory of complete rootoids.
  \item The full  subcategories of $\rtd$ with     preprincipal rootoids as objects has all limits of functors from small  categories with only finitely many objects. In particular, it has all small limits, and coequalizers of arbitrary families of parallel morphisms
   \end{num}
  \end{thm*}
  \begin{rem*}  The categories in (b) do not have infinite  products. This is similar to the fact that an infinite product of Coxeter groups is not necessarily a Coxeter group.
     \end{rem*}
    \subsection{Functor rootoids} \label{ss14.2}
    The construction of  \emph{functor rootoids} discussed below significantly extends the construction of normalizer rootoids   from \ref{ss13.5}, puts it in a more conceptual framework and suggests  the ubiquity of rootoid local embeddings.
    
 Let $\CR=(G,\L,N)$ be a protorootoid and $H$ be a groupoid.
 For simplicity, assume that $\CR$ is faithful and connected and  $H$ is connected.
 Let $G^{H}$ denote the  functor groupoid with   objects the functors
   $F\colon H\to G$ and in which  a morphism $\nu\colon F\to F'$ in $G^{H}$ is a 
   natural transformation.  Composition is the usual composition of 
   natural transformations: if $\mu\colon F'\to F''$ is another morphism, then
 $\mu\nu$ has components $(\mu\nu)_{a}=\mu_{a}\nu_{a}$ for all $a\in \ob(G)$.
 
 Let $G^{H}_{\square}$ denote the subgroupoid of $G^{H}$ on all objects, but containing  only morphisms $\nu\colon F\to F'$ in $G^{H}$ such that for each morphisms $g\in \lrsub{a}{G}{b}$ of $G$, the commutative  diagram
  \begin{equation}\label{eq14.2.1}
 \xymatrix{
 {F(b)}\ar[r]^{\nu_{b}}&{F'(b)}\\
 {F(a)}\ar[r]_{\nu_{a}}\ar[u]^{F(g)}&{F'(a)}\ar[u]_{F'(g)}
 }
 \end{equation} in $G$   from the defining property of the   natural  transformation $\nu$ is a commutative square of $\CR$.  Using Lemma  \ref{ss12.4}, one easily checks that this defines a groupoid $G^{H}_{\square}$ as claimed.
 
 For any   $b\in \ob(H)$, there is an evaluation functor 
 $\e_{b}\colon G^{H}_{\square }\to G$  sending each functor  
 $F'\colon H\to G$  to $F'(b)$ and each morphism  $\nu\colon F'\to F''$ 
 in $G^{H}_{\square }$ to its component  $\nu_{b}\colon F'(b)\to F''(b)$ 
 as natural transformation. 
 
 Recall that for any groupoid $K$ and object $a$ of $K$, $K[a]$ denotes the component of $K$ containing $a$. Choose   base points  $F\in \ob(G^{H}_{\square})$ (fixed throughout the 
 following discussion) and $b\in \ob(H)$ (which will be allowed to vary).
 Let $\rho:=\rho_{b}\colon  G^{H}_{\square}[F]\to G$ denote the 
 restriction  of the previously defined evaluation functor.
  Define the protorootoid $\CT= \CT_{b}=\rho_{b}^{\nat}(\CR)$ and the 
  morphism $(\rho_{b})^{\flat}\colon  \CT_{b}\to \CR$ in $\prootc$ (these 
  depend only on $b$ and   $G^{H}_{\square}[F]$, 
  not  on $F$ itself).

 \begin{thm*} Assume above that $\CR$ is a rootoid. Then
 \begin{num}  
 \item  $\CT_{b}$ is a rootoid.
 \item If $\CR$ is complete (resp., interval finite, cocycle finite or preprincipal), then $\CT_{b}$ is complete (resp., interval finite, cocycle finite, or  preprincipal).
 \item $(\rho_{b})^{\flat}$ is a morphism in $\rlec$.
 \item The abridgement $(\CT_{b})^{a}$ is  independent of the choice of $b$, up to canonical isomorphism.
 \end{num}\end{thm*}
 If $\CR$ is a  principal rootoid, one  obtains a well-defined (up to isomorphism) principal rootoid $(\CT_{b})^{a}$ from any groupoid homomorphism from a connected groupoid to the underlying groupoid of  $\CR$.  The rootoids $\CT_{b}$ (resp., $(\CT_{b})^{a}$) are called \emph{functor rootoids} (resp., \emph{abridged functor rootoids}) of $\CR$.    
 \subsection{} \label{ss14.3} The proof of Theorem \ref{ss13.5} is  reduced to that of Theorem \ref{ss14.2}  by the  following   technical   lemma (the main point in  the proof of which is  rigidity of squares) and generalities (\cite[Lemma 4.9]{DyGrp1}) about covering protorootoids.  The statement uses the notation of \ref{ss13.5} and \ref{ss14.2}. 

\begin{lem*} Let $\CR$ be a faithful protorootoid and  $A=(a,X)$ be an object of the normalizer groupoid $L$ with $X\neq \eset$. Fix   a connected, simply connected groupoid
$H$ and  an   object $b$ of $H$ such that there is a bijection $\phi\colon \lsub{b}{H}\xrightarrow{\cong} X$. 
\begin{num}\item  There is a unique groupoid homomorphism $F\colon H\to G$  such that, if $H$ is non-empty, $F(b)=a$ and $F(h)=\phi(h)\in \lsub{a}{G}$ for all 
$h\in \lsub{b}{H}$.
\item  Let $\th'\colon L[A]\to G[a]$  and $\rho_{b}'\colon G^{H}_{\square}[F]\to G[a]$ denote the evident restrictions of $\th$ and $\rho_{b}$.
Then there exists a unique groupoid homomorphism $\tau\colon G^{H}_{\square}[F]\to L[A]$ such that $\tau(F)=A$ and  the diagram
\begin{equation*}
\xymatrix{
{L[A]}\ar[r]^{\th'}&{G[a]}\\
{G^{H}_{\square}[F]}\ar[u]^{\tau}\ar[ur]_{\rho_{b}'}&
}
\end{equation*} commutes. Moreover,
$\tau$ is a covering.
\item 
The protorootoid $\CT_{b}$  is a covering protorootoid
of  the restriction $\CN_{L[A]}$. Hence  every component of $\CN$ is a covering quotient  of a functor protorootoid of  $\CR$.
\end{num}
\end{lem*}
\begin{rem*} The component $G^{H}_{\square}[F]$ may be alternatively regarded as a component of a centralizer rootoid as in Remark \ref{ss13.5}(1).
The covering $\tau$ in the above lemma  comes from the  group of automorphisms of $H$ which fix the object $b$  and permute $\lsub{b}{H}$ arbitrarily.  Similarly,  suitable automorphism groups  of  an arbitrary connected groupoid $H$  give rise to  covering quotients  of  associated functor rootoids.
\end{rem*}

 \subsection{Duality}\label{ss14.4} Define a  \emph{based  connected groupoid} to be a pair $( H  ,e)$ consisting of a
connected groupoid $ H  $ and  a specified object $e$ of $ H  $. A morphism of based connected groupoids is defined to be a base-point preserving groupoid morphism. With composition of underlying groupoid morphisms as composition, this defines a category $\cgrpdc _{\bullet}$ of based connected groupoids.

The above construction may be interpreted 
as follows.  Fix a connected rootoid $\CR$ and a base point $a$ of its 
underlying groupoid $G$.  Consider the slice category
 $C=\cgrpdc_{\bullet}/ (G,a)$.  The construction attaches to any object $F\colon (H,b)\to (G,a)$ of $C$ another object $e(F):=(G^{H}_{\square}[F],\rho_{b})$ of $C$.  
    An object of $C$ isomorphic to an object $e(F)$ will be called a \emph{stable object} of $C$.
     
 \begin{thm*}  \begin{num} \item  The map $e$ from $\ob(C)$ to $\ob(C)$ defined above naturally  extends to a contravariant functor  $e\colon C\to C$
 (induced by the contravariant functor $\grpdc\to \grpdc$ given on objects by
 $H\mapsto G^{H}$). 
  \item      There are natural isomorphisms
   \begin{equation*} \Hom_{C}(P_{1},e(P_{2}))\cong \Hom_{C}(P_{2},e(P_{1})) \end{equation*}
          for all $P_{1},P_{2}\in \ob(C)$, where each $\Hom$-set has at most one element. More precisely,
    letting $e'\colon C\to C^{\op}$ and $e''\colon  C^{\op}\to C$ be the two  covariant functors naturally identified with $e$, the functor $e'$ is left adjoint to $e''$.
    \item Let $D$ be the full subcategory of $C$ consisting of stable objects. 
    For all $P\in \ob(D)$, the component at $P$ of the unit (and counit) of the adjunction in $\text{\rm (b)}$ is an isomorphism. 
   
    \item The restriction of $e$ to a  contravariant functor $D\to D$ is a  contravariant equivalence i.e $e'$ restricts to a  category equivalence  $D\to D^{\op}$.
    \item Let $F$ be in $\ob(C)$ and  $F'':=e(e(F))\in \ob(C)$. Regard $F,F''$  as morphisms in $\cgrpdc_{\bullet}$.
    Then $F''$ is a monomorphism and $F=F''F'$ for a (unique) morphism $F'$  in $\cgrpdc_{\bullet}$
    \end{num}
\end{thm*} 

The stable morphism  $F''$ in (e) is called  the \emph{hull} of $F$. 
Part (e) shows that any morphism $F$  in $C$ factors uniquely 
through its hull. Moreover, $e(F)$ and $e(F'')$ are canonically 
isomorphic, and  $(F'')^{\flat}$ is a morphism in $\rlec$ with 
$F''$ as underlying groupoid morphism.  By (d), every hull has a dual.  Loosely speaking,
this means that in considering  components 
$\CT_{b}$ of  functor rootoids,  one may for most purposes  restrict 
$F$ to be a stable local embedding. There is a category of stable local embeddings of rootoids not discussed here.
 \subsection{}\label{ss14.5} It is a well known fact that a connected groupoid is  determined  
up to isomorphism by the isomorphism type of  its vertex groups  and the cardinality of the set of its objects.  The proof of Theorems \ref{ss14.2}  and \ref{ss14.4} make use of  a  more canonical version of this fact, described  in the proposition below.
 
 Define a   category $\datc_{\bullet}$ with objects the 
triples $(A,X,x)$ where $A$ is a group, $X$ is a free (left) 
$A$-set and $x\in X$. A morphism $(A,X,x)\to (B,Y,y)$ in 
$\datc_{\bullet}$ is a pair $(\rho,\th)$ where $\rho\colon A\to B$
 is a group homomorphism and $\th\colon X\to Y$ is a function
satisfying $\th(x)=y$ and $\th(ax')=(\rho(a))\th(x')$ for $a\in A$,
 $x'\in X$. Composition is componentwise: $(\rho,\th)(\rho',
 \th'):=(\rho\rho',\th\th')$.

\begin{prop*} The natural functor  $\mathfrak{F}\colon \cgrpdc_{\bullet}\to  \datc_{\bullet}$ defined  on objects  by mapping 
$(G,a)\mapsto (\lrsub{a}{G}{a},\lsub{a}{G},\lrsub{}{1}{a})$
is  a category equivalence.\end{prop*}
 The objects of $\datc_{\bullet}$ are accordingly  called \emph{based connected groupoid datums}.

\subsection{}\label{ss14.6} Let $\CR$ be a connected rootoid with underlying groupoid $G$ and fix $a\in \ob(G)$. Let $C$ be as in \ref{ss14.4}.
Define the coproduct (disjoint union) $\lsub{a}{\CX}:=
\lrsub{a}{G}{a}\coprod \lsub{a}{G}$ of sets. Let $F\colon (H,b)\to (G,a)$ be a morphism in $C$ i.e.  $F\colon H\to G$ is a groupoid homomorphism  with $F(b)=a$.
Write  $\mathfrak{F}(F)=(\rho,\th)$  where $\rho$ (resp., $\th$) is the group homomorphism (resp., function)  $\lrsub{b}{H}{b}\to \lrsub{a}{G}{a}$ (resp., $\lsub{b}{H}\to \lsub{a}{G}$) induced by $F$.
Define 
$\chi(F):=\Img(\rho)\coprod \Img (\th)\seq\, \lsub{a}{\nts\CX}$.

To any relation $R\seq X\times Y$ from a set $X$ to
 a set $Y$, there is an associated Galois connection  
 determined by  a pair of order-reversing   maps between
  $\wp(X)$ and $\wp(Y)$ as follows. The map $\wp(X)\to \wp(Y)$ 
  is given by \begin{equation}\label{eq14.6.1} A\mapsto A^{\dag}:=\mset{y\in Y\mid (x,y)\in 
  R\text{ for all $x\in A$} }.\end{equation} and the map 
  $\wp(Y)\to \wp(X)$ is defined by symmetry.  We shall be concerned below  only with 
  the case that  $R$ is symmetric (so $X=Y$), in  which case 
  these two order-reversing maps coincide and are  denoted $A\mapsto A^{\dag}$.  Recall from \ref{ssa10.4} that a subset $A$ of $X$ is  called stable
  if it is of the form $A=B^{\dag}$ for some $B\seq X$, or equivalently, if $A=A^{\dag\dag}$. The stable subsets form a complete lattice and the map $A\mapsto A^{\dag}$ is an order-reversing involution of this lattice.

\begin{thm*}  There is a (explicitly known)  symmetric  relation $R$ on
 $X=\lsub{a}{\nts\CX}$  such that the associated Galois connection 
 has the following properties: \begin{num}
\item The $R$-stable subsets of $X$  are precisely the  sets 
$\chi(F)$ for stable objects  $F\in\ob(C)$.
\item  For any $F$ in $\ob(C)$, $\chi(e(F))=\chi(F)^{\dag}$.
\end{num}
\end{thm*}

The definition of  $R\cap (\lsub{a}{G}\times \lrsub{a}{G}{a})$ is 
generalizes in an obvious way that of the relation on a Coxeter group described in 
\cite{DyWeak}, and the other   intersections $R\cap (\lrsub{a}{G}{a}\times \lrsub{a}{G}{a})$, $R\cap (\lrsub{a}{G}{a}\times \lsub{a}{G})$ and
$R\cap (\lsub{a}{G}\times \lsub{a}{G})$   have similar 
descriptions. The precise details are omitted here but are important for the proofs of Theorems \ref{ss14.2} and \ref{ss14.4}, and provide a basis for  calculation with functor rootoids.
The main part of the proofs  involves the study of combinatorics of squares in relation to these Galois connections for varying $a\in \ob(G)$.  The overall strategy is similar to that in the proof of Proposition \ref{ss13.2}, but the key step of expressing certain meets as  joins is more intricate,  involving   repetition (finite for interval finite rootoids, transfinite in general) of certain ``zig-zag'' and ``loop'' constructions 
to produce squares. 
\begin{rem*} Significant parts of the  formalism mentioned above apply in the context of faithful protorootoids (giving functor protorootoids), and  much of that formalism in turn  extends to a  general categorical context (e.g. as provided by  \cite[Remark 5.5(3)]{DyGrp1}). It seems very likely  that
non-trivial parts of the theory of rootoids have a natural extension  in which   underlying groupoids are replaced by  suitable subcategories of groupoids. An interesting open question is  whether  the notion of rootoid has a natural extension (e.g. in the framework suggested by  \cite[Remark 5.5(3)]{DyGrp1}) allowing underlying  general categories (instead of  only groupoids  or subcategories of groupoids).  \end{rem*}
\subsection{$n$-cubes}\label{ss14.7}
The final subsections of this section contain only vague indications
of results which will be formulated and proved elsewhere. As  Example \ref{ss12.5}(1) suggests, the notion of squares of a rootoid $\CR=(G,\L,N)$ 
extends naturally to higher dimensions.  Informally, a 
 \emph{commutative  $n$-cube of $\CR$} is defined to be a  commutative
  $n$-cubical diagram in $G$ such that all of its $2$-dimensional faces are 
  commutative squares of $\CR$.  The $n$-cubes for all $n$ naturally constitute a  \emph{cubical set} as in algebraic topology. The $n$-cubes (of $\CR$) admit $n$ natural
  compositions; the $i$-th is defined by stacking  two $n$-cubes together 
   in the $i$-th direction along a common 
  $(n-1)$-face  in an obvious way. For example, in \ref{ss12.4}, the first (horizontal) 
   $1$-composite   of the two small commutative squares is the larger 
   (exterior) square; 
   the  second  composition of (composable) $2$-squares given by stacking 
   them
    vertically.  These structures make the families of $n$-cubes into a  
   \emph{cubical $\o$-groupoid}  $G^{\o}$ (without connections) in the sense of higher   category  theory.      
       \subsection{} \label{ss14.8}    It can be shown, using the theory of functor rootoids, that the groupoid with morphisms given by a fixed one of the $n$-composition of $n$-cubes $(n\geq 1$) in $G^{\o}$  can be given a natural structure of rootoid.      This gives rise to what we  call a \emph{symmetric cubical $\o$-rootoid}  $\CR^{\o}$ (with  $G^{\o}$ as underlying  cubical $\o$-groupoid) associated to the rootoid $\CR$;   each  groupoid defined by an   $n$-composition in $G^{\o}$  is the underlying  groupoid of a  rootoid, in a compatible way.  For $n=1$, the unique $1$-composition underlies the original rootoid $\CR$. The term ``symmetric''  refers to  a  natural action of the 
    hyperoctahedral group (Coxeter group  of type $B_{n}$) on the set of  $n$-cubes for each $n$, which is compatible  with the various rootoid structures.  
     
 \subsection{Non-triviality}\label{ss14.9} For  certain rootoids $\CR$ 
 (e.g. $\CJ_{(G,X\cup X^{*})}$ for a free  group $G$ on a set 
 $X$ or $\CJ_{\G}$ where $\G$ is a tree)  the theory discussed
  in this section largely reduces to trivialities.
 
 That this is not the case in general can be seen in several 
 ways.  Firstly, for Coxeter systems $(W,S)$,  the theory of  
 functor rootoids of $\CC_{(W,S)}$  extends the 
 results in 
   \cite{BrHowNorm}, from which one sees that   non-trivial  functor 
   rootoids certainly exist whenever $(W,S)$ has non-trivial braid 
   relations. 
    
The existence of non-trivial stable objects in \ref{ss14.4} can be seen 
from the following fact:  for
 $\CC_{(W,S)}$ with $W$ finite, there is a natural injection  from $W$ to the set of  $R$-stable subsets of $\lsub{a}{\CX}$ (where $a$ is the unique object of $W$) in Theorem \ref{ss14.6}.    (The  dual of this fact,  
 in a sense related to  Remark \ref{ss13.5}(2), actually holds for arbitrary $(W,S)$; for finite $W$, the fact is self-dual).  This injection is a bijection in the case of finite dihedral groups, but not in general; for 
 example,  for $W$ of  type $A_{3}$, with $\vert W\vert =24$, there are $26$ $R$-stable subsets.

  Finally, a crude   quantitative    measure of richness of the theory of 
  functor rootoids  associated to a  general rootoid $\CR$
  may be given as follows.   Say that   an $n$-cube of $\CR$ is  
  non-trivial if  it  has  no identity morphism assigned to 
  any of its edges. 
    Let  
  $N$ be  the largest integer  $n$ for which non-trivial $n$-cubes exist 
  (or $N=\infty$ 
  if there is no largest such $n$).  
   If 
    $\CR=\CC_{(W,S)}$, then $N$ turns out to be   the supremum of the 
    ranks  of  
    finite standard parabolic subgroups of $(W,S)$ (compare Example
    \ref{ss12.5}(1))  and it has a similar (but slightly more subtle)  
   description for principal rootoids in general.
     \begin{rem*} (1) The quantity $N$-defined above    is the sup of 
     the $n\in \Nat$ for which $\CR^{\o}$   has non-trivial $n$-morphisms in the 
     sense of higher category theory

 (2)   Many  structures similar to $\CR^{\o}$ could be 
 constructed by taking components of suitable families of functor 
 rootoids.              Although the resulting structures are 
 genuinely higher categorical in nature, their definition  involves     only two-dimensional structures (squares) satisfying conditions 
    determined by a  $1$-cocycle.
        It is not known  if there is an extension of these ideas incorporating, for instance, $n$-cocycles for $n>1$.
        \end{rem*}


\begin{thebibliography}{10}

\bibitem{BjBr}
Anders Bj{\"o}rner and Francesco Brenti.
\newblock {\em Combinatorics of {C}oxeter groups}, volume 231 of {\em Graduate
  Texts in Mathematics}.
\newblock Springer, New York, 2005.

\bibitem{BjEZ}
Anders Bj{\"o}rner, Paul~H. Edelman, and G{\"u}nter~M. Ziegler.
\newblock Hyperplane arrangements with a lattice of regions.
\newblock {\em Discrete Comput. Geom.}, 5(3):263--288, 1990.

\bibitem{DySd}
C{\'e}dric Bonnaf{\'e} and Matthew~J. Dyer.
\newblock Semidirect product decomposition of {C}oxeter groups.
\newblock {\em Comm. Algebra}, 38(4):1549--1574, 2010.

\bibitem{Bour}
N.~Bourbaki.
\newblock {\em \'{E}l\'ements de math\'ematique. {F}asc. {XXXIV}. {G}roupes et
  alg\`ebres de {L}ie. {C}hapitre {IV}: {G}roupes de {C}oxeter et syst\`emes de
  {T}its. {C}hapitre {V}: {G}roupes engendr\'es par des r\'eflexions.
  {C}hapitre {VI}: syst\`emes de racines}.
\newblock Actualit\'es Scientifiques et Industrielles, No. 1337. Hermann,
  Paris, 1968.

\bibitem{BrHowNorm}
Brigitte Brink and Robert~B. Howlett.
\newblock Normalizers of parabolic subgroups in {C}oxeter groups.
\newblock {\em Invent. Math.}, 136(2):323--351, 1999.

\bibitem{BF}
J.~Richard B{\"u}chi and William~E. Fenton.
\newblock Large convex sets in oriented matroids.
\newblock {\em J. Combin. Theory Ser. B}, 45(3):293--304, 1988.

\bibitem{CH}
M.~Cuntz and I.~Heckenberger.
\newblock Weyl groupoids with at most three objects.
\newblock {\em J. Pure Appl. Algebra}, 213(6):1112--1128, 2009.

\bibitem{Deod}
Vinay~V. Deodhar.
\newblock On the root system of a {C}oxeter group.
\newblock {\em Comm. Algebra}, 10(6):611--630, 1982.

\bibitem{Dies}
Reinhard Diestel.
\newblock The cycle space of an infinite graph.
\newblock {\em Combin. Probab. Comput.}, 14(1-2):59--79, 2005.

\bibitem{DyHS1}
M.~J. Dyer.
\newblock Hecke algebras and shellings of {B}ruhat intervals.
\newblock {\em Compositio Math.}, 89(1):91--115, 1993.

\bibitem{DyGrp1}
M.~J. Dyer.
\newblock Groupoids, root systems and weak order {I}.
\newblock {\em \tt arXiv:1110.3217 [math.GR]}, 2011.

\bibitem{DyWeak}
M.~J. Dyer.
\newblock On the weak order of {C}oxeter groups.
\newblock {\em {\tt arXiv:1108.5557 [math.GR]}}, 2011.

\bibitem{DyRef}
Matthew Dyer.
\newblock Reflection subgroups of {C}oxeter systems.
\newblock {\em J. Algebra}, 135(1):57--73, 1990.

\bibitem{DyEmb1}
Matthew~J. Dyer.
\newblock Embeddings of root systems. {I}. {R}oot systems over commutative
  rings.
\newblock {\em J. Algebra}, 321(11):3226--3248, 2009.

\bibitem{DyEmb2}
Matthew~J. Dyer.
\newblock Embeddings of root systems. {II}. {P}ermutation root systems.
\newblock {\em J. Algebra}, 321(3):953--981, 2009.

\bibitem{DyAdm}
Matthew~J. Dyer.
\newblock Elementary roots and admissible subsets of {C}oxeter groups.
\newblock {\em J. Group Theory}, 13(1):95--107, 2010.

\bibitem{EdJ}
Paul~H. Edelman and Robert~E. Jamison.
\newblock The theory of convex geometries.
\newblock {\em Geom. Dedicata}, 19(3):247--270, 1985.

\bibitem{Gal}
S.~Gal.
\newblock On normal subgroups of {C}oxeter groups generated by standard
  parabolic subgroups.
\newblock {\em Geom. Dedicata}, 115:65--78, 2005.

\bibitem{HV}
Istv{\'a}n Heckenberger and {V}olkmar Welker.
\newblock Geometric combinatorics of {W}eyl groupoids.
\newblock {\em {\tt arXiv:1003.3231 [math.QA]}}, 2010.

\bibitem{HY}
Istv{\'a}n Heckenberger and Hiroyuki Yamane.
\newblock A generalization of {C}oxeter groups, root systems, and {M}atsumoto's
  theorem.
\newblock {\em Math. Z.}, 259(2):255--276, 2008.

\bibitem{Hee}
Jean-Yves H{\'{e}}e.
\newblock Syst{\` {e}}mes de racines sur un anneau commutatif totalement
  ordonn{\' {e}}.
\newblock {\em Geom. Ded.}, 37:65--102, 1991.

\bibitem{How}
Robert~B. Howlett.
\newblock Normalizers of parabolic subgroups of reflection groups.
\newblock {\em J. London Math. Soc. (2)}, 21(1):62--80, 1980.

\bibitem{Hum}
James~E. Humphreys.
\newblock {\em Reflection groups and {C}oxeter groups}, volume~29 of {\em
  Cambridge Studies in Advanced Mathematics}.
\newblock Cambridge University Press, Cambridge, 1990.

\bibitem{Muhl}
B.~M{\"u}hlherr.
\newblock Coxeter groups in {C}oxeter groups.
\newblock In {\em Finite geometry and combinatorics (Deinze, 1992)}, volume 191
  of {\em London Math. Soc. Lecture Note Ser.}, pages 277--287. Cambridge Univ.
  Press, Cambridge, 1993.

\bibitem{Pilk}
Annette Pilkington.
\newblock On the weak order of orthogonal groups.
\newblock {\em Preprint}, 2011.

\bibitem{STrees}
Jean-Pierre Serre.
\newblock {\em Trees}.
\newblock Springer Monographs in Mathematics. Springer-Verlag, Berlin, 2003.
\newblock Translated from the French original by John Stillwell, Corrected 2nd
  printing of the 1980 English translation.

\bibitem{Spey}
David~E. Speyer.
\newblock Powers of {C}oxeter elements in infinite groups are reduced.
\newblock {\em Proc. Amer. Math. Soc.}, 137(4):1295--1302, 2009.

\end{thebibliography}
\end{document}